\documentclass[12pt,oneside]{amsart}

\usepackage{setspace}

\usepackage{amssymb}   

\usepackage{amsmath}

\usepackage{mathrsfs}

\usepackage{amsthm}
\usepackage{newlfont}
\usepackage{amscd}
\usepackage{mathtools}
\usepackage{bm}
\usepackage{tikz-cd}
\usepackage{graphicx}

\usepackage{hyperref}
\usepackage{framed}

\usepackage{lmodern}
\usepackage[T1]{fontenc}

\usepackage{enumitem}

\usepackage[all]{xy}

\usepackage{textcomp}

\usepackage{amsbsy}

\usepackage{tikz}
\usetikzlibrary{positioning, arrows.meta}

\theoremstyle{plain}
\newtheorem{thm}{Theorem}[section]

\newtheorem{thmdefn}[thm]{Theorem/Definition}
\newtheorem{propdefn}[thm]{Proposition/Definition}
\newtheorem{lem}[thm]{Lemma}
\newtheorem{prop}[thm]{Proposition}
\newtheorem{cor}[thm]{Corollary}
\newtheorem{conj}[thm]{Conjecture}

\newtheorem{defn}[thm]{Definition}

\newtheorem{eg}[thm]{Example}

\newtheorem{construction}[thm]{Construction}

\newtheorem{rem}[thm]{Remark}

\newtheorem{criterion}[thm]{Criterion} % Ensure a theorem style is set in your preamble

\numberwithin{equation}{section}

\usepackage{relsize}
\usepackage[bbgreekl]{mathbbol}
\usepackage{amsfonts}

\DeclareSymbolFontAlphabet{\mathbb}{AMSb}
\DeclareSymbolFontAlphabet{\mathbbl}{bbold}
\newcommand{\Prism}{{\mathlarger{\mathbbl{\Delta}}}}
\DeclareMathOperator{\Gal}{Gal}

\DeclareMathOperator{\Spec}{Spec}
\DeclareMathOperator{\Spa}{Spa}
\DeclareMathOperator{\Spf}{Spf}
\DeclareMathOperator{\Spd}{Spd}

\DeclareMathOperator{\Frac}{Frac}

\DeclareMathOperator{\Fil}{Fil}
\DeclareMathOperator{\Hom}{Hom}

\DeclareMathOperator{\End}{End}
\DeclareMathOperator{\spe}{sp}
\DeclareMathOperator{\Isom}{\operatorname{Isom}}
\DeclareMathOperator{\Lat}{\operatorname{Lat}}

\DeclareMathOperator{\AdmPair}{\operatorname{AdmPair}}

\DeclareMathOperator{\colim}{\mathop{\mathrm{colim}}\limits}

\newcommand{\Frob}{\mathrm{Frob}}
\newcommand{\rank}{\mathrm{rank}}
\newcommand{\Vect}{\mathrm{Vect}}
\newcommand{\Loc}{\mathrm{Loc}}
\newcommand{\perf}{\mathrm{perf}}

\newcommand{\an}{\mathrm{an}}

\newcommand{\cris}{\mathrm{cris}}
\newcommand{\CRIS}{\mathrm{CRIS}}
\newcommand{\st}{\mathrm{st}}
\newcommand{\pst}{\mathrm{pst}}
\newcommand{\dR}{\mathrm{dR}}
\newcommand{\et}{\mathrm{\acute{e}t}}
\newcommand{\proet}{\mathrm{pro\acute{e}t}}
\newcommand{\qproet}{\mathrm{qpro\acute{e}t}}
\newcommand{\Perfd}{\mathrm{Perfd}}

\newcommand{\N}{\mathbb{N}}
\newcommand{\F}{\mathbb{F}}
\newcommand{\Z}{\mathbb{Z}}
\newcommand{\Q}{\mathbb{Q}}
\newcommand{\C}{\mathbb{C}}

\newcommand{\Perf}{\mathrm{Perf}}

\renewcommand{\log}{\mathrm{log}}

\newcommand{\calE}{\mathcal{E}}
\newcommand{\calF}{\mathcal{F}}
\newcommand{\calG}{\mathcal{G}}
\newcommand{\calH}{\mathcal{H}}
\newcommand{\calI}{\mathcal{I}}

\newcommand{\calO}{\mathcal{O}}
\newcommand{\calP}{\mathcal{P}}

\newcommand{\calR}{\mathcal{R}}
\newcommand{\calS}{\mathcal{S}}
\newcommand{\calT}{\mathcal{T}}
\newcommand{\calU}{\mathcal{U}}

\newcommand{\calX}{\mathcal{X}}
\newcommand{\calY}{\mathcal{Y}}

\newcommand{\bA}{\mathbb{A}}
\newcommand{\bB}{\mathbb{B}}

\newcommand{\bL}{\mathbb{L}}
\newcommand{\bM}{\mathbb{M}}

\newcommand{\bP}{\mathbb{P}}
\newcommand{\bQ}{\mathbb{Q}}
\newcommand{\bR}{\mathbb{R}}

\newcommand{\bT}{\mathbb{T}}

\newcommand{\bV}{\mathbb{V}}

\newcommand{\Zp}{\mathbb{Z}_p}
\newcommand{\Qp}{\mathbb{Q}_p}

\newcommand{\NP}{\mathcal{N}\!\mathcal{P}}
\newcommand{\Sht}{\mathrm{Sht}}
\newcommand{\Bun}{\mathrm{Bun}}
\newcommand{\Shv}{\mathrm{Shv}}

\newcommand{\vv}{\mathrm{v}}
\newcommand{\Isoc}{\mathrm{Isoc}}

\newcommand{\Rep}{\mathrm{Rep}}
\newcommand{\rmint}{\mathrm{int}}
\newcommand{\bd}{\mathrm{bd}}
\newcommand{\Conv}{\operatorname{Conv}}
\newcommand{\op}{\mathrm{op}}

\newcommand{\BdRp}{\mathbf{B}_{\mathrm{dR}}^+}

\newcommand{\bDhat}{\hat{\mathbb{D}}}
\newcommand{\admmodif}{\mathrm{Modif}^{\mathrm{adm}}}

\newcommand{\RpMT}{\mathrm{R}p\mathrm{MT}}
\newcommand{\CNP}{\mathrm{C}\NP}

\newcommand{\bbL}{\mathbb L}
\newcommand{\bbT}{\mathbb T}

\providecommand{\OL}{\mathcal O_L}

\providecommand{\Cov}{\operatorname{Cov}}
\providecommand{\GCov}{\operatorname{GCov}}
\providecommand{\UCov}{\operatorname{UCov}}
\providecommand{\oc}{\mathrm{oc}}

\providecommand{\frakX}{\mathfrak X}
\providecommand{\frakY}{\mathfrak Y}
\providecommand{\frakU}{\mathfrak U}

\begin{document}

\pagenumbering{arabic}

\title[$p$-adic Hodge Theory of de Rham Local Systems, I]{$p$-adic Hodge Theory of de Rham Local Systems, I: Newton Polygon and Monodromy}

\author{Heng Du} 
\address{Yau Mathematical Sciences Center, Tsinghua University, Beijing 100084, China}
\email{hengdu@mail.tsinghua.edu.cn}

\begin{abstract}
We prove that the relative $p$-adic monodromy theorem holds over a dense open subset. Moreover, we establish the equivalence of the following two statements: the local constancy of the Newton polygon function associated with a de Rham local system around rank-$1$ points, and the relative $p$-adic monodromy theorem near rank-$1$ points. We demonstrate how to extend the relative $p$-adic monodromy conjecture from the neighborhood of rank-$1$ points to the entire interiors of Newton partitions. Finally, we verify a conjecture of Howe--Klevdal concerning the potential good reduction of admissible pairs arising from variant of $p$-adic Hodge structures.
\end{abstract}

%\subjclass[2020]{14F30, 14G22, 14G45}
\maketitle

\tableofcontents

\section{Introduction}
A guiding principle in algebraic geometry, originating with Grothendieck, is that the analytic or arithmetic singularities of objects of geometric origin should be encoded by a finite amount of data. Specifically, a local system or differential equation arising from geometry should not exhibit arbitrary wild behavior near a codimension-one boundary; rather, after a finite base change, its singularity should be entirely encoded by finite nilpotent data. In the arithmetic $\ell$-adic setting, this manifests as the local monodromy theorem, cf.\ \cite[Exp.~I, Prop.~1.1]{SGA7I}.

The quest for a $p$-adic analog of this principle culminated in the $p$-adic local monodromy theorem, proved independently by Andr\'e \cite{AndreHasseArf}, Mebkhout \cite{MebkhoutAnaloguepadic}, and Kedlaya \cite{Kedlayapadiclocalmonodromy}. Subsequently, by demonstrating that the ``de Rham'' condition for a $p$-adic representation translates to the regularity of its associated differential equation over Robba rings, Berger used the $p$-adic local monodromy theorem to prove Fontaine's conjecture that de Rham representations are potentially log-crystalline, cf.\ \cite{Bergerpadicmonodromy}.

With the development of relative $p$-adic Hodge theory (cf.\ \cite{scholze-p-adic-hodge,kedlaya-liu-relative-padichodge}), as well as \cite{liu-zhu-rigidity,du-liu-moon-shimizu-completed-prismatic-F-crystal-loc-system,GuoReinecke-Ccris,du-liu-moon-shimizu-purity-F-crystal, GuoYangpointwise}, it is now possible to study these phenomena in $p$-adic geometric families. In this paper, we explore how Grothendieck's philosophy, or a relative analog of Berger's theorem, governs the behavior of $p$-adic local systems of geometric origin on rigid-analytic varieties over $p$-adic fields. 

The main result of this paper is the following.

\begin{thm}[Corollary~\ref{cor: RpMT over dense}]\label{thm: RpMT over dense open}
Let $X$ be a connected smooth rigid-analytic space over a $p$-adic field $K$, and let $\mathbb{L}$ be a de Rham $\mathbb{Q}_p$-local system on $X$. There exists a dense open subset $U \subset X$ such that the relative $p$-adic monodromy theorem holds over $U$. That is, there exists a \emph{de Jong} \'etale covering $V \to U$ of smooth rigid-analytic spaces over $K$, such that $\mathbb{L}|_{V}$ is log-crystalline at all classical points of $V$. Moreover, $U$ can be chosen to contain all classical points of $X$.
\end{thm}

The notion of a \emph{de Jong} \'etale covering of a rigid-analytic space was introduced by de Jong in \cite{dJfundamentalgroup}. The key new ingredient for proving Theorem~\ref{thm: RpMT over dense open} is the following. 

\begin{thm}[$\CNP \Longrightarrow \RpMT$, Theorem~\ref{thm: global case}]\label{thm: CNP to RpMT}
Let $X$ be a smooth rigid-analytic space over a $p$-adic field $K$, and let $\mathbb{L}$ be a de Rham $\mathbb{Q}_p$-local system on $X$. Suppose that the Newton polygon function $\NP(\mathbb{L})$ associated with $\mathbb{L}$ is constant \emph{($\CNP$)}, then the relative $p$-adic monodromy theorem \emph{($\RpMT$)} holds for $\mathbb{L}$. More precisely, there is a \emph{de Jong} \'etale covering $Y \to X$, such that $\mathbb{L}|_{Y}$ is log-crystalline at all classical points of $Y$. 
\end{thm}

Here, the Newton polygon function $\NP(\mathbb{L})$ is defined via the Newton polygon function of the admissible modification associated with the de Rham local system $\mathbb{L}$. We will review this construction later in the introduction; it is essentially motivated by Pappas and Rapoport \cite[\S2.6]{Pappas-Rapoport-padicsht}, who used this to study integral models of Shimura varieties. Furthermore, we prove that there is always a dense open subset of $\lvert X \rvert$ on which the Newton polygon function is locally constant (cf.\ Proposition~\ref{prop: CNP is dense open}). We will show that Theorem~\ref{thm: RpMT over dense open} holds over this dense open subset. 

We emphasize that Theorem~\ref{thm: CNP to RpMT} is genuinely global: even if $\RpMT$ holds in a neighborhood of every point of $X$, the additional content of Theorem~\ref{thm: CNP to RpMT} is that the covering can be chosen is a de Jong covering. Taking the disjoint union of local coverings does not give this conclusion. In particular, using the properties of de Jong covering, one can show if $X$ is connected in Theorem~\ref{thm: CNP to RpMT}, then $Y$ can also be chosen to be connected. Those properties will be reviewed in Appendix~\ref{app:local-systems}.

Theorem~\ref{thm: CNP to RpMT} shows that the relative $p$-adic monodromy theorem can follow from a purely topological property: the local constancy of the Newton polygon function associated with $\mathbb{L}$. While this implication cannot always be reversed, we will demonstrate that the converse holds locally around rank-$1$ points. To make this precise, we view $X$ as an adic space, and the second main result of this paper establishes that the local relative $p$-adic monodromy theorem around these points is governed entirely by this topological condition.

\begin{thm}[$\RpMT_1 \iff \CNP_1$, Theorem~\ref{thm: one direct} and \ref{thm: constant case}]\label{thm:main_rpmt1}
Let $X$ be a smooth rigid-analytic space over a $p$-adic field $K$, and let $\mathbb{L}$ be a de Rham $\mathbb{Q}_p$-local system on $X$. The local relative $p$-adic monodromy theorem around rank-$1$ points \emph{($\RpMT_1$)} holds for $\mathbb{L}$ if and only if the Newton polygon function $\NP(\mathbb{L})$ is locally constant around rank-$1$ points.
\end{thm}

\begin{rem}\label{rem: intro rem on main thm }
\begin{enumerate}
	\item The claim $\CNP_1 \Longrightarrow\RpMT_1$ in Theorem~\ref{thm:main_rpmt1} is a direct consequence of Theorem~\ref{thm: CNP to RpMT}. We will see that over a base where the Newton polygon function $\NP(\mathbb{L})$ associated with a de Rham local system $\mathbb{L}$ is constant, both the local and global relative $p$-adic monodromy theorems hold. See Theorems~\ref{thm: constant case} and \ref{thm: global case} for details.
	\item The converse, however, fails if one replaces rank-$1$ points by arbitrary points: the Newton polygon function associated with a de Rham local system is not always locally constant around higher-rank points; see Remark~\ref{rem: nonCNP}.
\end{enumerate}
\end{rem}

Here, by the local relative $p$-adic monodromy theorem around rank-$1$ points, we mean the following conjecture.

\begin{conj}[$\RpMT_1$]\label{conj: RpMT1}
Let $X$ be a smooth rigid-analytic space over a $p$-adic field $K$, and let $\mathbb{L}$ be a de Rham $\mathbb{Q}_p$-local system on $X$. Every rank-$1$ point $x \in X$ admits an open neighborhood $V \subset X$ and a \emph{finite} \'etale covering $V' \to V$ such that the restriction $\mathbb{L}|_{V'}$ is log-crystalline at all classical points of $V'$.
\end{conj}

\begin{rem}
\begin{enumerate}
	\item In \cite{Shimizu-monodromy}, Shimizu establishes a local relative $p$-adic monodromy theorem around classical points. As we discuss later in Remark~\ref{rem: shimizu thm}, this result provides an alternative proof for the local constancy of the Newton polygon function around classical points. In forthcoming work, we will directly prove $\CNP_1$ for de Rham local systems.
	\item In \cite[\S3.2]{kedlaya2025monodromyfamily}, Kedlaya proposes a strategy toward $\RpMT_1$ using the theory of $p$-adic differential equations in families. In this paper, we do not take $\RpMT_1$ as an input unless explicitly stated. In particular, Theorem~\ref{thm: RpMT over dense open} is proved without assuming $\RpMT_1$. 
	\item However, the local results mentioned above, combined with Theorem~\ref{thm: CNP to RpMT}, yield global implications. For example, Shimizu's result implies that the dense open subset $U$ in Theorem~\ref{thm: RpMT over dense open} contains all classical points (cf.\ Corollary~\ref{cor: RpMT over dense}). Furthermore, we show that $\RpMT_1$ implies a substantially stronger statement; see Theorem~\ref{CNP to RpMTw}. We further conjecture that $\RpMT_1$ should fully imply $\RpMT$; see Remark~\ref{rem: on RpMT1 to RpMT}.
\end{enumerate}
\end{rem}

\begin{rem}\label{rem: intro rem on main conjecture}
Assume Conjecture~\ref{conj: RpMT1}, and let $V'\to V$ be a finite \'etale covering as in its statement. By Theorem~\ref{thm: localalteration}, after a further \'etale covering of $V'$, we may work on the generic fiber of a semistable formal scheme over the ring of integers of a finite extension of $K$. The pointwise criterion \cite[Thm.~1.1]{GuoYangpointwise} then shows that the pullback of $\mathbb L$ is log-crystalline with respect to this formal model. 
\end{rem}

As we will also demonstrate, $\RpMT_1$, or equivalently by Theorem~\ref{thm:main_rpmt1}, the local constancy of the Newton polygon function associated with a de Rham local system around rank-$1$ points (which we denote $\CNP_1$ for short), actually has global implications.

\begin{thm}[{$\RpMT_1\Longrightarrow\RpMT^w$, cf.\ Corollary~\ref{cor: RpMTw}}]\label{CNP to RpMTw}
Let $X$ be a connected smooth rigid-analytic space over a $p$-adic field $K$, and let $\bL$ be a de Rham $\mathbb{Q}_p$-local system on $X$. \emph{Assuming $\RpMT_1$}, there is a smooth rigid-analytic space $Y$, and an \'etale morphism $\pi\colon Y \to X$ of smooth rigid-analytic varieties over $K$, such that $\mathbb{L}|_{Y}$ is log-crystalline at all classical points of $Y$. Moreover, the image of $\pi$ contains all rank-$1$ points of $X$, and $\pi$ defines a de Jong \'etale covering of its image.
\end{thm}

In other words, assuming $\RpMT_1$, for any de Rham local system over a connected rigid-analytic variety, there is a ``\emph{non-admissible} but still nice \'etale covering'' such that the de Rham local system becomes pointwise log-crystalline. We refer to this result as $\RpMT^w$. To compare with this, we recall the relative $p$-adic monodromy conjecture.

\begin{conj}[{$\RpMT$ for de Rham $\mathbb{Q}_p$-local systems}]\label{conj: RpMT}
Let $X$ be a connected smooth rigid-analytic space over a $p$-adic field $K$, and let $\bL$ be a de Rham $\mathbb{Q}_p$-local system on $X$. There is a \emph{de Jong} \'etale \emph{covering} $\pi\colon Y \to X$ of smooth rigid-analytic varieties over $K$ such that $\mathbb{L}|_{Y}$ is log-crystalline at all classical points.
\end{conj}

\begin{rem}\label{rem: on RpMT1 to RpMT}
\begin{enumerate}
	\item Note that although we deduce a global theorem from local behavior around rank-$1$ points, we do not require $X$ to be quasi-compact. 	
	\item The main results of this paper establish the following logical relations:
	\[
	\RpMT_1 \iff \CNP_1 \iff \RpMT^w \Longleftarrow \RpMT.
	\] 
	Furthermore, because higher-rank points in adic spaces behave analogously to higher-codimension boundaries, it is natural to expect purity-type results for de Rham and log-crystalline local systems over smooth adic spaces. We therefore conjecture that $\RpMT^w \implies \RpMT$, which would render all four statements actually equivalent.
	\item In \cite[Rem.~1.4]{liu-zhu-rigidity}, Liu and Zhu conjecture that if one further assumes $\mathbb{L}$ is a de Rham $\Z_p$-local system, the covering space $Y$ can be chosen to be \emph{finite} \'etale over $X$. At present, we are able to establish this only locally (cf.\ Theorem~\ref{thm: constant case}), under the assumption of locally constancy of Newton polygons.
\end{enumerate}	
\end{rem}

\begin{rem}
Conjecture~\ref{conj: RpMT} is widely open. The only known case is when $\bL$ is Hodge--Tate of weight $0$, which is due to Shimizu (cf.\ \cite[Thm.~1.6]{Shimizu-HT}) and a result of de Jong (cf. (1) in Remark~\ref{rem: relation with lattice cover}); Theorem~\ref{thm: CNP to RpMT} should be viewed as a generalization of such result. 
\end{rem}

A different but closely related application is the following result conjectured by Howe and Klevdal.

\begin{thm}[{\cite[Conj.~4.3.4]{HoweKlevdalIII}}]\label{thm: intro HK conjecture}
Let $X$ be a smooth rigid-analytic variety over a $p$-adic field $K$. The following functor is an equivalence:
\[
\omega_{\et}^{\mathcal L}\colon
\AdmPair^{\mathrm{basic},\mathrm{pgr}}(X)
\xrightarrow{\sim}
\mathrm{VHS}(X).
\]
\end{thm}

We recall the definitions of the two categories in \S\ref{sec: proof of HK conjecture}. The functor in the theorem is constructed in \cite[Prop.~4.3.3]{HoweKlevdalIII}, and its full faithfulness follows from \cite[Thm.~3.1.2]{HoweKlevdalIII}. It therefore remains to prove essential surjectivity.

Roughly speaking, the conjecture says that basic admissible pair attached to a variation of $p$-adic Hodge structure on $X$ becomes neutral after pullback to a locally potentially unramified pro-\'etale covering of $X$, see Definition~\ref{defn:unramified-proet}. That is, after such pullback, the admissible pair is associated with a constant isocrystal. This provides the structure required to define period maps associated with variations of $p$-adic Hodge structure envisioned by Howe--Klevdal. A new ingredient in the proof is a purity theorem for strongly \'etale morphisms, established in Appendix~\ref{app:valuation}.

\medskip\noindent
\textbf{Shtukas associated with de Rham local systems.} Let $X$ be a smooth rigid-analytic space over a $p$-adic field $K$, and let $\bL$ be a de Rham local system on $X$. An important arithmetic setting in which these structures naturally emerge is the study of Shimura varieties and their local analogs. For instance, over the rigid analytification of a Shimura variety of Hodge type (base changed to a $p$-adic field), the relative $p$-adic Tate module $\bL$ of the universal abelian scheme defines a $\mathbb{Z}_p$-local system. This $\bL$ is of geometric origin, and it is in fact de Rham by \cite[Cor.~4.9]{liu-zhu-rigidity}. Moreover, the $p$-adic Hodge-theoretic properties of this universal local system are expected to govern the geometry of the Shimura variety, including its reduction modulo $p$. Equivalently, they can be used to demonstrate the existence of, or characterize, integral models. These expectations lie at the heart of the original motivation for relative $p$-adic Hodge theory; for an incomplete list, see \cite{Faltings-CryscohandGalrep}, \cite{Kisin-Sh}, \cite{Pappas-Rapoport-padicsht}, and \cite{ImaiKYtannakian}.

The geometric approach utilized by Pappas and Rapoport originates from Scholze's transformative vision in his Berkeley lectures \cite{SWBerkeleyNotes} and ICM lecture \cite{ScholzeICM}. Scholze defined \emph{local Shimura varieties} as purely $p$-adic analogs of classical Shimura varieties, over which the universal family is governed by a \emph{$p$-adic local shtuka}. Building on this foundation, Pappas and Rapoport demonstrated that these universal shtukas provide the machinery to capture the arithmetic data necessary for the study of canonical integral models.

One of Pappas and Rapoport's observations is that the relevant universal shtuka over the (rigid-analytified) Shimura variety can be constructed from the universal de Rham local system. Moreover, their construction works for a general smooth base. This differs from previous perfectoid methods, which require passing to the infinite-level Shimura variety and then studying the Hodge--Tate period maps (cf.\ Scholze \cite{Scholze15torsion}, Caraiani--Scholze \cite{CaraianiScholzecpt}).

\medskip\noindent
\textbf{The Newton polygon function at a point.} To explain the definition of the Newton polygon function and the construction of Pappas and Rapoport, let us first consider the situation where the base rigid space is a point $x=\Spa(K,\calO_K)$ defined by a $p$-adic field $K$. In this case, a de Rham $\Z_p$-local system over $x$ is equivalent to a stable $\Z_p$-lattice $T$ in a de Rham representation of $G_K$. The admissible modification on the Fargues--Fontaine curve associated with a de Rham local system at $x$ was essentially developed by Fargues--Fontaine in \cite{FarguesFontaineAsterisque} by geometrizing the earlier work of Berger on $B$-pairs into vector bundles over the Fargues--Fontaine curve.

Recall that the absolute Fargues--Fontaine curve $X_C$ is defined for any complete algebraically closed non-archimedean field $C$ over $\F_p$. Fargues and Fontaine showed that isomorphism classes of vector bundles over $X_C$ are in bijection with multisets of rational numbers given by the Harder--Narasimhan slopes of the vector bundle, cf.\ \cite{FarguesFontaineAsterisque}. To discuss the Newton polygon function associated with a de Rham representation of $G_K$, let $C$ be the tilt of a complete algebraic closure $\C_p$ of $K$. There is a distinguished closed point $\infty$ on $X_C$ such that the complete stalk $\calO_{X_C,\infty}^\wedge$ of the structure sheaf $\calO_{X_C}$ at $\infty$ is isomorphic to the de Rham period ring $B_{\dR}^+$ defined by Fontaine, and $X_C\setminus \{\infty\}$ is affine with ring global sections $B_e$ that is a PID and naturally a subring of $B_{\dR}$. The Beauville--Laszlo theorem holds for the Fargues--Fontaine curve $X_C$, meaning that a vector bundle on $X_C$ is equivalent to a pair of vector bundles over $B_e$ and $B_{\dR}^+$ together with gluing data over $B_{\dR}$. Given these facts, from a de Rham representation $V$ of $G_K$, one can construct a vector bundle $\calE_0(V)$ corresponding to the triple $(V\otimes B_e, D_{\dR}(V)\otimes_K B_{\dR}^+, \alpha_{\dR})$, where the gluing is provided by the de Rham comparison isomorphism $\alpha_{\dR} \colon D_{\dR}(V)\otimes_K B_{\dR} \xrightarrow{\cong} V\otimes_{\Q_p} B_{\dR}$. The Newton polygon function $\NP(V)$ is defined to take values in multisets of rational numbers, and the value $\NP(V)(x)$ at $x$ is defined as the multiset of slopes of $\calE_0(V)$. 

The $p$-adic Hodge-theoretic meaning of the multiset of slopes of $\calE_0(V)$ was also explained in \cite{FarguesFontaineAsterisque}, clarifying its natural relationship to the $p$-adic monodromy theorem. By Berger's theorem, the de Rham representation $V$ is potentially log-crystalline; let $D_{\pst}(V)$ be its potentially log-crystalline period module, which is by definition an admissible filtered $(\varphi, N, G_K)$-module. Fargues--Fontaine showed that the multiset of slopes defined by the Frobenius structure of the underlying $\varphi$-module of $D_{\pst}(V)$ is equal to $\NP(V)(x)$ up to a negative sign. The key input is Fargues' theorem regarding the classification of admissible modifications of vector bundles on $X_C$ along $\infty$; it also suggests that one does not just work with $\calE_0(V)$; $V$ and $D_{\pst}(V)$ define a $G_K$-equivariant admissible modification of vector bundles $\calE_0(V) \dashrightarrow \calE_1(V)$ along $\infty$.

\medskip\noindent
\textbf{The Newton polygon function associated with arithmetic local systems.} Using the relative $p$-adic Hodge theory of de Rham local systems developed in \cite{scholze-p-adic-hodge} and \cite{kedlaya-liu-relative-padichodge}, and building upon the work of Pappas and Rapoport \cite{Pappas-Rapoport-padicsht} (who essentially defined $\calE_0(\bL) \dashrightarrow \calE_1(\bL)$ for a de Rham local system $\bL$), we generalize the construction of $\NP(V)$ to $\NP(\mathbb{L})$ as a function from the underlying topological space of $X$ to multisets of rational numbers in \S\ref{sec: Shtukas}. 

To understand the $p$-adic Hodge-theoretic meaning of $\NP(\mathbb{L})$, one must apply recent developments in relative $p$-adic Hodge theory for log-crystalline local systems developed in \cite{du-liu-moon-shimizu-purity-F-crystal}, which generalize the trivial log structure setting from \cite{du-liu-moon-shimizu-completed-prismatic-F-crystal-loc-system,GuoReinecke-Ccris}. The key is that log-crystalline local systems over a semistable $p$-adic formal scheme $\calS$ over $\Spf \calO_K$ can be classified using log prismatic $F$-crystals over $\calS$. Restricting to the perfect site defines an admissible modification of vector bundles over the diamantine relative Fargues--Fontaine curve $X_S^\diamondsuit$ along $S^\diamondsuit \to \Spd K$, where $S$ is the rigid generic fiber of $\calS$. (The definitions of $X_S^\diamondsuit$ and admissible modifications on it are detailed in \S\ref{sec: diamantine FF curves} and \S\ref{sec: Shtukas}). In particular, given a log-crystalline local system $\bL$ on $S$ defined with respect to $\calS$, we will construct an admissible modification $\calE^{\cris}_0(\bL) \dashrightarrow \calE^{\cris}_1(\bL)$ of vector bundles over $X_S^\diamondsuit$ along $S \to \Spd K$. By generalizing Fargues' theorem to families, we can show the following.

\begin{prop}[Proposition~\ref{prop: Tsht are compatible}]
$\calE^{\cris}_0(\bL) \dashrightarrow \calE^{\cris}_1(\bL)$ agrees with $\calE_0(\bL) \dashrightarrow \calE_1(\bL)$, where the second admissible modification is given by viewing $\bL$ as a de Rham local system and applying the construction of Pappas--Rapoport.
\end{prop}

A log-crystalline local system $\bL$ defines an $F$-isocrystal $\calE_\cris(\bL)$ on the log-crystalline site of the reduced special fiber $\calS_{\mathrm{red}}$ of $\calS$, given by the crystalline realization functor of the log prismatic $F$-crystal. Note that for any point $y$ in $\calS_{\mathrm{red}}$, the restriction of $\calE_\cris(\bL)$ at $y$ defines an isocrystal $\calE_\cris(\bL)_y$ over $\kappa(y)$, and $\calE_\cris(\bL)$ defines a Newton polygon function $\NP(\calE_\cris(\bL))$ by taking the multiset of Newton slopes of these isocrystals. On the other hand, since $\calS$ is an admissible formal model of $S$ by definition, it defines a specialization map $\spe_\calS\colon \lvert S \rvert \to \lvert \calS_{\mathrm{red}} \rvert$. We will consider a modified version of it given by
\[
\spe_{\max,\calS}\colon \lvert S \rvert \longrightarrow \lvert \calS_{\mathrm{red}} \rvert
\]
\[
\quad\qquad\qquad s \mapsto \spe_\calS(s_{\max})
\]
where $s_{\max}$ is the (unique) maximal (rank-$1$) generalization of $s$. An important observation of ours is the following.

\begin{thm}\label{thm: intro NP under sp}
The Newton polygon function $\NP(\bL)$ at $s$ in $\lvert S \rvert$ agrees with $\NP(\calE_\cris(\bL))$ at $\spe_{\max,\calS}(s)$ up to a minus sign.
\end{thm}

Theorem~\ref{thm: intro NP under sp} and Remark~\ref{rem: intro rem on main conjecture} provide the key inputs for proving that $\RpMT_1$ implies $\CNP_1$.

\begin{rem}
In Lemma~\ref{lem: anticontinuous}, we show that the map $\spe_{\max,\calS}$ is anticontinuous. That is, locally, the preimage of a closed (resp.\ open) subset of the special fiber is open (resp.\ closed). This explains why the closure relations of the Newton stratification on flag varieties, as noted in \cite[Rem.~1.17]{CaraianiScholzecpt}, are reversed compared to those of the usual stratification on the special fiber. Lemma~\ref{lem: anticontinuous} also generalizes the corresponding statement for Berkovich spaces, formulated using the specialization map $\spe_{\calS}$; cf.\ \cite[Exercise~3.4.1.5]{TemkinIntroBerkovich}.
\end{rem}

\medskip\noindent
\textbf{Newton partition and towards $\RpMT$.} The function $\NP(\bL)$ defines a Newton partition of $X$ with respect to $\bL$. Assuming $\RpMT_1$, we will show that the parts of this partition have large interiors in the following sense.

\begin{propdefn}[cf.\ Definition~\ref{defn: newton partitions} and Proposition~\ref{prop: properties of |S|P}]
Let $X$ be a connected smooth rigid-analytic space over $K$, and let $\bL$ be a de Rham $\mathbb{Q}_p$-local system on $X$. The function $\NP(\bL)$ defines a \emph{finite} Newton partition of the underlying topological space of $X$:
\[
\lvert X \rvert = \coprod_{\calP} \lvert X \rvert^{\calP},
\]
where the disjoint union is indexed by multisets of rational numbers $\calP$, and a point $x \in \lvert X \rvert$ belongs to $\lvert X \rvert^{\calP}$ if and only if $\NP(\bL)(x)=\calP$. The union of the interiors of these parts contains a dense open subset of $\lvert X \rvert$. Moreover, assuming $\RpMT_1$, the interior of each non-empty part $\lvert X \rvert^{\calP}$ contains all rank-$1$ points of $\lvert X \rvert^{\calP}$.
\end{propdefn}

\begin{rem}
We remark that for any analytic adic space $S$ that admits a $\vv$-covering $S_\infty \to S^\diamondsuit$ by a perfectoid space $S_\infty$, a vector bundle $\calE$ over the relative Fargues--Fontaine curve $X_{S_\infty}$ considered by them defines a Newton ``partition'' of the underlying topological space of $S$. However, the parts can have empty interiors. A classic example is the Hodge--Tate period map for the modular curve $\pi_{\mathrm{HT}}\colon \mathcal{X} \to \bP_{\Q_p}^{1,\an}$. The Newton polygon function of the universal family defines a decomposition $\mathcal{X} = \mathcal{X}^{\mathrm{ord}}\coprod \mathcal{X}^{\mathrm{ss}}$, and both should have large interiors by our result. However, the image of $\mathcal{X}^{\mathrm{ord}}$ collapses into $\bP^{1}(\Q_p)$ under $\pi_{\mathrm{HT}}$. See the discussions after \cite[Thm.~III.1.2]{Scholze15torsion}.
\end{rem}

If $\bL$ is log-crystalline with respect to a smooth or semistable formal model of $X$, Theorem~\ref{thm: intro NP under sp} expresses its Newton polygon function in terms of its crystalline realization on the special fiber of the formal model. In general, $\NP(\bL)$ is defined on the generic fiber without choosing a formal model. The Newton partition reflects the potential reduction behavior of $\mathbb{L}$, cf. Remarks~\ref{rem: intro rem on main conjecture} and \ref{rem: intro rem on main thm }, Theorem~\ref{thm: intro NP under sp}, and the following example. 

\medskip\noindent
\textbf{An example.} We find the following example particularly illustrative of a general Newton partition of $X$ defined by a de Rham local system $\mathbb{L}$.

In this example, let $p\neq 2$, and consider the Legendre family of elliptic curves over the base $S = \mathbb{A}^1_{\mathbb{Q}_p} \setminus \{0, 1\}$. Let $X$ be the rigid analytification $S^{\mathrm{an}}$ of $S$, which is the analytic affine line minus two points. Over $S$, we have the universal Legendre elliptic curve given by the Weierstrass equation:
\[
    E_{\lambda}: \quad y^2 = x(x-1)(x-\lambda).
\]
The rational Tate module of the elliptic curve $E_{\lambda}$, denoted $\mathbb{L}=V_p(E_{\lambda})^\vee$, defines a de Rham local system on $X$. For any rank-$1$ point $x \in X$, we let $\mathcal{H}(x)$ denote its completed residue field, and let $V_x = \mathbb{L}|_{\overline{x}}$ be the restriction of $\mathbb{L}$ at a geometric point lying over $x$, viewed as a continuous representation of $\mathrm{Gal}(\overline{\mathcal{H}(x)}/\mathcal{H}(x))$.

When $x$ is a classical point, by $p$-adic Hodge theory, $V_x$ is a potentially semi-stable representation. Let $D_{\mathrm{pst}}(V_x)$ be its associated potentially semi-stable filtered $(\varphi, N)$-module. The Frobenius slopes of $D_{\mathrm{pst}}(V_x)$ coincide with the slopes of the Newton polygon $\mathrm{NP}(\mathbb{L})(x)$ (up to a minus sign), and they are dictated by the stable reduction type of $E_{\lambda,x}$, the fiber of $E_\lambda$ at $x$ (see \cite[Proposition VII.5.5]{Silverman1986book}).

\smallskip\noindent
\textbf{The good reduction locus $X^\circ$.} The Legendre family naturally arises from an integral model over $\mathbb{Z}_p$, defined by the affine scheme
\[
    Y = \mathrm{Spec}\, \mathbb{Z}_p\left[\lambda, \frac{1}{\lambda(\lambda-1)}\right].
\]
The scheme $Y$ is smooth over $\mathbb{Z}_p$ with special fiber $Y_{\mathbb{F}_p} = \mathbb{A}^1_{\mathbb{F}_p} \setminus \{0, 1\}$. Taking the $p$-adic completion of $Y$ yields the $p$-adic formal scheme 
\[
    \widehat{Y} = \Spf\mathbb{Z}_p\langle \lambda, U \rangle / \big(U\lambda(\lambda-1) - 1\big).
\]
The good reduction locus $X^\circ \subset X$ is defined as the Raynaud generic fiber of $\widehat{Y}$. Analytically, $X^\circ$ is the $\mathbb{Q}_p$-affinoid domain defined by $\{ x \in X \mid \lvert\lambda\rvert_x = 1 \text{ and } \lvert\lambda -1\rvert_x = 1 \}$. It is helpful to view $X^\circ$ as the complement in $\mathbb{P}^{1,\mathrm{an}}_{\mathbb{Q}_p}$ of three open unit discs: the discs around $0, 1$, and $\infty$, given respectively by $\{ x \in X \mid \lvert\lambda\rvert_x < 1\}$, $\{ x \in X \mid \lvert\lambda - 1\rvert_x < 1\}$, and $\{x \in X \mid \lvert\lambda\rvert_x > 1\}$. Inside $X^\circ$, the slopes of $D_{\mathrm{pst}}(V_x)$ are $(1/2,1/2)$ if and only if the Legendre curve $E_{\lambda,x}$ has supersingular good reduction; otherwise, the slopes are $(0,1)$ and the Legendre curve $E_{\lambda,x}$ has ordinary good reduction. $X^\circ$ is also the (potentially) good reduction locus in $X$ in the sense of \cite[Thm.~5.17]{ImaiMiedaPotgoodredloci}.

\smallskip\noindent
\textbf{Outside the good reduction locus $X^\circ$.} Consider the locus in $X$ given by points with $|\lambda|_x < 1$, $|\lambda-1|_x < 1$, or $|\lambda|_x > 1$. For classical points in this region, the $p$-adic absolute value of the $j$-invariant of $E_{\lambda,x}$ is strictly greater than $1$. Thus, $E_{\lambda,x}$ has potentially multiplicative reduction (it achieves split multiplicative reduction over a finite extension of $\mathcal{H}(x)$). Consequently, the slopes of $D_{\mathrm{pst}}(V_x)$ are identically $(0, 1)$.

\smallskip\noindent
\textbf{An explicit illustration for $p=7$.} To illustrate this with a concrete example, specialize to $p=7$. Notice that over $\overline{\mathbb{F}}_7$, the reduction of $E_{\lambda,x}$ is supersingular if and only if its parameter satisfies $\bar{\lambda} \in \{2, 4, 6\}$. 

Therefore, the locus of rank-$1$ points in $X$ where the slopes of $D_{\mathrm{pst}}(V_x)$ are strictly $(1/2, 1/2)$ is exactly the disjoint union of three open unit discs strictly contained inside $X^\circ$:
\[
    \{x \in X \mid |\lambda-2|_x < 1\} \sqcup \{x \in X \mid |\lambda-4|_x < 1\} \sqcup \{x \in X \mid |\lambda-6|_x < 1\}.
\]
On the entire complement of these three discs in $X$ (which includes both the ordinary good reduction locus in $X^\circ$ and the potentially multiplicative tubes outside $X^\circ$), the slopes of $D_{\mathrm{pst}}(V_x)$ are identically $(0, 1)$. See Figure~\ref{fig:region-decomposition}.

\begin{figure}[htbp]
\centering
\begin{tikzpicture}[x=1cm,y=1cm]

% outer rectangle
\draw[line width=1pt] (0,0) rectangle (12,7);

% red diagonal shading on the whole rectangle
\begin{scope}
  \clip (0,0) rectangle (12,7);
  \foreach \t in {-7.5,-7.25,...,12.5}{
    \draw[red, line width=0.35pt] (\t,0) -- ++(7.5,7.5);
  }
\end{scope}

% yellow vertical shading only inside the ellipse
\begin{scope}
  \clip (6.15,3.55) ellipse (5.55 and 2.95);
  \foreach \x in {0.6,0.72,...,11.7}{
    \draw[yellow!70!orange, line width=0.4pt] (\x,0.3) -- (\x,6.8);
  }
\end{scope}

% remove yellow from circles 0 and 1
\fill[white] (2.20,2.25) circle (0.65);
\fill[white] (3.25,4.10) circle (0.85);

% redraw red shading inside circles 0 and 1
\begin{scope}
  \clip (2.20,2.25) circle (0.65);
  \foreach \t in {-7.5,-7.25,...,12.5}{
    \draw[red, line width=0.35pt] (\t,0) -- ++(7.5,7.5);
  }
\end{scope}

\begin{scope}
  \clip (3.25,4.10) circle (0.85);
  \foreach \t in {-7.5,-7.25,...,12.5}{
    \draw[red, line width=0.35pt] (\t,0) -- ++(7.5,7.5);
  }
\end{scope}

% blue reverse diagonal shading in circles 2,4,6
\begin{scope}
  \clip (6.90,4.10) circle (0.90);
  \foreach \t in {-2.0,-1.75,...,18.0}{
    \draw[blue, line width=0.5pt] (\t,7.5) -- ++(-7.5,-7.5);
  }
\end{scope}

\begin{scope}
  \clip (8.95,4.10) circle (0.92);
  \foreach \t in {-2.0,-1.75,...,18.0}{
    \draw[blue, line width=0.35pt] (\t,7.5) -- ++(-7.5,-7.5);
  }
\end{scope}

\begin{scope}
  \clip (7.95,2.00) circle (0.88);
  \foreach \t in {-2.0,-1.75,...,18.0}{
    \draw[blue, line width=0.35pt] (\t,7.5) -- ++(-7.5,-7.5);
  }
\end{scope}

% ellipse boundary
\draw[yellow!70!orange, line width=1.6pt] (6.15,3.55) ellipse (5.55 and 2.95);

% circle boundaries
\draw[yellow!70!orange, line width=1pt] (2.20,2.25) circle (0.65); % 0
\draw[yellow!70!orange, line width=1pt] (3.25,4.10) circle (0.85); % 1
\draw[yellow!70!orange, line width=1pt] (4.70,3.25) circle (0.58); % 3
\draw[yellow!70!orange, line width=1pt] (5.10,1.95) circle (0.68); % 5
\draw[yellow!70!orange, line width=1pt] (6.90,4.10) circle (0.90); % 2
\draw[yellow!70!orange, line width=1pt] (8.95,4.10) circle (0.92); % 4
\draw[yellow!70!orange, line width=1pt] (7.95,2.00) circle (0.88); % 6

% small red circles inside 0 and 1
\draw[black, line width=0.5pt] (1.90,2.24) circle (0.05);
\draw[black, line width=0.5pt] (2.90,4.25) circle (0.05);

% labels with white background
\node[fill=white, fill opacity=0.9, text opacity=1, inner sep=1pt] at (2.20,2.25) {\Huge 0};
\node[fill=white, fill opacity=0.9, text opacity=1, inner sep=1pt] at (3.25,4.10) {\Huge 1};
\node[fill=white, fill opacity=0.9, text opacity=1, inner sep=1pt] at (4.70,3.25) {\Huge 3};
\node[fill=white, fill opacity=0.9, text opacity=1, inner sep=1pt] at (5.10,1.95) {\Huge 5};
\node[fill=white, fill opacity=0.9, text opacity=1, inner sep=1pt] at (6.90,4.10) {\Huge 2};
\node[fill=white, fill opacity=0.9, text opacity=1, inner sep=1pt] at (8.95,4.10) {\Huge 4};
\node[fill=white, fill opacity=0.9, text opacity=1, inner sep=1pt] at (7.95,2.00) {\Huge 6};

\end{tikzpicture}
\caption{A picture of the Newton partition of $(\bA_{\Q_7}^1)^{\an}\setminus \{0,1\}$ defined by the universal elliptic curve in Legendre family. 
The base is \(X=(\bA_{\Q_7}^1)^{\an}\setminus \{0,1\}\), depicted here as a rectangle with the two points \(0\) and \(1\) removed. 
The blue region is \(X^{(-\frac{1}{2},-\frac{1}{2})}\), and the red region is \(X^{(-1,0)}\). 
The yellow region stands for the good reduction locus \(X^\circ\).}
\label{fig:region-decomposition}
\end{figure}
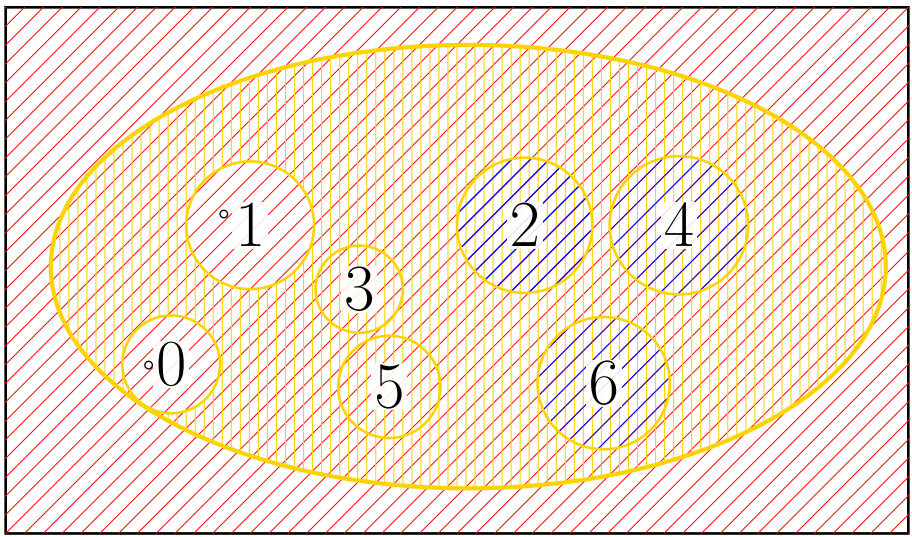

\begin{rem}\label{rem: nonCNP}
We note that $\NP(\bL)(x)=\NP(\bL)(x_{\max})$; that is, the Newton polygon function is invariant under taking maximal generalizations. In particular, the example above also demonstrates that $\NP(\bL)$ is, in general, not locally constant around higher-rank points. For instance, $\NP(\bL)$ is not locally constant at the rank-$2$ point that generalize to the Gauss point defined by the circle of radius $1$ centered at $2$.
\end{rem}

\medskip\noindent
\textbf{The idea of proving Theorem~\ref{thm: CNP to RpMT}.} Finally, let us briefly outline the strategy for proving the relative $p$-adic monodromy theorem under the assumption that the Newton polygon function $\NP(\bL)$ is constant.

We note that there are several established approaches to proving that a de Rham representation is potentially log-crystalline. These include methods utilizing $p$-adic differential equations \cite{Bergerpadicmonodromy}, Banach--Colmez spaces \cite{colmez-vector-space-de-Rham, Fontaine2004DeRhamSemistable}, improved methods using the absolute Fargues--Fontaine curve \cite{FarguesFontaineAsterisque}, and recent developments using analytic de Rham stacks \cite{analyticrhamstacks}.

Our strategy is heavily inspired by the approach of Fargues and Fontaine, which we review in detail in \S\ref{sec: proof of RpMT}. To illuminate the core idea, we recall that their method can be elegantly framed within the abstract ``criterion for equality of subcategories'' introduced by Colmez in the preface of \cite[\S2.2.3]{FarguesFontaineAsterisque}.

\begin{criterion}\label{crit: equality_subcategories}
Let $\mathscr{T}$ be a category in which the notions of dimension and exact sequence make sense, and suppose it is equipped with a symmetric relation called ``being neighbors''. Suppose we are given subcategories $\mathscr{T}' \supset \mathscr{T}''$ of $\mathscr{T}$ satisfying the following conditions:
\begin{enumerate}
    \item[(0)] The objects of $\mathscr{T}'$ and $\mathscr{T}''$ are of finite dimension.
    \item[(1)] If $T \in \mathscr{T}$ is of dimension $1$, then $T \in \mathscr{T}'$ implies $T \in \mathscr{T}''$.
    \item[(2)] If $T_1$ and $T_2$ are neighbors in $\mathscr{T}$, one has:
    \begin{enumerate}
        \item[a)] if $T_1$ and $T_2$ are of finite dimension, then $T_1 \in \mathscr{T}'' \Leftrightarrow T_2 \in \mathscr{T}''$,
        \item[b)] $T_1$ irreducible in $\mathscr{T}' \Rightarrow T_2 \in \mathscr{T}'$.
    \end{enumerate}
    \item[(3)] If $0 \to T_1 \to T \to T_2 \to 0$ is an exact sequence in $\mathscr{T}'$, and if $T_1, T_2 \in \mathscr{T}''$, then $T \in \mathscr{T}''$.
    \item[(4)] If $T \in \mathscr{T}'$, there exists a chain $T = T_0, T_1, \dots, T_r$ of elements in $\mathscr{T}$ such that $T_{i+1}$ is a neighbor of $T_i$ for all $i$, and $T_r \in \mathscr{T}''$.
\end{enumerate}
Then one has $\mathscr{T}'' = \mathscr{T}'$.
\end{criterion}

This criterion can be proved by induction on the dimension, provided all structures are well-defined. In \emph{loc.\ cit.}, Colmez reviews how this abstract framework can be utilized to prove both ``weakly admissible implies admissible'' and ``de Rham implies potentially log-crystalline,'' relying on the ``lemme fondamental'' established in \cite{colmez-fontaine} and \cite{Fontaine2004DeRhamSemistable}. Colmez also explains how this approach matches the strategy used to prove ``de Rham implies potentially log-crystalline'' in \cite{colmez-vector-space-de-Rham} and \cite{Fontaine2004DeRhamSemistable}.

The proof of ``de Rham implies potentially log-crystalline'' by Fargues and Fontaine offers a slightly more elegant formulation regarding their choices of the subcategories $\mathscr{T}'$ and $\mathscr{T}''$, as well as their interpretation of dimension and the induction step. To adapt their ideas to the framework of Criterion~\ref{crit: equality_subcategories}, one takes $\mathscr{T}'$ (resp.\ $\mathscr{T}''$) to be the category of de Rham (resp.\ potentially log-crystalline) $B_{\dR}^+$-flat Galois representations on $B$-pairs. (We will review the precise definition of $B_{\dR}^+$-flat Galois representations on $B$-pairs in \S\ref{sec: proof of RpMT}.) In this context, dimension and the neighborhood relation are defined via the Harder--Narasimhan theory of the absolute Fargues--Fontaine curve. In particular, objects of ``dimension $1$'' correspond to those Galois representations on $B$-pairs whose underlying vector bundles are semistable over the curve, and Condition $(2)$ in Criterion~\ref{crit: equality_subcategories} is not needed. To proceed with the proof, the key step is verifying Conditions $(1)$ and $(3)$ of Criterion~\ref{crit: equality_subcategories}. Condition $(3)$ corresponds to an ``$H^1_\st = H^1_g$'' or ``Hyodo's result'' type of statement, which in this setting eventually reduces to \cite[Prop.~10.11]{colmez-vector-space-de-Rham} (see also \cite[\S10.6.5]{FarguesFontaineAsterisque}). Condition $(1)$ asserts that all $B_{\dR}^+$-flat Galois representations on semistable $B$-pairs are potentially log-crystalline, a fact established in \cite[\S10.6.4]{FarguesFontaineAsterisque}.

In this paper, we adopt this same strategy to prove the relative $p$-adic monodromy theorem under the assumption that the Newton polygon function $\NP(\bL)$ is constant. By \cite[Thm.~7.4.9]{kedlaya-liu-relative-padichodge} (see also \cite[Thm.~II.2.19]{FarguesScholze} for the adic space setting), the constancy of the Newton polygon function guarantees the existence of a global Harder--Narasimhan filtration. We generalize all the essential ingredients of the Fargues--Fontaine proof to the relative setting, ensuring compatibility with their original construction over classical points. Specifically, the requisite ``Hyodo's result'' is already available, as it has been established pointwise.

To verify Condition $(1)$ of Criterion~\ref{crit: equality_subcategories} in our relative setting, we develop the theory of $\bB_{\dR}^+$-flat vector bundles on the relative diamantine Fargues--Fontaine curve, alongside the theory of $\bDhat_\mu^{\op}$-local systems on rigid-analytic spaces. Ultimately, we establish a Morita-type equivalence to reduce Condition $(1)$ to a relative version of Sen's theorem.

\medskip \noindent
\textbf{Outline.} In \S\ref{sec: arith loc sys}, we review the fundamentals of relative $p$-adic Hodge theory for arithmetic local systems on smooth rigid-analytic varieties. In \S\ref{sec: diamantine FF curves}, we recall the theory of vector bundles over the relative Fargues--Fontaine curve and the diamantine Fargues--Fontaine curve. In \S\ref{sec: Shtukas}, we review the Pappas--Rapoport construction of shtukas associated with de Rham local systems, define analogous shtukas for log-crystalline local systems, and establish their compatibility. We then define the Newton polygon function associated with de Rham local systems. We study the properties of the Newton polygon function associated with log-crystalline local systems under the specialization map, and show that it is locally constant around all rank-$1$ points. In \S\ref{sec: slopes}, we review slope filtrations, develop the theory of $(d,h)$-local systems, and relate them to semistable vector bundles over diamantine Fargues--Fontaine curves. In \S\ref{sec: proof of RpMT}, we recall the Fargues--Fontaine proof of ``$\dR \Longrightarrow \text{potentially log-crystalline}$'' and generalize their strategy to prove Theorems~\ref{thm: RpMT over dense open}, \ref{thm: CNP to RpMT}, and \ref{thm:main_rpmt1}. In \S\ref{sec: proof of HK conjecture}, as a further application of the theory of $\bDhat_\mu^{\op}$-local systems, we prove Theorem~\ref{thm: intro HK conjecture}, thereby verifying a conjecture of Howe and Klevdal. Finally, Appendix~\ref{sec: specialization} reviews specialization maps and tubes, together with the anticontinuity of $\spe_{\max,\calS}$. Appendix~\ref{app:local-systems} collects basic facts about de Jong coverings and geometric coverings appearing in the statements of our main results. Appendix~\ref{app:valuation} establishes a purity theorem for strongly \'etale morphisms, which is used in the proof of Theorem~\ref{thm: intro HK conjecture}.

\medskip
\noindent
\textbf{Notation and conventions.}
Throughout this paper, $K$ will denote a complete discrete valuation field containing $\Q_p$ with a perfect residue field, and $S$ will denote a smooth rigid-analytic variety over $K$. For any field extension $L$ of $K$, $S_L$ will be the base change of $S$ to $L$. We will view $S$ as an adic space. Let $\bT^d\coloneqq \Spa(R,R^{\circ})$ with $R=\Q_p\langle T_{1}^\pm,\ldots, T_{d}^\pm\rangle$.

Let $G_K$ be the absolute Galois group of $K$. Let $\C_p$ be the $p$-adic completion of a fixed algebraic closure of $K$, and let $\C_p^\flat$ denote its tilt. Let $K_\mathrm{cycl}$ be the $p$-adic completion of $\cup_{n\geq 1}K(\zeta_{p^n})$. Let $\breve{K}$ be the $p$-adic completion of the maximal unramified extension of $K$, so that $G_{\breve{K}}$ is the inertia subgroup $I_K$ of $G_K$.

Let $B_{\mathrm{st}}=B_{\cris}[u_{\underline{\varpi}}]$, where we fix a compatible system $\underline\varpi=(\varpi_n)_{n\geq0}$ with $\varpi_0=p$ and $\varpi_{n+1}^p=\varpi_n$, and put
\[
u_{\underline\varpi}=\log\left(\frac{[\underline\varpi]}p\right)
=-\sum_{i\geq1}\frac{(1-[\underline\varpi]/p)^i}{i}
\in\Fil^1 B_{\dR}^+.
\]
We equip $B_{\mathrm{st}}$ with the unique $B_{\cris}$-derivation $N$ determined by $N(u_{\underline{\varpi}})=1$. 

Let \(\Perf\) denote the category of perfectoid spaces in characteristic p, and \(\Perfd\) the category of all perfectoid spaces.

For a perfect field $k$, we let $\Isoc_k$ denote the category of isocrystals over $k$, i.e., pairs $(D,\varphi_D)$ where $D$ is a finite-dimensional vector space over the fraction field $W(k)[1/p]$ of the Witt vectors of $k$, and $\varphi_D$ is a $\sigma$-linear automorphism of $D$. 

Let $h>0$ be a real number, and let $\Conv([0,h])$ be the set of real-valued continuous \emph{convex} functions on $[0,h]$.
For $f,g\in\Conv([0,h])$, we define a partial order on $\Conv([0,h])$ as follows:
\[
f \preceq g
\quad\text{ if and only if } \quad
f(0)=g(0) \text{, }\ f(h)=g(h),\ \text{and}\ f(t)\le g(t)\ \text{for all }t\in[0,h].
\]
We also define 
\[
\Conv \coloneqq \bigsqcup_{0<h<\infty}\Conv([0,h])
\]
to be the set of real-valued continuous \emph{convex} functions defined on $[0,h]$ for some finite $h$.

In any exact category, we say a sub-object of a torsion-free object is \emph{saturated} if its corresponding quotient is torsion-free.

For any (possibly non-commutative) ring $R$, let $\mathrm{FPMod}_{R}$ denote the category of finite projective \emph{left} $R$-modules. The category of finite projective \emph{right} $R$-modules is naturally equivalent to $\mathrm{FPMod}_{R^{\op}}$, where $R^{\op}$ is the opposite ring of $R$. In practice, we write $\mathrm{FPMod}_{R^{\op}}$ for the category of finite projective right $R$-modules.

For any $r \in \bR$, we let $\lfloor r \rfloor \coloneqq \max\{n\in\Z \mid n \leq r\}$.

When we write $\lambda=d/h \in \Q$, we always assume it is in simplest form; that is, $h$ is the minimal strictly positive integer such that $h\lambda \in \Z$, and $d$ is the integer $h\lambda$. For example, $0=0/1$ under these conventions.

\section{Arithmetic local systems on smooth rigid analytic varieties}\label{sec: arith loc sys}

Let $K$ denote a $p$-adic field, and let $X$ be a smooth rigid analytic variety over $K$. Let $X_{\proet}$ be the pro-\'{e}tale site of Scholze \cite[\S3]{scholze-p-adic-hodge}, let $\widehat{\mathcal{O}}_{X_{\proet}}$ be its structure sheaf, and let $\mathcal{O}_{X_{\et}}$ be the structure sheaf on the \'etale site $X_{\et}$. Let $\nu\colon X_{\proet} \rightarrow X_{\et}$ be the natural map of sites.

A \emph{$\Q_p$-local system} on $X$ means an \'etale $\Q_p$-sheaf $\mathbb L$ on the \'etale site $X_{\et}$ in the sense of \cite[Defn.~8.4.3]{kedlaya-liu-relative-padichodge}. Note that if $\mathbb{L}$ admits a $\Z_p$-lattice, we call it an isogeny $\Z_p$-local system. For any \emph{classical point} $x \in X$, $\overline{x}$ denotes a geometric point lying over $x$, defined by an algebraic closure of $k(x)$ or a complete algebraic closure, depending on the context. For a $\Q_p$-local system $\mathbb L$ on $X$, the fiber $\mathbb L_{\overline{x}}$ over $\overline{x}$ is a continuous $\Q_p$-representation of $G_{k(x)}\coloneqq\Gal(\overline{k(x)}/k(x))$, where $k(x)$ is the residue field of $x$.

The $p$-adic Hodge-theoretic properties of geometric families of Galois representations, namely $\Q_p$-local systems or isogeny $\Z_p$-local systems over $X$, have been extensively studied following the work of Scholze \cite{scholze-p-adic-hodge}, Kedlaya--Liu \cite{kedlaya-liu-relative-padichodge}, Tan--Tong~\cite{Tan-Tong}, Liu--Zhu \cite{liu-zhu-rigidity}, and Shimizu \cite{Shimizu-HT}, as well as through recent approaches based on prismatic methods \cite{bhatt-scholze-prismaticFcrystal, du-liu-moon-shimizu-completed-prismatic-F-crystal-loc-system, GuoReinecke-Ccris, du-liu-moon-shimizu-purity-F-crystal}. This section recalls several foundational results that will be utilized later in the paper.

We will freely use the period sheaves on the pro-\'etale site $X_{\proet}$, such as $\mathbb{B}_{\text{dR}}^{+}$, $\mathbb{B}_{\text{dR}}$, and $\mathcal O\mathbb B^+_{\dR}$ and $\mathcal O\mathbb B_{\dR}$ defined in \cite{scholze-p-adic-hodge} and \cite{diao-lan-liu-zhu-logRH}, $\bA_\cris$, $\bB_\cris^+$, and $\bB_\cris$ defined in \cite[\S2A]{Tan-Tong}, and $\mathcal{O}\mathbb{B}_{\text{HT}} = \operatorname{gr}^{\bullet}\mathcal{O}\mathbb{B}_{\text{dR}}$ and $\mathcal O\mathbb C\coloneqq \mathrm{gr}^0 \mathcal O\mathbb B_{\dR}$ studied in \cite[\S2.1]{liu-zhu-rigidity} and \cite[\S3.1]{Shimizu-HT}. Recall that there exists a canonical isomorphism $\operatorname{gr}^{i}\mathcal{O}\mathbb{B}_{\text{dR}} \cong \mathcal{O}\mathbb{C}(i)$ for each $i$, so that $\mathcal{O}\mathbb{B}_{\text{HT}} = \oplus_{i \in \mathbb{Z}} \mathcal{O}\mathbb{C}(i)$.

\subsection{de Rham and Hodge--Tate local systems}\label{sec: dR and HT loc sys}

In this subsection, we recall the notions of de Rham and Hodge--Tate local systems. The main references for this part are \cite{scholze-p-adic-hodge}, \cite{liu-zhu-rigidity}, and \cite{Shimizu-HT}.

The sheaves $\mathbb{B}_{\text{dR}}^{+}$ and $\mathcal{O}\mathbb{B}_{\text{dR}}^{+}$ are equipped with structural morphisms $\theta\colon \mathcal{O}\mathbb{B}_{\dR}^{+} \rightarrow \widehat{\mathcal{O}}_{X_{\proet}}$ and $\theta\colon\mathbb{B}_{\text{dR}}^{+} \rightarrow \widehat{\mathcal{O}}_{X_{\proet}}$, which endow them with the filtrations $\operatorname{Fil}^{i} = (\ker(\theta))^{i}$ for $i \ge 0.$ Let $t\in\Fil^1\mathbb B_{\dR}^+$ be the usual period, chosen locally on $X_{\proet}$. Endow $\mathcal O\mathbb B_{\dR}^+[1/t]$ with the filtration
\[
\Fil^i\bigl(\mathcal O\mathbb B_{\dR}^+[1/t]\bigr)=\sum_{j\geq-i}t^{-j}\Fil^{i+j}\mathcal O\mathbb B_{\dR}^+,\qquad i\in\Z.
\]
Following \cite[Defn.~2.2.10 and Rem.~2.2.11]{diao-lan-liu-zhu-logRH}, define $\mathcal O\mathbb B_{\dR}$ by completing with respect to this filtration. It carries the induced filtration.

The sheaf $\mathcal{O}\mathbb{B}_{\text{dR}}$ is naturally endowed with a $\mathbb{B}_{\text{dR}}$-linear connection
\[
\nabla = d \otimes \operatorname{id} \colon \mathcal{O}\mathbb{B}_{\text{dR}} \rightarrow \mathcal{O}\mathbb{B}_{\text{dR}} \otimes_{\mathcal{O}_{X_{\text{\'{e}t}}}} \Omega_{X/K}^{1}
\]
which satisfies Griffiths transversality.

For a $\mathbb{Q}_{p}$-local system $\bL$ on $X_{\et}$, let $\hat{\bL}$ be the corresponding local system on $X_{\proet}$. 

\begin{defn} \label{defn: de Rham local system}
A $\mathbb{Q}_{p}$-local system $\bL$ is called de Rham if $\hat{\bL}\otimes_{\underline{\bQ}_p} \mathbb{B}_{\text{dR}}^+$ is associated to a filtered vector bundle with integrable connection $(\mathcal{E}, \operatorname{Fil}^{\bullet}, \nabla)$ in the sense that there is an isomorphism
\[
\hat{\bL}\otimes_{\underline{\bQ}_p} \mathcal{O}\mathbb{B}_{\text{dR}} \cong \mathcal{E} \otimes_{\mathcal{O}_{X_{\text{\'{e}t}}}} \mathcal{O}\mathbb{B}_{\text{dR}},
\]
compatible with filtrations and connections. 
\end{defn}

\begin{rem}\label{rem: M0 and M_1 for de Rham as proet vb}
Following \cite[\S7]{scholze-p-adic-hodge}, the condition in the above definition is equivalent to the statement that there is an isomorphism of $\mathbb{B}_{\text{dR}}^{+}$-modules:
\[
\hat{\bL}\otimes_{\underline{\bQ}_p} \mathbb{B}_{\text{dR}}^{+} \cong \operatorname{Fil}^{0}(\mathcal{E} \otimes_{\mathcal{O}_{X_{\text{\'{e}t}}}} \mathcal{O}\mathbb{B}_{\text{dR}})^{\nabla=0}.
\]
Let $\bM_1(\bL) \coloneqq \hat{\bL}\otimes_{\underline{\bQ}_p} \mathbb{B}_{\text{dR}}^{+}$ be this $\mathbb{B}_{\text{dR}}^{+}$-local system in the sense of \cite[Defn.~7.1]{scholze-p-adic-hodge}. If the filtered vector bundle with integrable connection $(\calE,\nabla,\Fil^\bullet)$ associated with $\bL$ exists, it is unique (cf.\ \cite[Thm.~7.6]{scholze-p-adic-hodge}), and we will sometimes denote it by $(D_{\dR}(\bL),\nabla,\Fil^\bullet)$ following \cite{liu-zhu-rigidity}. There is an auxiliary $\mathbb{B}_{\text{dR}}^{+}$-local system $\bM_0(\bL)$ that is useful in the study of de Rham local systems and filtered vector bundles with integrable connections. We define it as follows: let $\bM_0(\bL) \coloneqq (D_{\dR}(\bL) \otimes_{\mathcal{O}_{X_{\text{\'{e}t}}}} \mathcal{O}\mathbb{B}_{\text{dR}}^+)^{\nabla=0}$, which agrees with $\bM_1(\bL)$ when $\Fil^\bullet$ is the trivial filtration. Note that 
\[
\bM_0(\bL) \otimes_{\mathbb{B}_{\text{dR}}^{+}, \theta} \widehat{\mathcal{O}}_{X_{\proet}} \cong D_{\dR}(\bL) \otimes_{\mathcal{O}_{X_{\et}}} \widehat{\mathcal{O}}_{X_{\proet}}, 
\]
following \cite[Prop.~7.9]{scholze-p-adic-hodge}.
\end{rem}

We now turn to Hodge--Tate theory, following \cite{Shimizu-HT} and \cite{Gerththesis}. Recall that the sheaf $\mathcal{O}\mathbb{B}_{\text{HT}}$ is equipped with a Higgs field
\[
\Theta = \operatorname{gr}^{\bullet}\nabla \colon \mathcal{O}\mathbb{B}_{\text{HT}} \rightarrow \mathcal{O}\mathbb{B}_{\text{HT}} \otimes_{\mathcal{O}_{X_{\text{\'{e}t}}}} \Omega_{X/K}^{1},
\]
which restricts to a Higgs field
\[
\Theta \colon \mathcal{O}\mathbb{C} \rightarrow \mathcal{O}\mathbb{C} \otimes_{\mathcal{O}_{X_{\text{\'{e}t}}}} \Omega_{X/K}^{1}(-1).
\]

For any $\mathbb{Q}_{p}$-local system $\bL$ on $X$, one defines a coherent $\mathcal{O}_{X_{\et}}$-module
\[
D_{\text{HT}}(\bL) = \nu_{*}(\hat{\bL}\otimes_{\underline{\bQ}_p} \mathcal{O}\mathbb{B}_{\text{HT}}).
\]

$D_{\text{HT}}(\bL)$ is $\Z$-graded with gradings $
\operatorname{gr}^{i}D_{\text{HT}}(\bL) = \nu_{*}(\hat{\bL}\otimes_{\underline{\bQ}_p} \mathcal{O}\mathbb{C}(i))$ and carries a Higgs field
\[
\theta_{\mathbb{L}} \colon D_{\text{HT}}(\bL) \rightarrow D_{\text{HT}}(\bL) \otimes_{\mathcal{O}_{X_{\text{\'{e}t}}}} \Omega_{X/K}^{1}
\]
which is homogeneous of degree $-1$ with respect to the grading above. 

\begin{defn}[{\cite[Rem.~2.5]{liu-zhu-rigidity}}]
A $\mathbb{Q}_{p}$-local system $\bL$ is called Hodge--Tate if $D_{\text{HT}}(\bL)$ is a vector bundle of rank equal to $\rank_{\mathbb{Q}_{p}} \bL$. 
\end{defn}

On the other hand, associated to $\hat{\bL}\otimes_{\underline{\bQ}_p} \mathcal{O}\mathbb{C}$, Liu and Zhu defined a (nilpotent) Higgs bundle $(\calH(\bL),\vartheta_{\bL}, \rho)$ over $X_{K_{\mathrm{cycl}}}$ equipped with a semilinear $\Gal(K_\mathrm{cycl}/K)$-action under the $p$-adic Simpson correspondence, cf.\ \cite[Thm.~2.1(i)]{liu-zhu-rigidity}. Moreover, Shimizu constructed the arithmetic Sen operator $\phi_{\bL} \in \End_{\calO_{X_{K_{\mathrm{cycl}}}}}(\calH(\bL))$ in \cite[Lem.-Defn.~3.4]{Shimizu-HT}, which recovers the classical Sen theory over $p$-adic fields.

We will need the following equivalent description of Hodge--Tate local systems, which also facilitates the definition of Hodge--Tate weights.

\begin{thm}[{\cite[Thm.~5.5]{Shimizu-HT}}]\label{thm: phiL and HT local sys}
Let $\bL$ be a $\mathbb{Q}_{p}$-local system on a smooth rigid space $X/K$. Then $\bL$ is Hodge--Tate if and only if $\phi_{\bL}$ is semisimple with integer eigenvalues as an endomorphism of $\calH(\bL)$. 
\end{thm}

\begin{rem}\label{rem: local constancy of HT wts}
For a Hodge--Tate local system $\bL$, the multiset of eigenvalues of $\phi_{\bL}$ from Theorem~\ref{thm: phiL and HT local sys} is called the Hodge--Tate weights of $\bL$. In \cite{Shimizu-HT}, the \emph{generalized} Hodge--Tate weights of an arbitrary local system $\bL$ are studied and proven to be locally constant, cf.\ \cite[Thm.~1.1]{Shimizu-HT}.
\end{rem}

We will be particularly interested in Hodge--Tate local systems with a single weight of $0$.

\begin{prop}\label{prop: Condition for L to be HT wt 0}
For a $\Q_p$-local system $\bL$ over $X$, if $\hat{\bL} \otimes_{\underline{\Q}_{p}} \widehat{\mathcal{O}}_{X_{\proet}}$ is a subquotient of a locally free $\widehat{\mathcal{O}}_{X_{\proet}}$-module $\widehat{\calF}$, with $\widehat{\calF} \cong \calE \otimes_{\mathcal{O}_{X_{\et}}} \widehat{\mathcal{O}}_{X_{\proet}}$ for a vector bundle $\calE$ over $X_{\et}$, then $\bL$ is Hodge--Tate with a single weight $0$.
\end{prop}
\begin{proof}
By \cite[Exm.~5.8]{Shimizu-HT}, it suffices to argue at each classical point $x \in X$. The assumption implies that as a $G_{k(x)}$-representation, $\bL_{\overline{x}}\otimes \C_p$ is a subquotient of $\C_p^r$ for some $r\in \N$. This implies $\bL_{\overline{x}}$ is a Hodge--Tate Galois representation of weight $0$.
\end{proof}

\begin{rem}\label{rem: HT wt 0 local system}
The converse of the above proposition also holds: $\bL$ is Hodge--Tate with a single weight $0$ if and only if $\hat{\bL} \otimes_{\underline{\Q}_{p}} \widehat{\mathcal{O}}_{X_{\proet}} \cong \calE \otimes_{\mathcal{O}_{X_{\et}}} \widehat{\mathcal{O}}_{X_{\proet}}$ for a vector bundle $\calE$ over $X_{\et}$, cf.\ \cite[Prop.~2.7.6]{Gerththesis}. This provides a general framework to define Hodge--Tate local systems of weight $0$. In this case, for a Hodge--Tate local system $\bL$ of weight $0$ defined over a $p$-adic field $K$, the associated $\widehat{\mathcal{O}}_{X_{\proet}}$-vector bundle $\hat{\bL} \otimes_{\underline{\Q}_{p}} \widehat{\mathcal{O}}_{X_{\proet}}$ over $X$ has arithmetic and geometric Sen operators both equal to $0$. Actually, the above discussion implies that the geometric Sen operators is equal to $0$ for $\hat{\bL} \otimes_{\underline{\Q}_{p}} \widehat{\mathcal{O}}_{X_{\C_p, \proet}}$ once the arithmetic Sen operator is $0$, where the geometric Sen operator is considered via \cite[Thm.~4.8]{HeuerpadicSimpson} and \cite[Thm.~1.0.3]{camargo2025geometricsen}, based on \cite{Pan2022Pi}. 
\end{rem}

\begin{thm}\label{thm: RpMT for HT0}
Let $\bL$ be an isogeny $\Z_p$-local system defined over a smooth rigid space $X$ over a $p$-adic field. If $\bL$ is Hodge--Tate with single weight $0$, then there is a finite \'etale covering $Y$ of $X$ such that $\bL|_Y$ is unramified at all classical points.
\end{thm}
\begin{proof}
This is \cite[Thm.~5.15]{Shimizu-HT}. Note that in \emph{loc.\ cit.}, the base field is assumed to be finite over $\Q_p$. This restriction is unnecessary because the key input is \cite[\S3.2]{Sen1980Invent}, where the residue field is assumed to be algebraically closed. For any $p$-adic field $K$ with a perfect residue field, one simply applies Sen's theory to $\breve{K}$. See also \cite[Cor.~3.2]{OhkuboLie}.
\end{proof}

\begin{defn}[Horizontally de Rham]\label{defn: horizontally de Rham}
A de Rham $\Q_p$-local system $\mathbb L$ on $X$ is called \emph{horizontally de Rham} if $(D_{\dR}(\mathbb L),\nabla)$ has a full set of horizontal sections, cf. \cite[Defn.~9.1(i)]{Shimizu-monodromy}.
\end{defn}

\begin{defn}\label{defn: phi-N-G-filtered}
Let $K$ be a $p$-adic field over $\Q_p$, and let $L$ be a finite Galois extension of $K$ with Galois group $G_{L/K}$. Write $K_0=\Frac W(k_K)$ and $L_0=\Frac W(k_L)$.
\begin{enumerate}
\item A \emph{$(\varphi,N)$-module over $L_0$} is a finite-dimensional $L_0$-vector space 
$D$ endowed with an injective $\varphi$-semilinear endomorphism $\varphi\colon D\to D$, and an $L_0$-linear endomorphism $N \colon D \to D$ such that $N\varphi=p \varphi N$.
Morphisms are $L_0$-linear maps commuting with $\varphi$ and $N$. When $N=0$, we simply call it a \emph{$\varphi$-module over $L_0$}, which agrees with the notion of an isocrystal over $k_L$.
\item A \emph{filtered $(\varphi,N,G_{L/K})$-module} is a $(\varphi,N)$-module $(D,\varphi,N)$ equipped with a separated and exhaustive decreasing filtration $\Fil^\bullet D_L$ on $D_L \coloneqq D\otimes_{L_0} L$, together with a semilinear action $\rho_D$ of $\Gal(L/K)$ on $D$ (via $\Gal(L/K)\to \Gal(L_0/K_0)$) that commutes with $\varphi$ and $N$, and stabilizes the filtration $\Fil^\bullet D_L$ on $D_L$. We denote the corresponding category by $(\varphi,N,G_{L/K})\text{-}\mathrm{ModFil}_{L/L_0}$. When $L=K$, we simply write it as $(\varphi,N)\text{-}\mathrm{ModFil}_{K/K_0}$, calling it the category of filtered $(\varphi,N)$-modules.
\item We define the category of \emph{filtered $(\varphi,N,G_K)$-modules} by setting
\[
(\varphi,N,G_K)\text{-}\mathrm{ModFil}
=
\bigcup_{L/K}(\varphi,N,G_{L/K})\text{-}\mathrm{ModFil}_{L/L_0},
\]
where $L/K$ ranges over all finite Galois extensions of $K$ inside $\overline K$.
\end{enumerate}
For $(3)$, we utilize the fact that if $K \subset L \subset F$, with $F$ and $L$ being finite Galois over $K$, then the scalar extension functor
\[
(\varphi,N,G_{L/K})\text{-}\mathrm{ModFil}_{L/L_0}
\longrightarrow
(\varphi,N,G_{F/K})\text{-}\mathrm{ModFil}_{F/F_0}
\]
is fully faithful. We can also drop the filtration structure in the above definitions, referring to the resulting categories as $(\varphi,N,G_K)$-modules (and $(\varphi,G_K)$-modules in the case where $N=0$). If $R$ is a topologically finite type $K$-algebra such that there is an \'etale map $\Spa(R,R^{\circ})\to \bT^d_K$. One can similarly define the category of \emph{filtered $(\varphi,N,G_{L/K},R)$-modules} by requiring the filtration to be a filtration of $G_{L/K}$-stable saturated submodules on $D\otimes_{L_0}R_L$, cf.\ \cite[Defn.~5.2]{Shimizu-monodromy}.
\end{defn}

\subsection{Log-crystalline local systems}\label{sec: log-crystalline local systems}

In this subsection, let $(\calS, M_{\calS})$ be a bounded $p$-adic log formal scheme over $\Spf \calO_K$ such that $M_{\calS}$ is integral. More precisely, we will focus exclusively on the following two cases:

\begin{enumerate}
    \item $(\calS, M_{\calS})$ is a semistable formal scheme over $\Spf \calO_K$, and $M_{\calS}$ is the canonical log structure defined by $\calO_{\calS} \cap (\calO_{\calS}[1/p])^{\times} \to \calO_{\calS}$;
    \item $(\calS, M_{\calS})$ is a smooth $p$-adic formal scheme over $\Spf \calO_K$, and $M_{\calS}$ is the trivial log structure. In this case, we will simply write $\calS$ instead of $(\calS, M_{\calS})$.
\end{enumerate}
Let $S$ denote the adic generic fiber of $(\calS, M_{\calS})$. Note that in both cases, $S$ is a smooth rigid analytic variety over $K$ equipped with a trivial induced log structure. Let $(\calS_1,M_{\calS_1})$ be the modulo $p$ fiber with the induced log structure.

This subsection reviews basic facts about log-crystalline local systems on $S$ defined with respect to $\calS$ (which reduce to crystalline local systems when $M_{\calS}$ is trivial). We also review the prismatic classification of log-crystalline local systems, cf.\ \cite{du-liu-moon-shimizu-completed-prismatic-F-crystal-loc-system, GuoReinecke-Ccris}. 

\begin{rem}
For the definition of log-crystalline local systems, we primarily follow Faltings \cite{Faltings-CryscohandGalrep}; see also \cite[Defn.~2.31]{GuoReinecke-Ccris} and \cite[Defn.~3.39]{du-liu-moon-shimizu-purity-F-crystal}. This approach requires fixing a (logarithmic) formal model from the beginning. However, once a smooth rigid analytic variety $S$ over $K$ admits a semistable formal model, the definition of a (log-)crystalline local system becomes independent of the choice of formal model, cf.\ \cite[Cor.~5.5]{du-liu-moon-shimizu-purity-F-crystal}. Furthermore, when $S$ is affinoid, an affine semistable formal model for $S$ is unique if it exists. 
\end{rem}

\medskip
\noindent
\textbf{Log-crystalline local systems.} 
We have the definitions of $\Perfd/S_{\proet}$ and $S_{\proet}$ as in \cite[\S3]{du-liu-moon-shimizu-purity-F-crystal}. Recall that the crystalline period sheaves $\bA_\cris$, $\bB_\cris^+$, and $\bB_\cris$ on $S_{\proet}$ are defined in \cite[\S2A]{Tan-Tong}. We also recall the construction of the period sheaves defined by a finite locally free $F$-isocrystal $\calE_\Q$ on $(\calS_1,M_{\calS_1})_{\CRIS}$, as detailed in \cite{du-liu-moon-shimizu-purity-F-crystal}.

For any $U \in \Perfd/S_{\proet}$ associated with an affinoid perfectoid space $\Spa(A,A^+)$, we define a $\bB^+_\cris$-module $\bB_\cris^+(\calE_\Q)$ on $S_{\proet}$ by
\[
\bB_\cris^+(\calE_\Q)(U) \coloneqq \calE_\Q(\bA_{\cris}(A^+)) \otimes_{\bB_{\cris}^+(A^+)} \bB_\cris^+(U).
\]
Here, $\calE_\Q(\bA_{\cris}(A^+))$ is the evaluation of $\calE_\Q$ on the tuple
\[
(\Spec (A^+/p), \Spf(\bA_{\cris}(A^+)), M_{\Spf(\bA_{\cris}(A^+))}),
\]
which is naturally viewed as an (ind-)object in the log-crystalline site of $(\calS_1,M_{\calS_1})$. The sheaf $\bB_\cris(\calE_\Q)$ is defined analogously. Adopting the terminology of \cite{du-liu-moon-shimizu-purity-F-crystal} for finite locally free isocrystals, both $\bB_\cris^+(\calE_\Q)$ and $\bB_\cris(\calE_\Q)$ are finite locally free modules over $\bB^+_\cris$ and $\bB_\cris$, respectively. Finally, the assumption that $\calE_\Q$ is an $F$-isocrystal ensures that $\bB_\cris(\calE_\Q)$ admits a Frobenius map $F$.

\begin{construction}\label{const: F-isocrystal at x}
There is a variant of the above construction using the $\vv$-site. For $\calS$ as above, by \cite[\S17]{SWBerkeleyNotes}, we can attach a $\vv$-sheaf $\calS^\diamondsuit$ together with a morphism $\calS^\diamondsuit \to \Spd \calO_K$. Here, $\calS^\diamondsuit$ is the functor associating to each $T \in \Perf$ a pair $(T^\sharp, f)$, where $T^\sharp$ is an untilt of $T$ and $f\colon T^\sharp \to \calS$ is a morphism of adic spaces over $\Spa \Z_p$. The reduced special fiber $\calS_0$ of $\calS$ induces a subsheaf $\calS_0^\diamondsuit$ of $\calS^\diamondsuit$. By definition, $\calS_0^\diamondsuit=(\calS_0^{-p^\infty})^{\diamondsuit}$, where $\calS_0^{-p^\infty}$ is the perfection of $\calS_0$. We refer to \cite[App.~A]{zhuaffinegrassmanniansandsatake} for the theory of perfect schemes. 

Recall that we have the (perfect) prismatic sites $(\calS_0)^{(\perf)}_\Prism$ and $\calS^{(\perf)}_\Prism$. By a finite locally free $F$-isocrystal over $(\calS_0)_\Prism$, we mean a pair $(\calE, \varphi_{\calE})$, where $\calE$ is a functor associating to each $(A,(p))$ (note that since $\calS_0$ is in characteristic $p$, $I=(p)$ for any $(A,I)\in (\calS_0)_\Prism$) an $A$-module $\calE_A$ such that $\calE_A[1/p]$ is a finite locally free $A[1/p]$-module, satisfies the crystal property, and $\varphi_{\calE}\colon \varphi^\ast\calE_A[1/p] \xrightarrow{\cong} \calE_A[1/p]$. There is a natural equivalence between the category of $F$-isocrystals over $(\calS_0)_\Prism$ and the category of $F$-isocrystals over $(\calS_0)_{\CRIS}$ (cf.\ \cite{GuoReinecke-Ccris}). When $\calS$ admits a log structure, we will recall the fully faithful functor recalled in Proposition~\ref{prop: crystalline realization}.

In this context, we are primarily interested in $F$-isocrystals over $(\calS_0)^{\perf}_\Prism$. One sees immediately from the definition that $(\calS_0)^{\perf}_\Prism \cong (\calS_0^{-p^\infty})^{\perf}_\Prism$. Moreover, the assignment $(A,(p)) \mapsto A/p$ induces an isomorphism $(\calS_0^{-p^\infty})^{\perf}_\Prism \xrightarrow{\cong} (\calS_0^{-p^\infty})^{\perf}$, where $(\calS_0^{-p^\infty})^{\perf}$ denotes the (big) perfect site equipped with the \emph{flat} topology. 
For a finite locally free $F$-isocrystal $(\calE, \varphi_{\calE})$ over $(\calS_0)^{\perf}_\Prism$ and any point $x \in |\calS_0|$, let $\kappa(x)$ be the residue field at $x$ with $x\colon \Spec(\kappa(x)) \to \calS_0$ denoting the corresponding map. The restriction of $(\calE, \varphi_{\calE})$ to $\kappa(x)^{\perf}_\Prism$ along $x$ naturally defines an $F$-isocrystal $(\calE_x, \varphi_{\calE_x})$ over $\kappa(x)^{-p^\infty}$. 
\end{construction}

\begin{defn}[{cf.~\cite[Defn.~2.31]{GuoReinecke-Ccris} and \cite[Defn.~3.39]{du-liu-moon-shimizu-purity-F-crystal}}]\label{defn: log crystalline loc sys}
For a $\Q_p$-local system $\bL$ on $S_{\proet}$, we say that $\bL$ is $(\calS,M_{\calS})$-crystalline (or log-crystalline with respect to $(\calS,M_{\calS})$) if there is an $F$-isocrystal $\calE_\Q$ on $(\calS_1,M_{\calS_1})_{\CRIS}$ associated to $\bL$ in the sense that there exists a Frobenius-equivariant isomorphism of $\bB_\cris$-modules on $S_{\proet}$
\[
\alpha_{\cris}\colon \bB_{\cris}(\calE_\Q)\cong \bL\otimes_{\Q_p}\bB_\cris.
\]
For a $\Z_p$-local system $\bL$ on $S_{\proet}$, we say $\bL$ is crystalline if $\bL \otimes_{\Z_p} \Q_p$ is crystalline.
\end{defn}

One can show that for a log-crystalline $\Q_p$-local system $\bL$ on $S_{\proet}$, the restriction $\bL|_{\overline{s}}$ is a log-crystalline Galois representation at all classical points $s \in S$. The converse also holds by the following result.

\begin{thm}[{\cite[Thm.~1.1]{GuoYangpointwise}}]\label{thm: pointwise}
Let $S$ be a smooth rigid analytic space over $K$. If $S$ admits a semistable formal model, then an isogeny $\Z_p$-local system $\bL$ is log-crystalline in the sense of Definition~\ref{defn: log crystalline loc sys} if and only if $\bL|_{\overline{s}}$ is a log-crystalline Galois representation at all classical points $s \in S$.
\end{thm}

\medskip
\noindent
\textbf{Prismatic classification of log-crystalline \texorpdfstring{$\Z_p$}{Zp}-local systems.} 

We now review the prismatic classification of log-crystalline local systems, drawing from \cite{du-liu-moon-shimizu-completed-prismatic-F-crystal-loc-system}, \cite{GuoReinecke-Ccris}, and \cite{du-liu-moon-shimizu-purity-F-crystal}. The key definition is as follows.

\begin{defn}
Let $(\calS,M_{\calS})_\Prism$ denote the (strict) absolute log prismatic site. The objects $(A,I,M_{\Spf A})$ of $(\calS,M_{\calS})_\Prism$ are log prisms $(A,I,M_{\Spf A})$ equipped with a strict morphism $(\Spf(A/I), M_{\Spf A/I}) \to (\calS,M_{\calS})$ of log $p$-adic formal schemes, where $M_{\Spf A/I}$ is the log structure induced by $M_{\Spf A}$. We will omit the log structures from the notation when $M_{\calS}$ is trivial. 

We endow $(\calS,M_{\calS})_\Prism$ with the \emph{strict flat} topology. The $\delta$-structure on $A$ induces a Frobenius $\varphi_A$ on $A$. Let $\calO_\Prism$ (resp.\ $\calI_\Prism$) be the structure sheaf (resp.\ the ideal sheaf of the Hodge--Tate divisor) on $(\calS,M_{\calS})_\Prism$, defined by associating to $(\Spf A,I,M_{\Spf A})$ the ring $A$ (resp.\ the ideal $I$). 

Let $(\calS,M_{\calS})_\Prism^{\perf}$ (resp.\ $(\calS,M_{\calS})_\Prism^{\perf,\circ}$) denote the full subcategory consisting of perfect prisms (resp.\ perfect prisms $(A,I)$ with $A/I$ being $p$-torsion free).
\end{defn}

We refer the reader to \cite{BhattScholzePrism} for the general theory of prisms, and to \cite[\S2.1]{du-liu-moon-shimizu-purity-F-crystal} for the precise details regarding the strict log-prismatic site.

\begin{rem}\label{rem: perfect site log vs non log}
By \cite[Prop.~2.18]{min-wang-HT-crys-log-prism}, when $(\calS,M_{\calS})$ satisfies the conditions established at the beginning of this subsection, the forgetful functor induces an equivalence $(\calS,M_{\calS})_\Prism^{\perf}\xrightarrow{\cong} \calS_\Prism^{\perf}$. In other words, for every perfect prism $(A,I)$ inside $\calS_\Prism$, there is a unique log structure $M_{\Spf A}$ on $\Spf A$ such that $(A,I,M_{\Spf A}) \in (\calS,M_{\calS})_\Prism$.
\end{rem}

\begin{defn}[{\cite[Defn.~3.3]{du-liu-moon-shimizu-purity-F-crystal}}]\label{defn: analytic F crystal} 
For any $(A,I,M_{\Spf A}) \in (\calS,M_{\calS})_\Prism$, we define the following categories:
\begin{enumerate}
    \item Let $\Vect^{\mathrm{an}}(A,I)$ be the category of vector bundles over $\Spec(A)^{\mathrm{an}}\coloneqq \Spec(A)\smallsetminus V(p,I)$. We refer to $\Spec(A)^{\mathrm{an}}$ as the \emph{analytic locus} of $\Spec(A)$.
    \item Let $\Vect^{\varphi}(A,I)$ be the category of pairs $(\mathfrak{M}_A,\varphi_{\mathfrak{M}_A})$ where $\mathfrak{M}_A$ is a finite locally free $A$-module, and $\varphi_{\mathfrak{M}_A}$ is an isomorphism of $A[I^{-1}]$-modules
\[
    \varphi_{\mathfrak{M}_A}\colon \varphi_A^\ast(\mathfrak{M}_A)[I^{-1}] \simeq \mathfrak{M}_A[I^{-1}].
\]

    \item Let $\Vect^{\mathrm{an},\varphi}(A,I)$ be the category of pairs $(\mathcal{E}_A, \varphi_{\mathcal{E}_A})$ where $\mathcal{E}_A$ is a vector bundle over $\Spec(A)\smallsetminus V(p,I)$ and $\varphi_{\mathcal{E}_A}$ is an isomorphism of vector bundles
\[
    \varphi_{\mathcal{E}_A}\colon \varphi_A^\ast(\mathcal{E}_A)[I^{-1}] \simeq \mathcal{E}_A[I^{-1}].
\]
Such a pair $(\mathcal{E}_A,\varphi_{\mathcal{E}_A})$ is called \emph{effective} if $\varphi_{\mathcal{E}_A}$ is induced by a genuine morphism $\varphi_A^\ast(\mathcal{E}_A) \to \mathcal{E}_A$ of vector bundles over $\Spec(A)\smallsetminus V(p,I)$. 
    \item We define the category $\Vect^{\mathrm{an},\varphi}((\calS,M_{\calS})_\Prism)$ of \emph{analytic prismatic $F$-crystals} over $(\calS,M_{\calS})$ by
\[
\Vect^{\mathrm{an},\varphi}((\calS,M_{\calS})_\Prism)\coloneqq\lim_{(\Spf A,I,M_{\Spf A})\in (\calS,M_{\calS})_\Prism} \Vect^{\mathrm{an},\varphi}(A,I).
\]
An analytic prismatic $F$-crystal is typically denoted by $(\calE_\Prism,\varphi_{\calE_\Prism})$ or simply $\calE_\Prism$, and we write $(\calE_{\Prism,A},\varphi_{\calE_{\Prism,A}})$ for the associated object in $\Vect^{\mathrm{an},\varphi}(A,I)$.
    \item Let $\Vect^{\mathrm{an},\varphi}((\calS,M_{\calS})_\Prism)_{\Q_p}$ denote the $p$-isogeny category of $\Vect^{\mathrm{an},\varphi}((\calS,M_{\calS})_\Prism)$.
\end{enumerate}
\end{defn}

The following is the principal result of \cite{du-liu-moon-shimizu-purity-F-crystal}, providing the prismatic classification of log-crystalline local systems.

\begin{thm}[{\cite[Prop.~3.20 and Cor.~5.2]{du-liu-moon-shimizu-purity-F-crystal}}]\label{thm: prismatic classification}
The \'etale realization functor induces an equivalence
\[
T_{\et} \colon\Vect^{\mathrm{an},\varphi}((\calS,M_{\calS})_\Prism)\xrightarrow{\cong} \Loc_{\Z_p}^{\cris}(\calS,M_{\calS}).
\]
Moreover, if $T$ is the image of an analytic $F$-crystal $\calE_\Prism$ under $T_{\et}$, then $T$ is associated with the crystalline realization $\calE_{\cris}$ of $\calE_\Prism$ in the sense of Definition~\ref{defn: log crystalline loc sys}. 
\end{thm}
We will briefly recall the definition of crystalline realization in Proposition~\ref{prop: crystalline realization}.

Log-crystalline local systems are necessarily de Rham, cf.\ \cite[Cor.~3.46]{du-liu-moon-shimizu-purity-F-crystal}. In the remainder of this subsection, we review some essential properties of the \'etale realization functor and the crystalline realization functor mentioned in Theorem~\ref{thm: prismatic classification}.

\begin{defn} \label{defn: Laurent-F-crystals}
A \emph{Laurent $F$-crystal} on $(\calS,M_{\calS})_\Prism$ consists of a pair $(\mathscr{E},\varphi_{\mathscr{E}})$ where $\mathscr{E}$ is a crystal of vector bundles on the ringed site $((\calS,M_{\calS})_\Prism, \calO_\Prism[\calI_\Prism^{-1}]^\wedge_p)$ and $\varphi_{\mathscr{E}}$ is an isomorphism $\varphi_{\mathscr{E}}\colon \varphi^\ast {\mathscr{E}}\xrightarrow{\cong}{\mathscr{E}}$.
We let $\Vect((\calS,M_{\calS})_\Prism,\calO_\Prism[\calI_\Prism^{-1}]^\wedge_p)^{\varphi=1}$ denote the category of Laurent $F$-crystals on $(\calS,M_{\calS})_\Prism$.
\end{defn}

\begin{thmdefn}[{\cite[Thm.~3.14]{du-liu-moon-shimizu-purity-F-crystal}}] \label{thm: etale realization of Fcrystals}
Assume that $(\calS,M_{\calS})$ is as described at the beginning of this subsection. There are natural equivalences of categories
\[
\Vect((\calS,M_{\calS})_\Prism,\calO_\Prism[\calI_\Prism^{-1}]^\wedge_p)^{\varphi=1} \cong \Vect((\calS,M_{\calS})_\Prism^\perf,\calO_\Prism[\calI_\Prism^{-1}]^\wedge_p)^{\varphi=1} \cong \Loc_{\Z_p}(S).
\]
The \'etale realization functor is defined as the composite
\[
T_\et\colon \Vect^{\an,\varphi}((\calS,M_{\calS})_\Prism)\rightarrow \Vect((\calS,M_{\calS})_\Prism,\calO_\Prism[\calI_\Prism^{-1}]^\wedge_p)^{\varphi=1}\xrightarrow{\cong} \Loc_{\Z_p}(S),
\]
where the first arrow is given by base change along $\calO_\Prism \to \calO_\Prism[\calI_\Prism^{-1}]^\wedge_p$.
\end{thmdefn}

Note that $\Vect((\calS,M_{\calS})_\Prism^\perf,\calO_\Prism[\calI_\Prism^{-1}]^\wedge_p)^{\varphi=1} \cong \Vect(\calS_\Prism^\perf,\calO_\Prism[\calI_\Prism^{-1}]^\wedge_p)^{\varphi=1}$ by Remark~\ref{rem: perfect site log vs non log}. Consequently, the \'etale realization factors through the restriction to the perfect prismatic site $\calS_\Prism^\perf$. 

\begin{cor}
Let $(\calS, M_{\calS})$ be as defined at the beginning of this subsection. Restriction to the transversal perfect site $\calS^{\perf,\circ}_\Prism$ over $\calS$ defines a fully faithful functor:
\[
\Vect^{\an,\varphi}((\calS,M_{\calS})_\Prism)\rightarrow \Vect^{\an,\varphi}(\calS^{\perf,\circ}_\Prism).
\]
\end{cor}
\begin{proof}
This follows because the \'etale realization functor factors through this restriction functor, cf. \cite[Prop.~3.20 and Thm.~3.14]{du-liu-moon-shimizu-purity-F-crystal}.
\end{proof}

For the crystalline realization, we require the following description based on \cite[Ex.~4.7, Const.~4.12]{bhatt-scholze-prismaticFcrystal}, \cite[Const.~3.9, Thm.~6.4]{GuoReinecke-Ccris}, and \cite[Prop.~3.36]{du-liu-moon-shimizu-purity-F-crystal}.

Let $i_{\calS}\colon(\calS_1,M_{\calS_1}) \to (\calS,M_{\calS})$ denote the exact closed immersion of the modulo $p$ fiber. Note that every prism in $(\calS_1,M_{\calS_1})_\Prism$ is crystalline (i.e., its ideal is generated by $p$). In fact, this yields a fully faithful functor
\[
D_{\cris,\calS_1}\colon\Vect^{\an,\varphi}((\calS_1,M_{\calS_1})_\Prism) \xrightarrow{}  \Vect_\Q^\varphi((\calS_1,M_{\calS_1})_\CRIS)
\]
where $\Vect_\Q^\varphi((\calS_1,M_{\calS_1})_\CRIS)$ is the category of finite locally free $F$-isocrystals over $(\calS_1,M_{\calS_1})_\CRIS$ in the sense of \cite[Defn.~B.19]{du-liu-moon-shimizu-purity-F-crystal}. Moreover, this fully faithful functor is an equivalence when $M_{\calS}$ is trivial \cite[Thm~6.4(1)]{GuoReinecke-Ccris}; full faithfulness in the semistable case is given by \cite[Prop.~3.35(2)]{du-liu-moon-shimizu-purity-F-crystal}.

We only require the following description of the crystalline realization. 
\begin{prop}[{\cite[Prop.~3.36]{du-liu-moon-shimizu-purity-F-crystal}}]\label{prop: crystalline realization}
The crystalline realization functor $D_{\cris}$ fits into the following $2$-commutative diagram:
\begin{equation}\label{eq: crystalline realization-summary}
\xymatrix{
\Vect^{\an,\varphi}((\calS,M_{\calS})_\Prism)\ar[rd]_-{D_\cris}\ar[r]^-{i_{\calS,\Prism}^{-1}}
& \Vect^{\an,\varphi}((\calS_1,M_{\calS_1})_\Prism)\ar@{^{(}->}[d]^-{D_{\cris,\calS_1}}\\
& \Vect_\Q^\varphi((\calS_1,M_{\calS_1})_\CRIS),
}
\end{equation}
where $i_{\calS,\Prism}^{-1}$ is the pullback functor discussed in \cite[Defn~C.7]{du-liu-moon-shimizu-purity-F-crystal}.
\end{prop}

\begin{rem}
Consider the exact closed immersion of the reduced special fiber $(\calS_0,M_{\calS_0})$ of $(\calS,M_{\calS})$ into the mod $p$ fiber $(\calS_1,M_{\calS_1})$. The pullback along this immersion induces an equivalence of $F$-isocrystals. That is, we have
\[
\Vect_\Q^\varphi((\calS_1,M_{\calS_1})_\CRIS) \cong \Vect_\Q^\varphi((\calS_0,M_{\calS_0})_\CRIS).
\]
See \cite[Rem.~B.20]{du-liu-moon-shimizu-purity-F-crystal}.
\end{rem}

For our later applications, we will require the following result, which is essentially proven in \cite{du-liu-moon-shimizu-purity-F-crystal}.

\begin{propdefn}\label{propdefn: lattice realization}
For any $\calE_\Prism \in \Vect^{\an,\varphi}((\calS,M_{\calS})_\Prism)$ with crystalline realization $\calE_\Q\coloneqq \calE_\cris$, the assignment
\[
U \mapsto \bB_\cris^+(\calE_\Q)(U)\otimes_{\bB_\cris^+(U)}\bB_{\dR}^+(U)
\]
for $U\in \mathrm{Perfd}/S_{\proet}$ defines a $\bB_{\dR}^+$-local system $\calE_\Q(\bB_{\dR}^+)$ on $S_{\proet}$ (recall $\bB_\cris^+(\calE_\Q)$ is defined at the beginning of this subsection). Let $\bL$ be the \'etale realization of $\calE_\Prism$. Then $\bL$ is a de Rham local system, and there is a canonical isomorphism
\begin{equation}\label{eq: M0 for cris}
	\calE_\Q(\bB_{\dR}^+) \cong \bM_0(\bL)
\end{equation}
of $\bB_{\dR}^+$-local systems on $S_{\proet}$, where $\bM_0(\bL)$ is defined as in Remark~\ref{rem: M0 and M_1 for de Rham as proet vb} by treating $\bL$ as a de Rham local system. 

Moreover, for any $U \in \Perfd/S_{\proet}$ associated with an affinoid perfectoid space $T=\Spa(A,A^+)$, since $\calS$ is normal, $T$ naturally defines a perfect prism $(B,I)$ in $\calS_{\Prism}$ via \cite[Thm.~3.10]{BhattScholzePrism}, with $B=W(A^{+,\flat})$. One has an explicit isomorphism
\[
	\calE_\Q(\bB_{\dR}^+)(U) \cong \calE_\Prism(B, I)[1/p]\otimes_{B[1/p],\varphi}\bB_{\dR}^+(U).
\]
\end{propdefn}

\begin{proof}
The first isomorphism \eqref{eq: M0 for cris} follows by taking the horizontal sections in \cite[Prop.~3.44\,(1)]{du-liu-moon-shimizu-purity-F-crystal} for $\ast = +$. To see this, observe that the left-hand side of Eq.~\eqref{eq: M0 for cris} precisely matches the construction of $\calE_{\mathbf{Q}}(\bB_{\dR}^+)$ from \cite[Const.~3.43 and Prop.~3.45]{du-liu-moon-shimizu-purity-F-crystal} (where $\calE_{\mathbf{Q}}$ corresponds to the crystalline realization $\calE_\Q$ of $\calE_\Prism$ in our notation). 

The remaining explicit formula follows from the characterization of the crystalline realization functor provided by \cite[Cor.~3.31]{du-liu-moon-shimizu-purity-F-crystal}. The Frobenius twist $\varphi$ in the tensor product over $B[1/p]$ arises because the map $\mathfrak{S} \to S$ in \cite[Const.~3.29]{du-liu-moon-shimizu-purity-F-crystal} is induced by the absolute Frobenius $\varphi$, which canonically lifts to the composite $\bA_{\inf} \xrightarrow{\varphi} \bB_{\cris}^+ \to \bB_{\dR}^+$.
\end{proof}

\section{Diamantine Fargues--Fontaine curve and relative shtukas}
\label{sec: diamantine FF curves}

In this section, we review basic facts about the adic (relative) Fargues--Fontaine curve, and then introduce its diamantine generalization and the theory of vector bundles over it. Finally, we review the theory of relative shtukas defined over an analytic adic space, motivated by \cite{Pappas-Rapoport-padicsht}, and relate them to admissible modifications of vector bundles over the diamantine Fargues--Fontaine curve. 

Throughout this section, we work over a base $S$ assumed to be a locally spatial diamond. We reserve the letter $X$ for the Fargues--Fontaine curve.

\subsection{Local systems and vector bundles in different topologies}\label{sec: v loc sys and vb}
In this subsection, for $\tau \in \{\mathrm{an}, \et, \proet, \qproet, \vv\}$, we let $S_\tau$ denote the corresponding site endowed with the $\tau$-topology, cf. \cite[\S1]{MannWerner-Localsystemsondiamonds}. We will recall some foundational facts about local systems on $S_\tau$, cf. \cite[\S10.4]{SWBerkeleyNotes}. We fix a morphism $S \to \Spd \Z_p$, and recall the notation for vector bundles in the $\tau$-topology.

First, we recall the definition of the diamond associated with an analytic adic space.

\begin{defn}[{\cite[Lem.~15.1 \& Defn.~15.5]{scholze-etalecohomologyofdiamonds}}]
Let $S$ be an analytic adic space over $\Spa \mathbb{Z}_p$. 
We define a presheaf $S^\diamondsuit$ on $\mathrm{Perf}_{\vv}$ as follows: For any perfectoid space $T\in \mathrm{Perf}$ of characteristic $p$, let $S^\diamondsuit(T)$ be the set of isomorphism classes of triples $(T^\sharp, \iota, T^\sharp \to S)$, where $(T^\sharp, \iota)$ is an untilt over $\Z_p$, and $T^\sharp \to S$ is a morphism of pre-adic spaces. We also write $\Spd S$ for $S^\diamondsuit$, and we write $\Spd(R, R^+) = \Spa(R, R^+)^\diamondsuit$ (or simply $\Spd R$) when $S=\Spa(R, R^+)$ is an affinoid. 
\end{defn}

\begin{rem}\label{rem:equiv of PerfSdiam and PerfdS}
Assume $S$ is an analytic rigid space equipped with a structure morphism to $\Spa \Q_p$. Then $S^\diamondsuit$ admits a natural structure morphism to $\Spd \Q_p$. Let $\Perf_{S^\diamondsuit}$ be the category of pairs $(T,f_T)$ where $T \in \Perf$ and $f_T\colon T \to S^\diamondsuit$, and let $\Perfd_{S}$ be the category of pairs $(T',f_{T'})$ consisting of perfectoid spaces over $S$. For $T \in \Perf$, a morphism $f \in S^\diamondsuit(T)$ naturally yields a pair $(T^\sharp, T^\sharp\to S)$ via the definition of $S^\diamondsuit$. Conversely, given $(T', T' \to S)$, taking the associated diamonds defines $(T'^\flat, T'^\flat \to S^\diamondsuit)$. Using the tilting equivalence, it is straightforward to show that this yields an equivalence of categories $\Perfd_{S} \cong \Perf_{S^\diamondsuit}.$
\end{rem}

We recall the following facts relating $S$ to $S^{\diamondsuit}$.

\begin{thm}[{\cite[Prop.~10.2.3, 10.3.7, Thm.~10.4.2]{SWBerkeleyNotes}}]
Let $S$ be as above. Then:
\begin{enumerate}
	\item $S^\diamondsuit$ is a diamond;
	\item There is a natural homeomorphism of topological spaces $\lvert S\rvert \cong \lvert S^\diamondsuit \rvert$;
	\item The functor $T \mapsto T^\diamondsuit$ induces an equivalence of sites $S_\tau \cong S^\diamondsuit_\tau$ for $\tau \in \{\et, \mathrm{f}\et\}$;
	\item The functor $S \mapsto S^\diamondsuit$ defines a fully faithful functor from the category of smooth rigid analytic spaces over $K$ to the category of diamonds over $\Spd K$.
\end{enumerate}
\end{thm}

The definitions of isogeny $\Z_p$- and $\Q_p$-local systems over $S_\tau$ for a diamond $S$ are discussed in \cite[\S3]{MannWerner-Localsystemsondiamonds}. We will need the following results regarding local systems on $S_\tau$ when $S$ is a smooth rigid analytic variety over $K$. The primary references for this material are Scholze \cite{scholze-etalecohomologyofdiamonds}, \cite[\S3]{MannWerner-Localsystemsondiamonds}, and Kedlaya--Liu \cite{kedlaya-liu-relative-padichodge}\footnote{Note that in \cite{kedlaya-liu-relative-padichodge}, the results hold for (pre)adic spaces, which are locally defined by adic Banach rings $(A,A^+)$ where $A$ contains a topologically nilpotent unit. Since any analytic rigid space is covered by the adic spectrum of rings satisfying these conditions, we are free to apply their results.}.

\begin{thm}\label{thm: equiv of all kinds of locsys} 
\begin{enumerate}
\item The following categories are equivalent:
    \begin{enumerate}
    \item The category of $\Z_p$-local systems over $S_{\ast}$ for $\ast \in \{\et,\proet\}$;
    \item The category of $\Z_p$-local systems over $S^\diamondsuit_{\ast}$ for $\ast \in \{\et,\qproet,\vv\}$.
    \end{enumerate}
\item A similar result holds for $\Q_p$-local systems; the following categories are equivalent:
    \begin{enumerate}
    \item The category of $\Q_p$-local systems over $S_{\ast}$ for $\ast \in \{\et,\proet\}$;
    \item The category of $\Q_p$-local systems over $S^\diamondsuit_{\ast}$ for $\ast \in \{\et,\qproet,\vv\}$.
    \end{enumerate}
\end{enumerate}
Moreover, for any $\Q_p$-local system $\mathbb{L}$ over $S^\diamondsuit_{\vv}$ and any point $s\in \lvert S \rvert$, the following statements hold: 
\begin{enumerate}
\item[(3a)] There exists an \'etale covering $T \to S$ such that $\mathbb{L}|_{T^\diamondsuit}$ is an isogeny $\Z_p$-local system.\footnote{The author thanks Laurent Fargues for pointing out this result.} Moreover, one can choose $T \to S$ to be a de Jong covering in the sense of \cite[Defn.~2.1]{dJfundamentalgroup}, with its adic space version \cite[Defn.~5.2.2]{ALYArcs}. If $S$ is connected, then $T$ may also be chosen to be connected.
\item[(3b)] There exists an open subfunctor $U \subset S^\diamondsuit$, where $\lvert U \rvert$ is an open subset of $\lvert S \rvert$ containing $s$, such that $\mathbb{L}|_{U}$ is an isogeny $\Z_p$-local system.
\end{enumerate}
\end{thm}
\begin{proof}
The equivalences in $(1a)$\footnote{The case where $S$ is a locally Noetherian adic space follows from \cite[Prop.~8.2]{scholze-p-adic-hodge}.} and $(1b)$ follow from \cite[Lem.~9.1.11]{kedlaya-liu-relative-padichodge}. We also invoke \cite[Cor.~3.15]{MannWerner-Localsystemsondiamonds} to ensure that \cite[Defn.~3.3]{MannWerner-Localsystemsondiamonds} agrees with the standard definitions (cf.\ \cite[Defn.~8.1]{scholze-p-adic-hodge}). The categories in $(1a)$ (resp.\ $(2a)$) are equivalent to the corresponding categories in $(1b)$ (resp.\ $(2b)$) equipped with the \'etale topology because $S_{\et} \cong S^\diamondsuit_{\et}$, cf.\ \cite[Lem.~15.6]{scholze-etalecohomologyofdiamonds}. The categories in $(1b)$ are mutually equivalent by combining \cite[Props.~3.5, 3.7, \& 3.9]{MannWerner-Localsystemsondiamonds}. Statement $(3b)$ is \cite[Prop.~8.4.6]{kedlaya-liu-relative-padichodge}, and statement $(3a)$ is \cite[Lem.~3.10]{MannWerner-Localsystemsondiamonds}. 

The remaining parts of $(3a)$ follow from Theorem~\ref{thm: generalized de Jong lattice sheaf} and (1) in Remark~\ref{rem: relation with lattice cover} that the sheaf representing lattices in $\bbL$ is represented by a de Jong covering. A de Jong covering is a geometric covering, and one can use properties of geometric coverings to show the rest, cf. Proposition~\ref{prop: geometric covering-input}.

To prove that the categories in $(2b)$ are mutually equivalent, we first show that the pullback functor defines a fully faithful functor $\Loc_{\Q_p}(S^\diamondsuit_{\et}) \to \Loc_{\Q_p}(S^\diamondsuit_{\vv})$. By the definition of $\Q_p$-local systems on $S^\diamondsuit_{\et}$, it is sufficient to prove fully faithfulness when $S^\diamondsuit=\Spd(A,A^+)$. Moreover, we may pass to a further rational localization on which the local system admit $\Z_p$-lattices, cf. \cite[Cor.~8.4.7]{kedlaya-liu-relative-padichodge}, and fully faithfulness follows directly from the $\Z_p$-local system case. Essential surjectivity follows from \cite[Thm.~3.11]{MannWerner-Localsystemsondiamonds} and its proof: for any $\Q_p$-local system $\mathbb{L}$ on $S^\diamondsuit_\vv$, statement $(3a)$ guarantees an \'etale covering $U \to S^\diamondsuit$ such that $\mathbb{L}|_U$ is an isogeny $\Z_p$-local system over $U^{\diamondsuit}_\vv$. By $(1)$, this isogeny $\Z_p$-local system descends to $S^\diamondsuit_{\et}$. Fully faithfulness then implies that the descent isomorphism also descends, thereby producing a $\Q_p$-local system on $S^{\diamondsuit}_{\et}$ by definition. Applying \cite[Thm.~3.11]{MannWerner-Localsystemsondiamonds} again yields $(2b)$.
\end{proof}

\begin{rem}
By the theorem above, in the remainder of the paper, we will omit the hat and write $\bL$ for the object denoted by $\hat{\bL}$ in \S\ref{sec: arith loc sys}.
\end{rem}

If $S$ is a smooth rigid analytic variety over $K$, we have the natural structure sheaves $\calO_\tau\coloneqq \calO_S$ on $S_{\tau}$ for $\tau \in \{\mathrm{an}, \et\}$. More generally, once we fix a morphism $S\to \Spd \Z_p$, we can define the ``structure sheaf'' $\calO_\tau\coloneqq\widehat{\calO}_S$ on $S_{\tau}$ for $\tau \in \{\qproet, \vv\}$. This is defined by assigning to each $\Spa(R,R^+) \in S_{\tau}$ the ring $R^\sharp$, where $(R^\sharp,R^{\sharp,+})$ is the untilt of $(R,R^+)$ induced by the composition $\Spa(R,R^+) \to S \to \Spd \Z_p$.

Once the structure sheaves $\calO_\tau$ are defined, one can formalize the category of vector bundles over $S_\tau$\footnote{When $\tau =\vv$, we must specify a morphism $S \to \Spd \Z_p$ to define $\calO_\tau$ as discussed. We will omit this specification when the morphism is clear from the context.}. For the precise definition, we refer the reader to \cite[\S2.1]{Benvlinebundle}. We recall the following result concerning vector bundles over $S_\tau$.

\begin{prop}\label{prop: tau vector bundles}
If $S$ is a smooth rigid analytic space over $K$, then the categories of vector bundles over $S^{(\diamondsuit)}_{\tau}$ are equivalent for $\tau\in \{\proet,\qproet,\vv\}$.
\end{prop}
\begin{proof}
This is proven in \cite[\S1.1]{Benvlinebundle}.
\end{proof}

\begin{rem}\label{rem: extend BdR and HT to v loc sys}
Using Proposition~\ref{prop: tau vector bundles} and Theorem~\ref{thm: equiv of all kinds of locsys}, certain $p$-adic Hodge properties of local systems $\bL$ over $S_{\et}$ (discussed in \S\ref{sec: arith loc sys}) can be verified $\vv$-locally. 
\begin{enumerate}
	\item One can define potentially unramified local systems and Hodge--Tate local systems of weight $0$ on a smooth rigid analytic space purely using the $\vv$-site. To do so, one simply replaces the pro-\'etale site with the $\vv$-site in Proposition~\ref{prop: Condition for L to be HT wt 0} and Theorem~\ref{thm: RpMT for HT0}.
	\item The notion of $\mathbb{B}_{\text{dR}}^{+}$-local systems over $S_{\proet}$ in the sense of \cite[Defn.~7.1]{scholze-p-adic-hodge} naturally extends to $(S^\diamondsuit/\Spd K)_\vv$ via Proposition~\ref{prop: tau vector bundles}, and the fact that $\mathbb{B}_{\text{dR}}^{+}$ satisfies $\vv$-descent.
\end{enumerate}
\end{rem}

\subsection{The Fargues--Fontaine curve}

Recall that $\Perf$ denotes the category of perfectoid spaces over $\F_p$. We fix a complete discrete valuation ring $\calO_E$ of mixed characteristic $(0,p)$, with a finite residue field $k$ of cardinality $q\coloneqq p^h$. In later sections, we will be primarily interested in the case where $\calO_E=\Z_p$. Let $E$ be the fraction field of $\calO_E$, and let $\pi$ denote a uniformizer of $\calO_E$.

\begin{defn}
For any affinoid perfectoid space $S=\Spa(R,R^+)$ in $\Perf$, and any pseudouniformizer $\varpi \in R^+$, we define:
\begin{itemize}
	\item $\calY_S \coloneqq \Spa( W_{\calO_E}(R^+))\setminus V([\varpi])$;
	\item $Y_S \coloneqq \Spa( W_{\calO_E}(R^+))\setminus V(\pi[\varpi])$.
\end{itemize}
\end{defn}

Note that once $\varpi$ and $\pi$ are fixed as above, there exists a ``relative position function'' 
\[
\kappa\colon \lvert\Spa( W_{\calO_E}(R^+))^a\rvert \to [0,\infty]
\]
defined by $\kappa(x)\coloneqq \ln|[\varpi](x_{\max})|/ \ln |\pi(x_{\max})|$. Here, $\Spa( W_{\calO_E}(R^+))^a$ denotes the analytic locus of $\Spa( W_{\calO_E}(R^+))$, which is strictly $\Spa( W_{\calO_E}(R^+))\setminus (V(\pi) \cap V([\varpi]))$. For any point $x$ in an analytic adic space, $x_{\max}$ represents its maximal generalization (cf.\ \cite[Prop.~4.2.5]{SWBerkeleyNotes}). Because $x_{\max}$ is of rank $1$, the terms $\ln|[\varpi](x_{\max})|$ and $\ln |\pi(x_{\max})|$ are well-defined and take values in $[-\infty,0)$, as $[\varpi]$ and $\pi$ are topologically nilpotent. The assumption that $x$ lies in the analytic locus ensures that $\ln|[\varpi](x_{\max})|$ and $\ln |\pi(x_{\max})|$ cannot simultaneously be $-\infty$. The function $\kappa$ is continuous (cf.\ \cite[Prop.~II.1.16]{FarguesScholze}).

For any affinoid perfectoid space $S=\Spa(R,R^+)$ in $\Perf$, fix a pseudouniformizer $\varpi \in R^+$. For any interval $I=[a,b]\subset [0,\infty)$ with rational endpoints (allowing $a=b$), define
\[
Y_{S,I}=\{\lvert\pi\rvert^b\le \lvert[\varpi]\rvert\le \lvert\pi\rvert^a\} \subset \kappa^{-1}(I) \subset Y_S \quad \text{if $a>0$,}
\]
\[
Y_{S,[0,b]}=\{\lvert\pi\rvert^b\le \lvert[\varpi]\rvert \neq 0\} \quad \text{when $b>0$,}.
\]
One can also defines $Y_{S,J}=\cup_{I\subset J} Y_{S,I}$ for intervals $J=(a,b)$, $[a,b)$, or $(a,b]$ with rational endpoints inside $(0,\infty)$, and $I$ is compact with rational endpoints. 

One constructs $X_S$ as the quotient $Y_{S,[a,b]}/\varphi^{\Z}$, where $a, b$ are positive real numbers with $a > 0$, $b/a \geq q$, note that here we use the identification $\varphi\colon Y_{S,[a,a]}\xrightarrow{\sim} Y_{S,[qa,qa]}$. In particular, $X_S$ is qcqs when $S$ is affinoid, cf. \cite[Prop.~II.1.16]{FarguesScholze}.

For an affinoid perfectoid $S=\Spa(A,A^+)\in \Perf$, we define the following variants of Robba rings following \cite{kedlaya-liu-relative-padichodge}.
\begin{defn}\label{defn: Robba rings} 
For any rational number $r >0$, and a fixed pseudouniformizer $\varpi \in A^+$, we define the following rings:
\begin{enumerate}
	\item $\widetilde{\calR}_S^{\rmint, r} \coloneqq H^0(Y_{S,[0,r]},\calO)$;
	\item $\widetilde{\calR}_S^{\rmint}\coloneqq \colim_{r>0} \widetilde{\calR}_S^{\rmint, r}$, where the transition maps are given by rational localizations;
	\item $\widetilde{\calR}_S^{\mathrm{bd}}\coloneqq \colim_{r>0}\widetilde{\calR}_S^{\rmint, r}[1/\pi]$;
	\item $\widetilde{\calR}_S^{r} \coloneqq H^0(Y_{S,(0,r]},\calO)$, and $\widetilde{\calR}_S\coloneqq \colim_{r} \widetilde{\calR}_S^{r}$;
	\item $\widetilde{\calE}_S\coloneqq W_{\calO_E}(A)[1/\pi]$.
\end{enumerate}
\end{defn}

\begin{rem}
We also recall the following facts from \cite[Chap.~II, \S1.1 \& \S1.2]{FarguesScholze}. The definitions of $\calY_S$ and $Y_S$ do not depend on the choice of $\varpi$ and $\pi$. The $q$-th power Frobenius induces an action $\varphi$ on both $\calY_S$ and $Y_S$. Moreover, $\calY_S$ and $Y_S$ both carry the structure of adic spaces over $\calO_E$, and both are \emph{analytic}. This sheafiness allows us to define $\calY_S$ and $Y_S$ rigorously for any $S\in \Perf$. Using the map $\kappa$, one can easily show that $\varphi$ acts freely and totally discontinuously on $Y_S$. When $\calO_E=\Z_p$, $\calY_S$ is sometimes denoted as $S \dot{\times} \Z_p$ (cf.\ \cite{SWBerkeleyNotes}).
\end{rem}

In the remainder of this subsection, we define the diamantine Fargues--Fontaine curve over a diamond (associated to analytic adic spaces). 

Note that there is a natural $q$-th power Frobenius action $\varphi$ on the set of untilts, which defines an action of $\varphi^{\Z}$ on $\Spd S$. We possess the following ``functor of points'' descriptions of $\calY_S$ and $Y_S$:

\begin{prop}[{\cite[Props.~II.1.2 \& II.1.17]{FarguesScholze}}]
For any $S\in \Perf$, the following isomorphisms hold:
\begin{itemize}
	\item $\calY_S^\diamondsuit \simeq \Spd S \times \Spd \calO_E$;
	\item $Y_S^\diamondsuit \simeq \Spd S \times \Spd E$.
\end{itemize}
\end{prop}

This naturally motivates the following definition.

\begin{defn}
Let $S$ be an analytic adic space over $\Spa\,\mathbb{Z}_p$, or more generally any diamond $S^\diamondsuit$. Let $\varphi$ denote the absolute Frobenius on the diamond $S^\diamondsuit$. We make the following definitions:
\begin{itemize}
	\item $\mathcal{Y}_S^{\diamondsuit} \cong S^\diamondsuit \times \Spd\mathcal{O}_E$;
	\item $Y_S^{\diamondsuit} \cong S^\diamondsuit \times \Spd E$;
	\item $X_S^{\diamondsuit} \cong (S^\diamondsuit \times \Spd E) / (\varphi^{\Z}\times \mathrm{id})$.
\end{itemize}
We call $X_S^{\diamondsuit}$ the \emph{diamantine Fargues--Fontaine curve} over $S$.
\end{defn}

\begin{rem}\label{rem: cartier divisor on diam FF curve}
\begin{enumerate}
	\item When $S\in \Perf$, this definition agrees with the diamond associated to the \emph{adic Fargues--Fontaine curve} over $S$ (cf.\ \cite[Prop.~II.1.17]{FarguesScholze}).
	\item Fix a morphism $S^\diamondsuit \to \Spd \Z_p$. Then for any $T\in \Perf/S^{\diamondsuit}$, the composition $T \to S^\diamondsuit \to \Spd \Z_p$ defines an untilt $T^\sharp$ of $T$, which in turn defines a closed Cartier divisor $D_{T^\sharp} \subset X_T$ (cf.\ \cite[Prop.~II.1.18]{FarguesScholze}). The Cartier divisor $D_{T^\sharp} \subset X_T$ is functorial in $T$, yielding a closed Cartier divisor $D_{S^\sharp} \subset X_S^\diamondsuit$.
\end{enumerate} 
\end{rem}

\begin{eg}
Suppose $S^\sharp =\Spa(A,A^+)$ is an affinoid perfectoid space over $\Q_p$, and let $S=\Spa(R,R^+) \in \Perf$ be its tilt. The morphism $S^\sharp \to \Spa \Q_p$ defines a degree-$1$ Cartier divisor generated by the ideal $(\xi) \in \calO(\calY_S)$, where $\xi$ is primitive of degree $1$ in the sense of \cite[Defn.~6.2.9]{SWBerkeleyNotes}. If we further assume that $S^\sharp \to \Spa \Q_p$ factors through $\Spa \C_p \to \Spa \Q_p$, then we may assume that $\xi = \varphi^n(p-[p^\flat])$ for some $n\in\Z$. We will always normalize the ``relative position function'' $\kappa$ such that, in this situation, $\kappa$ is defined by choosing $\varpi=p^\flat$. Consequently, $\kappa(D_{S^\sharp})=1$. This follows the identical normalization utilized in \cite[\S11-14]{SWBerkeleyNotes}. 
\end{eg}

It is instructive to examine the case where $S$ is a Noetherian affinoid analytic adic space.

\begin{lem}\label{lem: S as G-torsor quot}
There exists a profinite group $G$ and a perfectoid pro-\'etale covering $\widetilde{S}$ of $S$ which is a $\underline{G}$-torsor (i.e., $\widetilde{S} \times_S \widetilde{S} \cong \underline{G} \times \widetilde{S}$, cf.\ \cite[Defn.~10.12]{scholze-etalecohomologyofdiamonds}). In particular, one has $S^\diamondsuit \cong \widetilde{S}^\diamondsuit/\underline{G} \cong \widetilde{S}^\flat/\underline{G}$.
\end{lem}
\begin{proof}
This construction shows why $S^{\diamondsuit}$ is a diamond; the proof of this fact can be found in \cite[Lem.~10.1.7]{SWBerkeleyNotes}.
\end{proof}

\begin{cor}
Maintaining the notation from Lemma~\ref{lem: S as G-torsor quot}, we have
\[
X_S^{\diamondsuit} \cong (\widetilde{S}^{\diamondsuit} \times \Spd E) / ((\underline{G} \times \varphi^{\Z}) \times \mathrm{id}).
\]
In particular, $X_S^{\diamondsuit}$ is a diamond over $\Spd E$.
\end{cor}
\begin{proof}
The fact that $X_S^{\diamondsuit}$ is a diamond follows directly from \cite[Prop.~11.6]{scholze-etalecohomologyofdiamonds}. 
\end{proof}

\begin{eg}\label{eg: absolute FF curve}
When $S=\Spa(L, L^+)$, where $L$ is a complete algebraically closed nonarchimedean field of characteristic $p$ and $E=\Q_p$, we call the curve $X_{(L, L^+)} \coloneqq X_S$ the \emph{absolute} Fargues--Fontaine curve over $L$. 
\end{eg}

\subsection{Vector bundles over the Fargues--Fontaine curve and shtukas}\label{sec: vb over FF}

In this subsection, we recall essential facts regarding vector bundles over the relative Fargues--Fontaine curve, $p$-adic shtukas, and admissible modifications of vector bundles over $X_S$ along a fixed characteristic-$0$ untilt. We then generalize these concepts to the \emph{diamantine} Fargues--Fontaine curve $X_S^\diamondsuit$ and relate them to other well-known categories, such as the categories of $(\varphi,\Gamma)$-modules and relative $B$-pairs. From now on, we assume $E=\Q_p$ in the definition of the relative Fargues--Fontaine curve unless stated otherwise.

First, consider the case where $S$ is a perfectoid space. Recall that in this scenario, $Y_S$ is an analytic adic space and is quasi-Stein (cf.\ \cite[Defn.~2.6.2]{kedlaya-liu-relative-padichodgeII}). Consequently, there is a well-behaved category of vector bundles over $Y_S$ and $X_S$ (see also \cite[\S~II.2]{FarguesScholze}). We summarize the requisite results when $S=\Spa(R,R^+)$ is an affinoid perfectoid space in the following lemma.

\begin{lem}
Let $S=\Spa(R,R^+)$ be an affinoid perfectoid space. For every vector bundle $\calE$ over $Y_S$, one has $H^i(Y_S,\calE)=0$ for all $i>0$.
\end{lem}
\begin{proof}
This is the result recalled at the beginning of \cite[\S II.2]{FarguesScholze}.
\end{proof}

An immediate consequence of the lemma above is that the category $\Bun(X_S)$ of vector bundles over $X_S$ can be characterized using the following categories of semilinear data.

\begin{lem}\label{lem: vb as phi modules}
When $S=\Spa(A,A^+)$ is an affinoid perfectoid space, the following categories are equivalent:
\begin{enumerate}
\item The category of vector bundles over $X_S$;
\item The category of $\varphi$-modules over $\calO(Y_S)$, consisting of finite projective $\calO(Y_S)$-modules $M$ equipped with an isomorphism $\varphi_M\colon\varphi^\ast M\xrightarrow{\sim}M$;
\item The category of $\varphi$-modules over $\widetilde{\mathcal R}_S$, consisting of finite projective $\widetilde{\mathcal R}_S$-modules $M$ equipped with an isomorphism $\varphi_M\colon\varphi^\ast M\xrightarrow{\sim}M$.
\end{enumerate}
\end{lem}
\begin{proof}
The assertion follows from \cite[Thms.~6.3.12 and 8.7.7]{kedlaya-liu-relative-padichodge}. The functor from \textup{(1)} to \textup{(2)} takes the global sections of the pullback to $Y_S$, with its Frobenius descent datum, and the functor from \textup{(2)} to \textup{(3)} is extension of scalars.
\end{proof}

\begin{rem}\label{rem: phi bundles}
It is convenient to recall that Kedlaya and Liu define the notion of $\varphi$-bundles over $\calO(Y_{S,I})$ (cf.\ \cite[Defn.~6.1.2]{kedlaya-liu-relative-padichodge}) for intervals $I=[a,\infty)$ or $I=[a,b]$ where $b/a>q$. They demonstrate that when $S$ is affinoid, restriction defines an equivalence between vector bundles over $X_S$ and $\varphi$-bundles over $\calO(Y_{S,I})$ (cf.\ \cite[Lem.~6.1.5]{kedlaya-liu-relative-padichodge}).
\end{rem}

This framework generalizes to the following definition.

\begin{defn}\label{defn: vb on X_Sdiamond}
Let $S$ be an analytic adic space over a $p$-adic field $K$. We define
\[
\Bun(X_S^{\diamondsuit}) = 2\text{-}\lim \Bun(X_{T}),
\]
where the $2$-limit is taken over pairs $(T, f_{T})$ with $T \in \Perf$ and $f_{T}\colon T \to S^{\diamondsuit}$ in $\Shv(\Perf_{\vv})$. In other words, an object in $\Bun(X_S^{\diamondsuit})$ is a functorial assignment of an object in $\Bun(X_{T})$ for each section $f_T \in S^{\diamondsuit}(T)$, satisfying the usual pullback isomorphisms.
\end{defn}

\begin{rem}
By \cite[Prop.~II.2.1]{FarguesScholze}, the assignment
\(T \mapsto \Bun(X_T)\)
forms a $\vv$-stack. Therefore, for any $\vv$-cover $\widetilde{S} \to S$ with \v{C}ech nerve $\widetilde{S}^\bullet$, the natural functor induces an equivalence of groupoids:
\[
\Bun(X_S) \;\simeq\; \lim\nolimits_{\Delta}\, \Bun\bigl(X_{\widetilde{S}^\bullet}\bigr).
\]
This implies the limit in Definition~\ref{defn: vb on X_Sdiamond} can be computed using a single \v{C}ech nerve.
\end{rem}

\begin{lem}\label{lem: Bun as phi gamma modules}
Assume $S=\Spa(A,A^+)$ is a Noetherian affinoid analytic space, and let $\widetilde{S}$ be the $\underline{G}$-torsor over $S$ from Lemma~\ref{lem: S as G-torsor quot}. Then $\Bun(X_S^{\diamondsuit})$ is equivalent to the category of triples $(M,\varphi_{M},\rho)$, where $(M,\varphi_{M})$ is a $\varphi$-module over $\calO(Y_{\widetilde{S}})$, and $\rho$ is a continuous semilinear $G$-action on $M$ that commutes with $\varphi_M$.
\end{lem} 

\begin{proof}
Note that the assignment $T \mapsto \Bun(X_T)$ satisfies $\vv$-descent by \cite[Prop.~II.2.1]{FarguesScholze}. Choosing $\widetilde{S}=\Spa(A_\infty,A_\infty^+) \to S$ to be the $\underline{G}$-torsor as in Lemma~\ref{lem: S as G-torsor quot}, we obtain a canonical isomorphism of $\vv$-sheaves $\underline{G} \times \widetilde{S} \xrightarrow{\sim} \widetilde{S} \times_S \widetilde{S}$ mapping $(g,x) \mapsto (gx,x)$. Because $G$ is a profinite group, it acts naturally as a $\vv$-sheaf (see discussions around \cite[Defn.~10.12]{scholze-etalecohomologyofdiamonds}). By \cite[Lem.~10.13]{scholze-etalecohomologyofdiamonds}, pullbacks of $\underline{G}$-torsors are $\underline{G}$-torsors, and $\underline{G}$-torsors over perfectoid spaces are representable by perfectoid spaces. It is straightforward to see that $\widetilde{S} \times_S \widetilde{S}$ is represented by the Huber pair $(C^0(G, A_\infty), C^0(G, A_\infty^+))$. The exact same reasoning applies to the triple product, yielding $\widetilde{S} \times_S \widetilde{S} \times_S \widetilde{S} \cong \underline{G \times G} \times \widetilde{S}$.

By $\vv$-descent, an object $\calE$ in $\Bun(X_S^{\diamondsuit})$ is uniquely determined by an object $\widetilde{\calE} \in \Bun(X_{\widetilde{S}})$ equipped with a descent datum $f\colon p_0^\ast \widetilde{\calE} \xrightarrow{\sim} p_1^\ast\widetilde{\calE}$ in $\Bun(X_{\widetilde{S} \times_S \widetilde{S}})$, which must satisfy the standard cocycle condition on the triple product (where $p_0, p_1$ are the two natural projections). 

Because $\widetilde{S}$ is perfectoid, we have $X_{\widetilde{S}} = Y_{\widetilde{S}}/\varphi^\Z$. Therefore, $\Bun(X_{\widetilde{S}})$ is canonically equivalent to the category of finite projective modules $M$ over $\calO(Y_{\widetilde{S}})$ equipped with an isomorphism $\varphi_M\colon \varphi^*M \xrightarrow{\sim} M$. 

To translate the descent datum $f$ into a continuous $G$-action, we evaluate global sections over the affinoid pieces $Y_{T, I}$ for a closed interval $I=[1/q,1]$ and a fixed pseudouniformizer $\varpi \in \calO^+(\widetilde{S}^\flat)$ (cf.\ \cite[Lem.~6.1.5 \& Thm.~6.3.12]{kedlaya-liu-relative-padichodge}). Recall that for any $T=\Spa(B,B^+) \in \Perf$, using \cite[Ex.~7.4.2 (2)]{SWBerkeleyNotes}, we have
\[
\calO(Y_{T,I})=W(B^{\flat, \circ})\langle \frac{p}{[\varpi]}, \frac{[\varpi]^q}{p} \rangle\Big[\frac{1}{p}\Big].
\]
Applying this to $T = \widetilde{S} \times_S \widetilde{S}$, we have  $B^{\flat, \circ} = C^0(G, A_\infty^{\flat, \circ})$. Because $G$ is a profinite group, the functor of continuous functions $C^0(G, -)$ commutes naturally with both the Witt vector functor and the $p$-adic and $\varpi$-adic completions. Thus, we canonically obtain an honest isomorphism (avoiding using almost isomorphisms):
\[
\calO(Y_{\widetilde{S} \times_S \widetilde{S}, I}) \cong C^0(G, \calO(Y_{\widetilde{S}, I})).
\]
A similar argument for the triple product yields $\calO(Y_{\widetilde{S} \times_S \widetilde{S} \times_S \widetilde{S}, I}) \cong C^0(G \times G, \calO(Y_{\widetilde{S}, I}))$.

Under these identifications, the pullback $p_1^\ast\widetilde{\calE}$ corresponds to the space of continuous functions from $G$ to $M$. The descent datum $f$ then translates exactly to an invertible, $\calO(Y_{\widetilde{S}})$-semilinear, continuous map $\rho\colon G \times M \to M$ that commutes with $\varphi_M$. Finally, the cocycle condition on the triple product maps precisely to the algebraic group associativity constraint $\rho(gh) = \rho(g)\rho(h)$, concluding the proof.
\end{proof}

\begin{eg}
For example, consider the case when $S=\Spa(A,A^+)$ where $\Spf(A^+)$ is a small affine $p$-adic formal scheme over $\Spf \calO_K$. A vector bundle over $X^\diamondsuit_{S}$ translates to a relative $(\varphi,\Gamma)$-module over $\calR_{A}$. When the $(\varphi,\Gamma)$-module is \'etale, the above statement matches the essence of \cite[Lem.~2.13]{min-wang-rel-phi-gamma-prism-F-crys}.   
\end{eg}

We now proceed to the definition of shtukas with one leg at a given characteristic-$0$ untilt. We fix a diamond $S^\diamondsuit \to \Spd K$. By Remark~\ref{rem: cartier divisor on diam FF curve}, this morphism defines a Cartier divisor $D_{S^\sharp} \subset X_S^\diamondsuit$. We first recall the definition when $S^\diamondsuit$ is represented by a space $S \in \Perf$.

\begin{defn}[Shtuka over $S$ with one leg at $D_{S^\sharp}$]\label{defn: shtuka-one-leg}
Let $S=\Spa(R,R^+)\in \Perf$, making $S$ an affinoid perfectoid space over $k$.
Let $S^\sharp=\Spa(R^\sharp,R^{\sharp +})$ be an untilt of $S$ over $\Z_p$. By Remark~\ref{rem: cartier divisor on diam FF curve}, $S^\sharp$ defines a morphism $S \xrightarrow{f_{S^\sharp}} \Spd \Z_p$, as well as a closed Cartier divisor $D_{S^\sharp} \xhookrightarrow{} \calY_S$.

A \emph{shtuka over $S$ with one leg at $S^\sharp$ (or $D_{S^\sharp}$)} is a pair $(\calE,\varphi_{\calE})$, where $\calE$ is a vector bundle on $\calY_S$ and
\[
\varphi_{\calE}:\ 
\Frob_{\calY_S}^*(\calE)\big|_{\calY_S\setminus D_{S^\sharp}}
\;\xrightarrow{\ \sim\ }\;
\calE\big|_{\calY_S\setminus D_{S^\sharp}}
\]
is an isomorphism on the complement of $D_{S^\sharp}$ in $\calY_S$, which is allowed to be meromorphic along the divisor $D_{S^\sharp}\subset \calY_S$. To maintain compatibility with Definition~\ref{defn: shtuka-vsheaf}, we will simply denote this category by $\Sht(S\xrightarrow{f_{S^\sharp}}\Spd \Z_p)$. 
\end{defn}

For any $n\in\Z$, there are well-defined categories $\Sht(S\xrightarrow{f_{S^\sharp}\circ\varphi^n}\Spd \Z_p)$, which corresponds to replacing $D_{S^\sharp}$ with $\varphi^{-n}(D_{S^\sharp})$ in the definition above.

\begin{defn}[Relative shtukas over a $\vv$-sheaf {\cite[Defn.~2.3.1]{Pappas-Rapoport-padicsht}}]\label{defn: shtuka-vsheaf}
Let $\calX$ be a $\vv$-sheaf equipped with a fixed morphism $f\colon \calX \to \Spd(\Z_p)$.
A \emph{shtuka} $(\calE,\varphi_{\calE})$ over $\calX/\Spd(\Z_p)$ is a section of the $\vv$-stack of the groupoid of shtukas over $\calX$. Equivalently, it is a functorial assignment which, for every $S\in \Perf$ and every point $x\in \calX(S)$, produces a shtuka $(\calE_S,\varphi_{\calE_S})$ over $S$ with one leg at the untilt $S^\sharp$ induced by the composition
\[
S \xrightarrow{x} \calX \to \Spd(\Z_p)
\]
(following Remark~\ref{rem: cartier divisor on diam FF curve}). We denote this category by $\Sht(\calX \xrightarrow{f} \Spd \Z_p)$.
\end{defn}

A closely related definition, which will be utilized later, is the following.

\begin{defn}\label{defn: modification of vector bundles}
For $S^\diamondsuit \to \Spd K$ as above, a \emph{modification of vector bundles} over $X_S^\diamondsuit$ along $S^\diamondsuit/\Spd K$ is a triple $(\calE_0,\calE_1,\beta)$, where $\calE_0$ and $\calE_1$ are objects in $\Bun(X_S^\diamondsuit)$ and $\beta$ is a modification of vector bundles along the untilt defined by $S^\diamondsuit/\Spd K$. More precisely, $\beta=\{\beta_T\}$ for each $T\in \Perf$ equipped with an arrow $f_T\colon T \to S^\diamondsuit$, such that 
\[
\beta_T\colon \calE_0|_{X_T\setminus D_{T^\sharp}} \xrightarrow{\cong} \calE_1|_{X_T\setminus D_{T^\sharp}}
\]
is an isomorphism of vector bundles over $X_T\setminus D_{T^\sharp}$, where $D_{T^\sharp}$ is the closed Cartier divisor on $X_T$ associated with the untilt defined by $T\to S^\diamondsuit \to \Spd K$. The map $\beta_T$ is meromorphic along $D_{T^\sharp}$, and $\beta_T$ is functorial in $T$.

A modification of vector bundles $(\calE_0,\calE_1,\beta)$ over $X_S^\diamondsuit$ along $S^\diamondsuit/\Spd K$ is called \emph{admissible} if $\calE_1$ is semistable of slope $0$ in the sense of Definition~\ref{defn: vector bundle being semistable of slope lambda}. We will use $\mathrm{Modif}(S^\diamondsuit \to \Spd K)$ (resp.\ $\admmodif(S^\diamondsuit \to \Spd K)$) to denote the category of modifications (resp.\ admissible modifications) of vector bundles over $X_S^\diamondsuit$ along $S^\diamondsuit/\Spd K$.
\end{defn}

\begin{eg}\label{ex: shtuka-over a field}
Let $S$ be an analytic adic space over $\Spa(\Q_p)$, and let $S^{\diamondsuit} \to \Spd(\Q_p)$ denote the associated morphism of $\vv$-sheaves. The preceding definition naturally yields the notion of a shtuka over $S^{\diamondsuit}/\Spd(\Q_p)$. Equivalently, a shtuka over $S^{\diamondsuit}/\Spd(\Q_p)$ is a functorial assignment equipping each object $(T,f_T) \in \Perf_{S^\diamondsuit}$ with a shtuka with one leg at the untilt given by
\[
T \xrightarrow{f_T} S^{\diamondsuit} \to \Spd(\Q_p),
\]
where $\Perf_{S^\diamondsuit}$ was defined in Remark~\ref{rem:equiv of PerfSdiam and PerfdS}. Remark~\ref{rem:equiv of PerfSdiam and PerfdS} gives the equivalence $\Perf_{S^\diamondsuit}\cong \Perfd_{S}$, so a shtuka over $S^{\diamondsuit}/\Spd(\Q_p)$ is exactly a functorial assignment equipping each object $(T,f_{T}) \in \Perfd_{S}$ with a shtuka over $T^\flat$ with one leg at the canonical untilt defined by $T$. We can succinctly write this as the limit formulation:
\[
\Sht(S^\diamondsuit/\Spd \Q_p) = (2\text{-})\lim_{(T\to \Spa\Q_p) \in \Perfd_S} \Sht(T/\Spa \Q_p).
\]
\end{eg}

In the remainder of this subsection, we recall basic facts concerning the \emph{absolute} Fargues--Fontaine curve $X_S$, where $S=\Spa(L, L^+)$ and $L$ is an algebraically closed nonarchimedean field over $\F_p$.

Recall from Example~\ref{eg: absolute FF curve} that when $L$ is a complete algebraically closed nonarchimedean field of characteristic $p$ and $E=\Q_p$, we define the absolute Fargues--Fontaine curve $X_{(L, L^+)}$. We recall the following fundamental result describing the isomorphism classes in $\Bun(X_{(L,L^+)})$.

\begin{lem}\label{lem: fact about Bun X L}
\begin{enumerate}
	\item The inclusion $L^+ \subset L^\circ$ induces a map $X_{(L, L^\circ)} \to X_{(L, L^+)}$ that identifies the categories of vector bundles on $X_{(L, L^\circ)}$ and $X_{(L, L^+)}$;
	\item The map $(D,\varphi_D) \mapsto \calE(\underline{D})\coloneqq(D\otimes \calO(Y_S), \varphi_D \otimes \varphi)$ defines a bijection between the isomorphism classes in $\Isoc_{\overline{\F}_p}$ and the isomorphism classes in $\Bun(X_{(L,L^+)})$. 
\end{enumerate}
\end{lem}
\begin{proof}
Statement $(1)$ is a consequence of GAGA (cf.\ \cite[Thm.~6.3.9]{kedlaya-liu-relative-padichodge}), noting that the schematic curve is constructed purely using the field $L$. Statement $(2)$ is established in \cite[Thm.~8.2.10]{FarguesFontaineAsterisque}.
\end{proof}

\begin{eg}\label{ex: example of vector bundles associated with de Rham and phiNGK-module}
\textbf{Examples of vector bundles associated with de Rham and $(\varphi,N,G_K)$-modules.}
\begin{enumerate}
\item A special case of Example~\ref{ex: shtuka-over a field} occurs when $S=\Spa(K,K^+)$ and $K$ is a $p$-adic field over $\Q_p$. In this setting, we write $X_K^\diamondsuit\coloneqq X_S^\diamondsuit$. Using Lemma~\ref{lem: Bun as phi gamma modules}, $\Bun(X_S^\diamondsuit)$ is equivalent to the category of ``$G_K$-equivariant vector bundles over $X_{\C_p^\flat}$''. These correspond precisely to triples $(M,\varphi_{M},\rho_{M})$, where $(M,\varphi_{M})$ is a $\varphi$-module over $\calO(Y_{\C_p^\flat})$, and $\rho_{M}$ is a continuous semilinear action of $G_K$ on $M$ that commutes with $\varphi_M$.
\item Let $\underline{D}=(D,\varphi_D,\rho_D)$ be a $(\varphi,G_K)$-module defined over $L/K$. We define a vector bundle $\calE(\underline{D})=\calE(D,\varphi_D,\rho_D)$ over $X_K^\diamondsuit$ as follows: Let $\calE(\underline{D})$ correspond to the triple $(M,\varphi_{M},\rho_{M})$, where $(M,\varphi_{M})=(D \otimes_{L_0} \calO(Y_{\C_p^\flat}), \varphi_D \otimes \varphi)$ and $\rho_M$ operates via the diagonal action of $G_K$. When $L=K$, we frequently denote this construction simply as $\calE(D,\varphi_D)$.
\item Let $(D,\varphi_D,N,\rho_D)$ be a $(\varphi,N,G_K)$-module defined over $L/K$. Fix a compatible system $\underline\epsilon=(1,\zeta_p,\zeta_{p^2},\ldots)$ of primitive $p$-power roots of unity and put $t=\log[\underline\epsilon]$. Define $c_{\underline\varpi}(g)\in\Z_p$ by $g(\underline\varpi)=\underline\varpi\,\underline\epsilon^{c_{\underline\varpi}(g)}$. Then
\[
g(u_{\underline\varpi})-u_{\underline\varpi}=c_{\underline\varpi}(g)t,\qquad c_{\underline\varpi}(gh)=c_{\underline\varpi}(g)+\chi_{\mathrm{cyc}}(g)c_{\underline\varpi}(h),
\]
where $\chi_{\mathrm{cyc}}$ is the cyclotomic character. We define $\calE(D,\varphi_D,N,\rho_D)$ by the triple $(M,\varphi_M,\rho_M)$, where
\[
(M,\varphi_M)=(D\otimes_{L_0}\calO(Y_{\C_p^\flat}),\varphi_D\otimes\varphi),
\]
and $\rho_M$ is the semilinear action determined by
\[
\rho_M(g)(y\otimes1)=\exp\bigl(-c_{\underline\varpi}(g)tN\bigr)\bigl(\rho_D(g)y\otimes1\bigr).
\]
The exponential is a finite sum because $N$ is nilpotent. The identity for $c_{\underline\varpi}(gh)$ implies the above a group action, and $N\varphi_D=p\varphi_DN$ together with $\varphi(t)=pt$ shows that $\rho_M$ commutes with $\varphi_M$.
\item From the definitions above, $\calE(D,\varphi_D,N,\rho_D)$ agrees with $\calE(D,\varphi_D,\rho_D)$ upon pullback to $X_{\C_p^\flat}$ (ignoring the $G_K$-action).
\end{enumerate}
\end{eg}

\begin{rem}
Example~\ref{ex: example of vector bundles associated with de Rham and phiNGK-module} is compatible with \cite[Chap.~10]{FarguesFontaineAsterisque} under GAGA for Fargues--Fontaine curve. Recall that we use the identical convention $N(\log[\underline{\varpi}])=1$. This is also similarly used in \cite[\S10.3.2]{FarguesFontaineAsterisque} when they defining $N$ on $B_{\log}$ and their formula for the semilinear Galois group action.
\end{rem}

\subsection{Shtukas in families and Newton polygon function}

We will provide an alternative description of shtukas over $S$ with one leg at $S^\sharp$ (when $S^\sharp$ is an untilt in characteristic $0$) as modifications of vector bundles over the relative Fargues--Fontaine curve. To accomplish this, we first introduce the theory of slopes. Specifically, we will define the Harder--Narasimhan slope function
\[
\NP(\calE)\colon \lvert S^\diamondsuit \rvert \to \Conv
\]
for any analytic adic space $S$ and any $\calE \in \Bun(X_S^\diamondsuit)$. We will then state a result describing shtukas over $S$ with \emph{no legs} in terms of local systems. Finally, we will adapt this to shtukas over $S$ with one leg and relate them to admissible modifications of vector bundles.

\begin{defn}\label{defn: many NP functions}
\begin{enumerate}
	\item The \emph{Newton polygon} of a multiset \[\underline{x}=\{x_i\}=\{x_1 \leq x_2 \leq \cdots \leq x_r\},\] with elements $x_i \in \mathbb{R}$, is defined to be the unique continuous, piecewise linear convex function $\NP(\underline{x}) \colon [0,r] \to \mathbb{R}$ passing through $(0,0)$ and having slope $x_i$ over the interval $[i-1,i]$. We will refer to the multiset $\underline{x}$ as the (multi)set of slopes of $\NP(\underline{x})$. 
	\item We define a partial order $\succeq$ on the collection of such multisets having the same size $r$ and the same total sum $\sum_{i=1}^r x_i$. For two such multisets $\underline{x} = \{x_1 \leq \dots \leq x_r\}$ and $\underline{y} = \{y_1 \leq \dots \leq y_r\}$, we say $\underline{x} \succeq \underline{y}$ if for all integers $1 \leq k < r$, we have
\[
\sum_{i=1}^k x_i \geq \sum_{i=1}^k y_i.
\]
In other words, this means that $\underline{x} \succeq \underline{y}$ if and only if their Newton polygons share the exact same endpoints, and the polygon $\NP(\underline{x})$ lies \emph{on or above} the polygon $\NP(\underline{y})$ (i.e., $\NP(\underline{x})(t) \geq \NP(\underline{y})(t)$ for all $t \in [0,r]$).
	\item For any multiset $\underline{x}$ as in $(1)$, and any $a \in \bR$, we let $a\underline{x}$ denote the multiset $\{a \cdot x_i\}$, and we write $a\NP(\underline{x})\coloneqq \NP(a\underline{x})$. We simply write $-\underline{x}$ for $(-1)\underline{x}$, and $-\NP(\underline{x})$ for $(-1)\NP(\underline{x})$. 
	\item Let $k$ be a perfect field of characteristic $p$. For $\underline D=(D,\varphi_D)\in\Isoc_k$, let $\{\lambda_1(D),\ldots,\lambda_r(D)\}$ be its multiset of Newton slopes, defined after base change to an algebraic closure then applying the Dieudonn\'e--Manin classification. Define $\NP(\underline D)$ to be the Newton polygon with these slopes.
\item Let $(L,L^+)$ be an algebraically closed nonarchimedean field of characteristic $p$. By Lemma~\ref{lem: fact about Bun X L}(2), any $\calE \in \Bun(X_L)$ admits an isomorphism $\calE \simeq \calE(\underline{D})$ for some $\underline{D}\in \Isoc_{\overline{\F}_p}$. We define $\NP(\calE)\coloneqq -\NP(\underline{D})$. We define the \emph{slopes} of $\calE$ to be the multiset of slopes of $\NP(\calE)$. By definition, the slopes of $\calE$ depend only on its isomorphism class. It is straightforward to check that for any morphism $(L_1,L_1^+) \to (L_2,L_2^+)$ between fields of this type, the slopes of any $\calE \in \Bun(X_{L_1})$ precisely match the slopes of the pullback $\calE\mid_{X_{L_2}}$.
\item Let $S$ be an analytic adic space. For any point $s\in \lvert S^\diamondsuit \rvert$, let $s\colon \Spd(L,L^+) \to S$ be the corresponding morphism of $\vv$-sheaves originating from an algebraically closed nonarchimedean field $(L,L^+)$ of characteristic $p$. For any $\calE \in \Bun(X_S^\diamondsuit)$, we define $\NP(\calE)(s)\coloneqq \NP(\calE_s)$, where $\calE_s \in \Bun(X_L)$ is the vector bundle obtained as the pullback of $\calE$ along $s\colon \Spd(L,L^+) \to S$. Note that following the discussion in $(4)$, $\NP(\calE)(s)$ is well-defined and depends only on the equivalence class of the map $\Spd(L,L^+) \to S$ corresponding to $s$, as delineated in \cite[Prop.~11.13]{scholze-etalecohomologyofdiamonds}.
\item Let $\calS_0$ be a scheme in characteristic $p$. As detailed in Construction~\ref{const: F-isocrystal at x}, for a finite locally free $F$-isocrystal $\underline{\calE}_{\cris}=(\calE, \varphi_{\calE})$ over $(\calS_0)^{\perf}_\Prism$ and any point $x \in |\calS_0|$, restriction defines an $F$-isocrystal $(\calE_x, \varphi_{\calE_x})$ over $\kappa(x)^{-p^\infty}$, which defines an object in $\Isoc_{\kappa(x)^{-p^\infty}}$. We define $\NP(\underline{\calE}_{\cris})(x)=\NP((\calE_x, \varphi_{\calE_x}))$.
\end{enumerate}
\end{defn}

For the remainder of this subsection, we fix an analytic adic space $S$ over a $p$-adic field $K$, and we let $\lambda =\frac{d}{h}$ denote a rational number in lowest terms (so $h>0$ and $\gcd(d,h)=1$ if $d\neq 0$; if $\lambda=0$, we adopt the convention $d=0$ and $h=1$).

\begin{defn}\label{defn: vector bundle being semistable of slope lambda}
A vector bundle $\calE \in \Bun(X_S^\diamondsuit)$ is called \emph{semistable of slope $\lambda$} if for all $s\in \lvert S \rvert$, the Newton polygon $\NP(\calE)(s)$ has only $\lambda$ (with multiplicity) as its slope. We let $\Bun(X_S^\diamondsuit)^{ss,\lambda}$ denote the full subcategory of semistable vector bundles of slope $\lambda$ over $X_S^\diamondsuit$.
\end{defn}

\begin{thm}\label{thm: local system and ssvb of slope 0}
There is a natural equivalence between the category $\Bun(X_S^\diamondsuit)^{ss,0}$ and the category of $\Q_p$-local systems over $S^\diamondsuit_{\vv}$.
\end{thm}

\begin{proof}
When $S\in\Perf$, this follows directly from \cite[Thm.~8.5.12]{kedlaya-liu-relative-padichodge} and Theorem~\ref{thm: equiv of all kinds of locsys}. The general case follows from the limit formula
\[
\Loc_{\Q_p}(S^\diamondsuit_{\vv}) \cong 2\text{-}\lim_{T \to S^\diamondsuit}\Loc_{\Q_p}(T_{\vv}) \cong 2\text{-}\lim_{T \to S^\diamond} \Bun(X_T)^{ss,0} \cong \Bun(X_S^\diamondsuit)^{ss,0},
\]
where the limit is taken over all $T\in \Perf$ admitting an arrow $T\to S^\diamondsuit$, with $T=\Spa(A,A^+)$. The first equivalence holds because $\Loc_{\Q_p}(S^\diamondsuit_{\vv})$ satisfies $\vv$-descent by definition. The second equivalence holds due to the perfectoid basis case, and the final equivalence follows from Definition~\ref{defn: vb on X_Sdiamond} and the fact that being semistable of slope $0$ descends along $\vv$-coverings. 
\end{proof}

The following definition and Corollary~\ref{cor: local system and shtuka with no leg} provide the analogous statement for $\Z_p$-local systems.

\begin{defn}
\begin{enumerate}
\item For $S\in \Perf$, a \emph{shtuka with no legs} over $S$ is a pair $(\calE,\varphi_{\calE})$, where $\calE$ is a vector bundle on $\calY_S$ and
\[
\varphi_{\calE}:\ 
\Frob_S^\ast \calE \xrightarrow{\sim} \calE
\]
is a global isomorphism of vector bundles over $\calY_S$. 
\item Let $\calX$ be a $\vv$-sheaf. A \emph{shtuka with no legs} $(\calE,\varphi_{\calE})$ over $\calX$ is a section of the $\vv$-stack of the groupoid of shtukas with no legs over $\calX$. 
\end{enumerate}
\end{defn}

\begin{cor}\label{cor: local system and shtuka with no leg}
Let $S$ be an analytic rigid space over $\Spa \Z_p$. Then the assignment
\[
\mathbb{L} \mapsto (\mathbb{L} \otimes \calO_{\calY_S^\diamondsuit}, 1 \otimes \varphi)
\]
defines an equivalence between the category $\Loc_{\Z_p}(S^\diamondsuit_{\vv})$ and the category of shtukas with no legs over $S^\diamondsuit$.
\end{cor}
\begin{proof}
This follows via the limit formula
\[
\begin{aligned}
\Loc_{\Z_p}(S^\diamondsuit_{\vv})
&\cong \varprojlim_{T \to S^\diamondsuit}\Loc_{\Z_p}(T_{\vv})
\cong \varprojlim_{T \to S^\diamondsuit}\{\text{shtukas with no legs over }T\} \\
&\cong \{\text{shtukas with no legs over }S^\diamondsuit\},
\end{aligned}
\]
where the limit is taken over all affinoid perfectoid spaces $T=\Spa(A,A^+)$ admitting an arrow $T\to S^\diamondsuit$. The first equivalence stems from the fact that $\Loc_{\Z_p}(S^\diamondsuit_{\vv})$ satisfies $\vv$-descent. The second equivalence relies on \cite{kedlaya-liu-relative-padichodge} or \cite[Prop.~22.6.1]{SWBerkeleyNotes}\footnote{To strictly apply the cited results, one replicates the proof from \cite[Prop.~12.3.5]{SWBerkeleyNotes}.}, combined with Theorem~\ref{thm: equiv of all kinds of locsys}. The final equivalence is true by definition. 
\end{proof}

Under the equivalence in Corollary~\ref{cor: local system and shtuka with no leg}, restriction from $\calY_S^\diamondsuit$ to $Y_S^\diamondsuit$ naturally induces a functor from the category of shtukas with no legs to the category of vector bundles over $X_S^\diamondsuit$. An object $\calE \in \Bun(X_S^\diamondsuit)$ is said to be \emph{extendable to a $\varphi$-module over $\calY_S^\diamondsuit$} if it lies in the essential image of this functor.

The following two corollaries follow from Theorem~\ref{thm: local system and ssvb of slope 0} and Corollary~\ref{cor: local system and shtuka with no leg}. Recall that the category of admissible modifications along $S^\diamondsuit \to \Spd K$ was addressed in Definition~\ref{defn: modification of vector bundles}. 

\begin{cor}\label{cor: shtukas with no leg and ss vb of slope 0}
An object $\calE \in \Bun(X_S^\diamondsuit)$ is extendable to a $\varphi$-module over $\calY_S^\diamondsuit$ if and only if $\calE \in \Bun(X_S^\diamondsuit)^{ss,0}$ and the $\Q_p$-local system corresponding to $\calE$ via Theorem~\ref{thm: local system and ssvb of slope 0} admits a $\Z_p$-lattice. Furthermore, there is a one-to-one correspondence between the isomorphism classes of $\varphi$-modules over $\calY_S^\diamondsuit$ mapping to $\calE$ and the isomorphism classes of $\Z_p$-lattices of the $\Q_p$-local system corresponding to $\calE$ via Theorem~\ref{thm: local system and ssvb of slope 0}. 
\end{cor}

\begin{cor}\label{cor: shtukas with one leg and admissible modification of vb}
Similarly, there is a well-defined functor from the category of shtukas over $S^\diamondsuit/\Spd K$ to the category of admissible modifications of vector bundles over $X_S^\diamondsuit$ along $S^\diamondsuit/\Spd K$. Its essential image consists precisely of those admissible modifications $(\calE_0,\calE_1,\beta)$ for which $\calE_1$ corresponds to a $\Q_p$-local system admitting a $\Z_p$-lattice. Furthermore, there is a one-to-one correspondence between the isomorphism classes of shtukas over $S^\diamondsuit/\Spd K$ that restrict to $(\calE_0,\calE_1,\beta)$ and the isomorphism classes of $\Z_p$-lattices of the $\Q_p$-local system associated with $\calE_1$ under Theorem~\ref{thm: local system and ssvb of slope 0}.
\end{cor}

It is also advantageous to recall the following lemma.
\begin{lem}[{\cite[Thm.~8.5.3]{kedlaya-liu-relative-padichodge}}]\label{lem: shtuka with no leg as phi modules}
Let $S=\Spa(A,A^+)\in \Perf$. Base change along the map $\calO(\calY_S) \to \widetilde{\calR}_S^{\rmint}$ establishes an equivalence between the category of shtukas with no legs over $S$ and the category of \'etale $\varphi$-modules over $\widetilde{\calR}_S^{\rmint}$ (i.e., pairs $(M,\varphi_{M})$ where $M$ is a finite projective $\widetilde{\calR}_S^{\rmint}$-module and $\varphi_{M} \colon \varphi^\ast M \xrightarrow{\cong} M$ is an isomorphism of $\widetilde{\calR}_S^{\rmint}$-modules).
\end{lem}

We can now formulate equivalent descriptions of shtukas with one leg over a fixed morphism $f\colon S^\diamondsuit \to \Spd K$, where $S$ is an analytic rigid space over a $p$-adic field $K$. We begin with a definition.

\begin{defn}[{\cite[Cor.~17.1.9]{SWBerkeleyNotes}, \cite[Defn.~3.5]{liu-zhu-rigidity}}]
Let $S$ and $K$ be as above. Let $\mathbb{B}_{\dR}^+$ be the functor defined by assigning to each $\Spa(R,R^+)\in \Perf$ equipped with an arrow $\Spd(R,R^+) \to S^\diamondsuit (\to \Spd K)$ the ring $\mathbb{B}_{\dR}^+(R^\sharp,R^{\sharp,+})$\footnote{By definition, these spaces $\Spa(R,R^+)$ form a basis for the $\vv$-site on $S^\diamondsuit_{\vv}$.}, where $(R^\sharp,R^{\sharp,+})$ is the characteristic-$0$ untilt of $(R,R^+)$ specified by $\Spd(R,R^+) \to \Spd K$. A $\mathbb{B}^+_{\dR}$-local system over $S^{\diamondsuit}/\Spd K$ is defined as a section of the $\vv$-stack\footnote{This structure follows directly from \cite[Cor.~17.1.9]{SWBerkeleyNotes}. By Proposition~\ref{prop: tau vector bundles}, this definition agrees with the definitions presented in \cite[Cor.~17.1.9]{SWBerkeleyNotes} and \cite[Defn.~3.5]{liu-zhu-rigidity}.} of the groupoid of vector bundles over $\mathbb{B}_{\dR}^+(R^\sharp,R^{\sharp,+})$.
\end{defn}

\begin{rem}[The complete stalk as a $\mathbb{B}_{\dR}^+$-local system]
Once we fix $f\colon S^\diamondsuit \to \Spd K$ as above, we can formalize the statement that the complete stalk of $X_S^\diamondsuit$ or $\calY_S^\diamondsuit$ at $f$ is exactly $\mathbb{B}_{\dR}^+$. For any $\Spd(R,R^+) \in  S^\diamondsuit_\vv$, let $(R^{\sharp},R^{\sharp,+},(\xi))$ be the specific untilt defined by $\Spd(R,R^+) \to \Spd K$. The natural map $\mathbb{A}_{\inf}(R^{\sharp},R^{\sharp,+})[1/p] \to \mathbb{B}_{\dR}^+(R^{\sharp},R^{\sharp,+})$ defined by taking the $(\xi)$-adic completion factors through $\calO_{\calY_{(R,R^+)}}$, establishing a well-defined morphism $\calO_{\calY_{(R,R^+)}} \to \mathbb{B}_{\dR}^+(R^{\sharp},R^{\sharp,+})$. This naturally constructs well-defined morphisms of $\vv$-sheaves (both of which we denote $f_\infty$) $\calO_{\calY_S^\diamondsuit} \to \mathbb{B}_{\dR}^+$\footnote{By this, we specifically mean a functorial assignment mapping $\calO_{\calY_{T^\sharp}} \to \mathbb{B}_{\dR}^+(T^\sharp)$ for all untilts $T^\sharp$ designated by $T \to S^\diamondsuit \to \Spd K$.} and $\calO_{X_S^\diamondsuit} \to \mathbb{B}_{\dR}^+$. We designate these morphisms as the completions at the untilt defined by $f\colon S^\diamondsuit \to \Spd K$. Consequently, if $\calE$ is a vector bundle over $\calY_S^\diamondsuit$ (or $X_S^\diamondsuit$), we define $\calE^+_{\dR}$ to be the $\vv$-sheaf obtained by pulling back $\calE$ along $f_\infty$. This serves as the "complete stalk of $\calE$ at the untilt defined by $f$." Note that if $\calE$ is the vector bundle corresponding to a $\Z_p$- (or $\Q_p$-) local system $\mathbb{L}$, then $\calE^+_{\dR}$ is canonically isomorphic to $\mathbb{L}\otimes \mathbb{B}_{\dR}^+$.
\end{rem}

\begin{thm}[Fargues' Theorem for admissible modifications]\label{thm: Fargues thm for admissible modification}
Let $S^\diamondsuit \to \Spd K$ be a fixed morphism of small $\vv$-sheaves. Then the assignment
\[
(\calE_0,\calE_1,\beta) \mapsto (\bL, (\calE_0)_{\dR}^+)
\]
(where $\bL$ is the $\Q_p$-local system corresponding to $\calE_1$ under Theorem~\ref{thm: local system and ssvb of slope 0}) defines an equivalence of categories between:
\begin{enumerate}
	\item The category $\admmodif(S^\diamondsuit \to \Spd K)$ of admissible modifications along $S^\diamondsuit \to \Spd K$;
	\item The category of pairs $\{(\mathbb{L},\Xi_{\dR}^+)\}$, where $\mathbb{L}\in \Loc_{\Q_p}(S^\diamondsuit_\vv)$ and $\Xi_{\dR}^+$ is a $\mathbb{B}^+_{\dR}$-local system over $S^\diamondsuit/\Spd K$ satisfying the isomorphism $\Xi_{\dR}^+ \otimes \mathbb{B}_{\dR} \cong \mathbb{L}\otimes \mathbb{B}_{\dR}$.
\end{enumerate}
\end{thm}
\begin{proof}
Because both categories can be written as $2$-limits, it suffices to verify the equivalence over perfectoid spaces over $K$. In that context, the result follows via Beauville--Laszlo gluing over the schematic relative Fargues--Fontaine curve (cf.\ \cite[Thm.~8.9.6]{kedlaya-liu-relative-padichodge} and \cite[Thm.~3.5.1]{CaraianiScholzecpt}).
\end{proof}

\begin{thm}[Fargues' Theorem for shtukas with one leg, cf.\ {\cite[Theorem 14.1.1]{SWBerkeleyNotes}}]\label{thm: Fargues' theorem}
Let $S^\diamondsuit \to \Spd K$ be a fixed morphism of small $\vv$-sheaves. Then the following categories are equivalent:
\begin{enumerate}
	\item The category of shtukas $\Sht(S\xrightarrow{f\circ\varphi} \Spd K)$ with one leg over $S\xrightarrow{f\circ\varphi} \Spd K$;
	\item The category of pairs $\{(\mathbb{L},\Xi_{\dR}^+)\}$, where $\mathbb{L}\in \Loc_{\Z_p}(S^\diamondsuit_\vv)$ and $\Xi_{\dR}^+$ is a $\mathbb{B}^+_{\dR}$-local system over $S^\diamondsuit/\Spd K$ such that $\Xi_{\dR}^+ \otimes \mathbb{B}_{\dR} \cong \mathbb{L}\otimes \mathbb{B}_{\dR}$;
	\item The category of quadruples $(\calE_0,\calE_1,\beta,\mathbb{L})$, where $(\calE_0,\calE_1,\beta)$ is an admissible modification of vector bundles over the diamantine Fargues--Fontaine curve $X_S^\diamondsuit$ along the untilt defined by $S^\diamondsuit \to \Spd K$, and $\mathbb{L}$ acts as a $\Z_p$-lattice inside the $\Q_p$-local system corresponding to $\calE_1$ under Theorem~\ref{thm: local system and ssvb of slope 0}.
\end{enumerate}
In the equivalence mapping from (3) to (2), $\Xi_{\dR}^+$ is obtained by taking the complete stalk of $\calE_0$ at the untilt defined by $f\colon S^\diamondsuit \to \Spd K$. For the equivalence from (1) to (3), upon evaluating over affinoid perfectoids, the pair $(\calE_1,\bL)$ corresponds to the shtuka with no legs obtained by the base change of a shtuka with one leg along $\calO(\calY_S) \to \widetilde{\calR}_S^{\rmint}$ (using Corollary~\ref{cor: local system and shtuka with no leg} and Lemma~\ref{lem: shtuka with no leg as phi modules}). The bundle $\calE_0$ is generated by restricting the shtuka with one leg to $Y_{S,[a,\infty)}$ for $a \gg 0$, under Remark~\ref{rem: phi bundles}.
\end{thm}
\begin{proof}
The equivalence of (1) and (2) was established in \cite[Prop.~2.5.1]{Pappas-Rapoport-padicsht}. As noted in the proof of \textit{loc.\ cit.}, the detailed arguments found in \cite[Prop.~12.4.6]{SWBerkeleyNotes} carry over identically, primarily resting on a Beauville--Laszlo type gluing theorem (see \cite[Lem.~5.2.9]{SWBerkeleyNotes}). The equivalence of \textup{(2)} and \textup{(3)} follows from Theorem~\ref{thm: Fargues thm for admissible modification} together with Corollary~\ref{cor: shtukas with one leg and admissible modification of vb}.
\end{proof}

Finally, we define the Newton polygon function for admissible modifications and shtukas with one leg. By Theorem~\ref{thm: Fargues thm for admissible modification} and Theorem~\ref{thm: Fargues' theorem}, there is a natural forgetful functor mapping $\Sht(S\xrightarrow{f\circ\varphi} \Spd K) \to \admmodif(S^\diamondsuit \to \Spd K)$, achieved simply by forgetting the associated $\Z_p$-lattice $\bL$ mandated in (3) of Theorem~\ref{thm: Fargues' theorem}.

\begin{defn}
For an admissible modification of vector bundles $(\calE_0,\calE_1,\beta)$ over $X_S^\diamondsuit$ along a fixed untilt, we define the associated Newton polygon function via 
	\[
	\NP((\calE_0,\calE_1,\beta))\coloneqq \NP(\calE_0)\footnote{Note that by our established convention, $\calE_1$ is inherently assumed to be semistable of slope $0$ (cf.\ Definition~\ref{defn: modification of vector bundles}).}.\] 
For a shtuka $\underline{\calF}$ with one leg corresponding to the quadruple $(\calE_0,\calE_1,\beta,\mathbb{L})$ under Theorem~\ref{thm: Fargues' theorem}, we explicitly set the associated Newton polygon function as:
\[
\NP(\underline{\calF})\coloneqq \NP((\calE_0,\calE_1,\beta,\mathbb{L}))= \NP(\calE_0).
\]
\end{defn}

\section{Shtukas attached to arithmetic local systems}\label{sec: Shtukas}

In this section, we fix a smooth rigid analytic variety $S$ over a $p$-adic field $K$, viewed as an adic space. Let $\mathbb{L}$ be a $\mathbb{Z}_p$-local system of rank $n$ on $S^\diamondsuit_{\vv}$. Note that by Theorem~\ref{thm: equiv of all kinds of locsys}, $\mathbb{L}$ can equivalently be viewed as a $\mathbb{Z}_p$-local system on $S_{\proet}$ or $S_{\et}$. Consequently, we may use the arithmetic properties of $\mathbb{L}$ defined in \S\ref{sec: arith loc sys}. Whenever we consider $p$-adic logarithmic formal schemes $(\calS,M_{\calS})$, we assume they satisfy the conditions specified at the beginning of \S\ref{sec: log-crystalline local systems}.  

\subsection{Geometric Galois representations and shtukas}\label{subsec: FF shtukas realization}

This subsection serves as a warm-up for the subsequent subsections. We begin with the case where the local system is defined over the diamond associated with a $p$-adic field $K$; equivalently, this corresponds to the setting of $G_K$-representations. This framework essentially traces back to Fargues and Fontaine \cite[Chap.~10]{FarguesFontaineAsterisque}.

Throughout this subsection, let $S=\Spa(K,K^+)$, where $K$ is a $p$-adic field. We briefly review the Fargues--Fontaine construction: we construct shtukas over $X_S^{\diamondsuit}$ arising from de Rham or potentially log-crystalline Galois representations of $G_K$. Because $G_K$ is a profinite group, every $p$-adic representation of $G_K$ admits a $\Z_p$-lattice.

As an application of Theorem~\ref{thm: Fargues' theorem}, we reproduce the following Fargues--Fontaine construction in the language of relative shtukas over $S \to \Spd \Q_p$.

\begin{construction}\label{const: FF dR}
Let $T \in \Rep_{\Z_p}^{\dR}(G_K)$. We define $\Xi=D_{\dR}(T\otimes \Q_p)\otimes B_{\dR}^+$ and endow it with the unique $G_K$-semilinear action extending the trivial action on $D_{\dR}(T\otimes \Q_p)$. By Part (2) of Theorem~\ref{thm: Fargues' theorem}, this defines a shtuka with one leg at $K/\Q_p$ over $X_K^\diamondsuit$. We denote the shtuka attached to $T$ in this manner by $\underline{\calF}_{\dR}(T)$ or $\underline{\calF}_{\dR}(\mathbb{L})$.
\end{construction}

\begin{construction}\label{const: FF logcris}
Let $(D,\varphi_D,N,\Fil^\bullet D_L,\rho_D)$ be a filtered $(\varphi,N,G_K)$-module defined over $L$. Following Example~\ref{ex: example of vector bundles associated with de Rham and phiNGK-module}, let $\calE(D,\varphi_D,N,\rho_D)$ denote the associated vector bundle over $X_K^\diamondsuit$. 

Following the proof of \cite[Prop.~10.3.18]{FarguesFontaineAsterisque}, we consider the complete stalk $\calE(D,\varphi_D,N,\rho_D)_{\infty}^\wedge$ of $\calE(D,\varphi_D,N,\rho_D)$ at the untilt $\infty$ defined by the structure morphism $\Spd K \to \Spd \Q_p$. By definition, this complete stalk is a $\bB^+_{\dR}$-local system over $S^\diamondsuit_{\vv}$, corresponding to a continuous semilinear representation of $G_K$ on a free module over $B^+_{\dR}\coloneqq \mathbb{B}^+_{\dR}(\C_p,\C_p^\circ)$. Specifically, there is a natural $G_K$-equivariant isomorphism between this stalk and $D\otimes_{L_0} B^+_{\dR}\cong D_L\otimes_{L} B^+_{\dR}$, where the latter is endowed with the unique semilinear extension of the $G_K$-action on $D$. 

Using the filtration $\Fil^\bullet D_L$, we define a modification $\calE(D,\varphi_D,N,\Fil^\bullet D_L,\rho_D)$ of the bundle $\calE(D,\varphi_D,N,\rho_D)$ along the untilt $\Spd K \to \Spd \Q_p$. This modification is uniquely determined by requiring its complete stalk to be
\[
\calE(D,\varphi_D,N,\Fil^\bullet D_L,\rho_D)_{\infty}^\wedge = \Fil^0(D_L \otimes_L B_{\dR}),
\]
which forms a $G_K$-stable $B_{\dR}^+$-lattice inside $\calE(D,\varphi_D,N,\rho_D)_{\infty}^\wedge[1/t]$. When the filtered module $(D,\varphi_D,N,\Fil^\bullet D_L,\rho_D)$ is weakly admissible, Fargues and Fontaine demonstrated that $\calE(D,\varphi_D,N,\Fil^\bullet D_L,\rho_D)$ is semistable of slope $0$. Consequently, there exists a Galois representation $V$ of $G_K$ whose restriction $V|_{G_L}$ is log-crystalline with $D_{\st}(V|_{G_L})= (D,\varphi_D,N)$, such that $\calE(D,\varphi_D,N,\Fil^\bullet D_L,\rho_D) \cong V\otimes_{\Q_p} \calO_{X_S^\diamondsuit}$. This geometric reinterpretation is essentially due to \cite[\S10.5]{FarguesFontaineAsterisque}, under the application of GAGA.

In summary, let $T$ be a $\Z_p$-lattice inside a potentially log-crystalline representation $V$ of $G_K$. Let its potentially log-crystalline period module $\underline{D}_{\pst}$ be the corresponding (weakly) admissible filtered $(\varphi,N,G_K)$-module $(D,\varphi_D,N,\Fil^\bullet D_L,\rho_D)$ defined over $L$. Then $T$ naturally defines a shtuka with one leg over $S^\diamondsuit/\Spd K$, which corresponds to the quadruple $(\calE_0,\calE_1,\beta,\mathbb{L})$ under Theorem~\ref{thm: Fargues' theorem}, where:
\begin{itemize}
    \item $\calE_1 \coloneqq \calE(D,\varphi_D,N,\Fil^\bullet D_L,\rho_D)$;
    \item $\calE_0 \coloneqq \calE(D,\varphi_D,N,\rho_D)$;
    \item $\beta$ is the modification at $\infty$ determined by the $B_{\dR}^+$-lattice $\Xi=D_{\dR}(V)\otimes_K B^+_{\dR}$;
    \item $\mathbb{L}$ is the $\Z_p$-local system over $S^\diamondsuit_\vv$ corresponding to the $G_K$-representation $T$.
\end{itemize}
\end{construction}

\begin{prop}[Newton polygon of the shtuka attached to a de Rham representation]
Let $T$ be a $\Z_p$-lattice inside a de Rham representation $V$ of $G_K$, and let $\underline{D}=(D,\varphi_D,N,\Fil^\bullet D_L,\rho_D)$ be the filtered $(\varphi,N,G_K)$-module corresponding to the potentially log-crystalline period of $V$. Then we have
\[
\NP(\underline{\calF}_{\dR}(T))(s) = -\NP(D,\varphi_D)
\]
for the unique point $s \in \lvert \Spd K \rvert = \{\ast\}$ (defined by an algebraically closed field).
\end{prop}
\begin{proof}
As representations of $G_K$ over $B_{\dR}^+$-lattices, we have the natural isomorphism $D_{\dR}(V)\otimes B_{\dR}^+ \cong D_{pst}(V)\otimes B_{\dR}^+$. The result then follows from Fargues' theorem (cf.\ Theorem~\ref{thm: Fargues thm for admissible modification}) and the discussion in Example~\ref{ex: example of vector bundles associated with de Rham and phiNGK-module}.
\end{proof}

\subsection{Admissible modifications attached to de Rham local systems}
\label{subsec:shtuka-derham}

This subsection extends Construction~\ref{const: FF dR} to the relative case. This construction also appears in \cite[\S2.6]{Pappas-Rapoport-padicsht}.

\begin{construction}\label{const: rel FF dR}
Let $S$ be a smooth rigid analytic variety over $K$, and let $\bL$ be a de Rham $\Q_p$-local system over $S$ as discussed in \S\ref{sec: dR and HT loc sys}. Recall from Remark~\ref{rem: M0 and M_1 for de Rham as proet vb} that attached to $\bL$ is an auxiliary $\mathbb{B}_{\text{dR}}^{+}$-local system $\bM_0(\bL)$ over $X_{\proet}$. This defines a $\mathbb{B}_{\text{dR}}^{+}$-lattice inside the $\mathbb{B}_{\text{dR}}$-local system
\[
\bL\otimes\mathbb{B}_{\text{dR}} \cong \bM_1(\bL)\otimes\mathbb{B}_{\text{dR}} \cong \bM_0(\bL)\otimes\mathbb{B}_{\text{dR}},
\] 
as detailed in Remark~\ref{rem: M0 and M_1 for de Rham as proet vb}. By Remark~\ref{rem: extend BdR and HT to v loc sys} and Theorem~\ref{thm: Fargues thm for admissible modification}, this isomorphism defines an object in $\admmodif(S^\diamondsuit/\Spd K)$.

We call the resulting functor 
\[
T_{\mathrm{modif}}\colon \Loc_{\Q_p}^{\dR}(S) \to \admmodif(S^\diamondsuit/\Spd K)
\]
the \emph{modification realization}. When $\bL$ is a de Rham $\Z_p$-local system over $S$, Theorem~\ref{thm: Fargues' theorem} allows us to define a functor
\[
T_{\mathrm{sht}}\colon \Loc_{\Z_p}^{\dR}(S) \to \mathrm{Sht}(S^\diamondsuit/\Spd K),
\]
which we call the \emph{shtuka realization}.
\end{construction}

\begin{lem}\label{lem: dR case compatible with pullback}
The functors $T_{\mathrm{modif}}$ and $T_{\mathrm{sht}}$ are fully faithful and compatible with pullbacks along arbitrary morphisms $S' \to S$ of smooth rigid analytic spaces over $K$.
\end{lem}
\begin{proof}
Compatibility with pullback follows from
\cite[Thm.~3.9(ii)]{liu-zhu-rigidity}. For full faithfulness, observe that
the $\calE_1$-term of $T_{\mathrm{modif}}(\bL)$ is $\bL\otimes_{\Q_p}\calO_{X_S^\diamondsuit}$. Hence a morphism of the associated modifications determines, and is determined by, a morphism of the corresponding $\Q_p$-local systems by Theorem~\ref{thm: local system and ssvb of slope 0}. The statement for $T_{\mathrm{sht}}$ follows similarly from Theorem~\ref{thm: Fargues' theorem}.
\end{proof}

\begin{defn}\label{defn: NP for dR loc sys}
For a de Rham $\Z_p$-local system (resp.\ $\Q_p$-local system) $\bL$ over $S$, we define $\NP(\bL)\coloneqq\NP(T_{\mathrm{sht}}(\bL))$ (resp.\ $\NP(\bL)\coloneqq\NP(T_{\mathrm{modif}}(\bL))$). We refer to this as the Newton polygon function on $\lvert S \rvert$ associated with $\bL$.
\end{defn}

\subsection{Shtukas attached to log-crystalline local systems}
\label{subsec:shtuka-logcrys}

This subsection extends Construction~\ref{const: FF logcris} to the relative case for log-crystalline local systems. In the purely crystalline setting, a similar construction is considered in \cite[Const.~3.19]{ImaiKYtannakian}.

\begin{construction}\label{const: rel FF logcris}
We begin with the following observation. Recall from Definition~\ref{defn: analytic F crystal} that for $(B,I) \in \calS^\perf_\Prism$, the category $\Vect^{\mathrm{an},\varphi}(B,\varphi(I))$ consists of pairs $(\mathcal{E}_B, \varphi_{\mathcal{E}_B})$, where $\mathcal{E}_B$ is a vector bundle over $\Spec(B)\smallsetminus V(p,\varphi(I))$ and $\varphi_{\mathcal{E}_B}$ is an isomorphism of vector bundles:
\[
    \varphi_{\mathcal{E}_B}\colon \varphi_B^\ast(\mathcal{E}_B)[\varphi(I)^{-1}] \simeq \mathcal{E}_B[\varphi(I)^{-1}].
\]
Assuming $\calS$ is normal, any affinoid perfectoid $T=\Spa(A,A^+)$ over the generic fiber $S$ of $\calS$ naturally defines a perfect prism $(B,I)$ in $\calS_{\Prism}$ via \cite[Thm~3.10]{BhattScholzePrism}, with $B=W(A^{+,\flat})$. We claim that this defines a shtuka with one leg along $\varphi(I)$. To prove this, it suffices to construct a pair $(\bL, \Xi_{\dR}^+)$ as required by Theorem~\ref{thm: Fargues thm for admissible modification}(2). We define $\bL$ to be the \'etale realization of $(\mathcal{E}_B, \varphi_{\mathcal{E}_B})$, and we define $\Xi_{\dR}^+$ as the base change of $\mathcal{E}_B[1/p]$ along the map $B[1/p]\to \bB_{\dR}^+(T)$ defined by taking $\varphi(I)$-adic completion. 

It is straightforward to verify that this construction is functorial in $T$. Thus, for a pair $(\calS,M_{\calS})$ conforming to the conditions at the beginning of \S\ref{sec: log-crystalline local systems}, this procedure defines a composition
\[
T_{\mathrm{sht}}\colon \Vect^{\mathrm{an},\varphi}((\calS,M_{\calS})_\Prism) \to \Vect^{\mathrm{an},\varphi}(\calS^\perf_\Prism) \to \Sht(S^\diamondsuit \xrightarrow{f^\diamondsuit\circ \varphi}\Spd K),
\]
where the first arrow restricts to the perfect site, the second arrow is given by the construction above, and $f^\diamondsuit$ is the functor associated with the structure morphism $S \to \Spa K$ of the generic fiber of $\calS$. We call $T_{\mathrm{sht}}$ the \emph{shtuka realization functor}.
\end{construction}

\begin{rem}
When the log structure $M_{\calS}$ is trivial and $\calS$ is formally smooth, it is straightforward to check that our construction agrees with the one in \cite[\S3.3]{ImaiKYtannakian}, courtesy of Proposition~\ref{propdefn: lattice realization}.
\end{rem}

\begin{prop}\label{prop: Tsht are compatible}
Let $\bL$ be a log-crystalline $\Z_p$-local system over $S$ with respect to $(\calS, M_{\calS})$. Its shtuka realization constructed by treating $\bL$ as a log-crystalline $\Z_p$-local system coincides with its shtuka realization constructed by treating $\bL$ as a de Rham $\Z_p$-local system.
\end{prop}
\begin{proof}
By Proposition/Definition~\ref{propdefn: lattice realization}, the \'etale realization of the analytic prismatic $F$-crystal attached to $\bL$ is $\bL$, and its lattice at the leg is the lattice occurring in the modification of Proposition/Definition~\ref{propdefn: lattice realization}. These are precisely the two pieces of data used in Construction~\ref{const: rel FF dR}. Theorem~\ref{thm: Fargues' theorem} therefore identifies the two constructions.
\end{proof}

\begin{cor}\label{cor: logcris case compatible with pullback}
The functor $T_{\mathrm{sht}}$ is fully faithful and is compatible with pullbacks along arbitrary morphisms $(\calS, M_{\calS}) \to (\calS', M_{\calS'})$ of $p$-adic logarithmic formal schemes as defined in \S\ref{sec: log-crystalline local systems}.
\end{cor}
\begin{proof}
This follows by combining Proposition~\ref{prop: Tsht are compatible}, Lemma~\ref{lem: dR case compatible with pullback}, and the pullback functoriality of prismatic $F$-crystals (cf.\ \cite[App.~C]{du-liu-moon-shimizu-purity-F-crystal}).
\end{proof}

\subsection{The Newton polygon function under specialization maps}

In this subsection, we maintain the notations established in \S\ref{sec: log-crystalline local systems}, where $(\mathcal{S}, M_{\mathcal{S}})$ denotes a (logarithmic) $p$-adic formal scheme over $\Spf \calO_K$. Let $S$ denote the adic generic fiber of $\calS$, which is, by definition, a smooth rigid analytic variety over $K$ (see also \cite[\S2.4]{du-liu-moon-shimizu-purity-F-crystal}). Because $\mathcal{S}$ is admissible and normal (regular), Construction~\ref{const: specialization map for formal scheme of certain type} applies and furnishes a map
\[
\mathrm{sp}_{\calS}:\ \lvert S\rvert \longrightarrow \lvert\calS_{\mathrm{red}}\rvert.
\]
Because the absolute Frobenius induces a universal homeomorphism on any scheme in characteristic $p$ (cf.\ \cite[Tag 0CC8]{stacks-project}), we may freely replace $\lvert\calS_{\mathrm{red}}\rvert$ (which is the same as $\lvert\calS_{1}\rvert$) with its perfection $\lvert\calS_{\mathrm{red}}^{-p^\infty}\rvert$.

Let $\bL$ be a log-crystalline local system over $S$ with respect to $(\mathcal{S}, M_{\mathcal{S}})$, meaning $\bL$ is associated with an $F$-isocrystal $\underline{\calE}_\cris$ over $\calS_1$ as in Definition~\ref{defn: log crystalline loc sys}. Let $x$ be a point of $S$ corresponding to a morphism $x\colon \Spa(K_x,K_x^+) \to S$ as described in Appendix~\ref{sec: specialization}. Let $\underline{\calF} = T_{\mathrm{sht}}(\bL)$ denote the shtuka realization of $\bL$ obtained via Construction~\ref{const: rel FF logcris} (which agrees with Construction~\ref{const: rel FF dR}). Recall the definition of the associated Newton polygon function $\NP(\bL)$ from Definition~\ref{defn: NP for dR loc sys}. 

The goal of this subsection is to establish the following theorem.

\begin{thm}\label{thm: NP function under specialization}
For any $y,x\in \lvert S\rvert$ with $x$ a generalization of $y$, we have 
\[
\NP(\bL)(x)=\NP(\bL)(y).
\] 
Furthermore, for all $x\in \lvert S\rvert$, we have
\[
\NP(\bL)(x)=-\NP(\calE_\cris)(\mathrm{sp}_{\calS}(x_{\max})),
\]
where $x_{\max}$ is the maximal (rank-$1$) generalization of $x$.
\end{thm}
\begin{proof}
First, Lemma~\ref{lem: fact about Bun X L}(1) ensures that $\NP(\bL)$ is invariant under generalizations. Therefore, it suffices to prove that
\[
\NP(\bL)(x)=-\NP({\calE}_\cris)(\mathrm{sp}_{\calS}(x))
\]
when $x \in S$ is specifically a rank-$1$ point. Let $x\colon \Spa(K, K^\circ) \to S$ be the map defining the point $x$, where we may assume $K$ is algebraically closed. This map induces a canonical morphism of formal schemes $\Spf(K^\circ) \to \calS$, which we will also denote by $x$.

Let $\calE_\Prism$ be the analytic $F$-crystal over $(\mathcal{S}, M_{\mathcal{S}})$ given by Theorem~\ref{thm: prismatic classification}. Because the shtuka realization is defined entirely over the perfect prismatic site $\mathcal{S}^{\perf}_\Prism$, it is sufficient to evaluate the restriction $\calE_\Prism^{\perf}$ of $\calE_\Prism$ to $\mathcal{S}^{\perf}_\Prism$. For any $p$-adic formal scheme $\calT$, let $\calT_1$ denote its modulo $p$ fiber, and let $i_{\calT}$ denote the corresponding closed immersion. 

Consider the following commutative diagram, where $i_{x,\Prism}^{-1}$ denotes the functor induced by the modulo $p$ morphism $K^\circ \to K^\circ/p$, and $\overline{x}\colon \Spec(K^\circ/p) \to \calS_1$ is induced by $x \colon \Spf(K^\circ) \to \calS$:
\[
\begin{tikzcd}
\Vect^{\mathrm{an},\varphi}(\calS^\perf_\Prism) \arrow[r, "x_{\Prism}^{-1}"] \arrow[d, "i_{\calS_1,\Prism}^{-1}"']
& \Vect^{\mathrm{an},\varphi}(\Spf(K^\circ)^\perf_\Prism) \arrow[d, "i_{x,\Prism}^{-1}"] \\
\Vect^{\mathrm{an},\varphi}((\calS_1)^\perf_\Prism) \arrow[r, "\overline{x}_{\Prism}^{-1}"]
& \Vect^{\mathrm{an},\varphi}(\Spec(K^\circ/p)^\perf_\Prism)
\end{tikzcd}
\]
Because the perfection of $K^\circ/p$ has to be reduced, the perfection is nothing but the residue field $k$ of $K^\circ$. Recall that $\calE_\cris$ is the crystalline realization of $\calE_\Prism$ (by Theorem~\ref{thm: prismatic classification}), and from Proposition~\ref{prop: crystalline realization}, the crystalline realization functor is strictly compatible with $i_{\calS_1,\Prism}^{-1}$. Following the definition of $\NP(\calE_{\cris})$ given in Definition~\ref{defn: many NP functions}, the theorem reduces to proving the following specific statement: For an analytic prismatic $F$-crystal $\calE_\Prism$ over $\Spf (K^\circ)_{\Prism}^{\perf}$, let $\calE_\cris$ be its restriction to $\Spec(k)_\Prism$, which is (by definition) an isocrystal $M$ over $k$. Then the multiset $-\NP(T_{\mathrm{sht}}(\calE_\Prism))(\ast)$ precisely equals the multiset of Frobenius slopes of $M$, where $\ast$ is the unique point in $\Spa(K,K^\circ)$.

Since $K^\circ$ is perfectoid, $(K^\circ)_\Prism^{\perf}$ has the final object $(\bA_{\inf}(K^\circ),(d))$. By \cite[Thm.~14.1.1]{SWBerkeleyNotes}, the category of analytic prismatic $F$-crystals over $\Spf (K^\circ)_{\Prism}^{\perf}$ is equivalent to the category of Breuil--Kisin--Fargues modules over $\bA_{\inf}(K^\circ)$ (in the sense of \cite[Defn.~11.4.3]{SWBerkeleyNotes}). Furthermore, restriction to $\Spec(k)_\Prism$ corresponds precisely to base change along the map $\bA_{\inf}(K^\circ) \to W(k)[1/p]$.

Finally, \cite[Thm.~13.2.1, Thm.~13.4.1, and Rem.~13.4.2]{SWBerkeleyNotes} show that the multiset $-\NP(T_{\mathrm{sht}}(\calE_\Prism))(\ast)$ agrees with the slopes of the isocrystal.
\end{proof}

Theorem~\ref{thm: NP function under specialization} motivates the formulation of $\mathrm{sp}_{\max,\calS}$ presented in Definition~\ref{defn: spmax}. In particular, $\mathrm{sp}_{\max,\calS}$ and $\mathrm{sp}_{\calS}$ coincide on rank-$1$ points of $S$, and Theorem~\ref{thm: NP function under specialization} can be concisely stated as
\[
\NP(\bL)(x)=-\NP(\calE_\cris)(\mathrm{sp}_{\max,\calS}(x)),
\]
for any log-crystalline local system $\bL$.

\medskip
\noindent
\textbf{Openness of level sets around rank-$1$ points.}

\begin{thm}\label{thm: logcris case_open around rk 1 point}
Let $\bL$ be a log-crystalline local system over $S$ with respect to $(\calS, M_{\calS})$ as before. For any fixed multiset of rational numbers $\calP$, define the level set 
\[
\lvert S\rvert^{\calP}\coloneqq\{x\in \lvert S\rvert \mid \NP(\bL)(x)=\calP \}.
\] 
If $\lvert S\rvert^{\calP}$ is non-empty for some $\calP$, then it must has a non-empty topological interior. Moreover, every rank-$1$ point of $\lvert S\rvert^{\calP}$ belongs to its interior. 
\end{thm}
\begin{proof}
It is sufficient to show that any rank-$1$ point $y \in \lvert S\rvert^{\calP}$ is contained within an open subset entirely situated inside the level set. If this holds, the remaining claims follow automatically because for any arbitrary $x \in \lvert S\rvert^{\calP}$, its maximal generalization $x_{\max}$ also belongs to $\lvert S\rvert^{\calP}$ by Theorem~\ref{thm: NP function under specialization}.

To prove this, we first examine the analogous question over the modulo $p$ fiber $\calS_1$. By \cite[Cor.~2.3.2]{katzslopefiltration} or \cite[Thm.~3.12(a)]{KedlayaNotesonisocrystals}, the level sets 
\[
\lvert \calS_1\rvert^{-\calP}\coloneqq\{x\in \lvert \calS_1\rvert \mid \NP(\calE_\cris)(x)=-\calP \}
\]
of $\NP(\calE_\cris)$ form locally closed subsets of $\lvert \calS_1\rvert$. By Theorem~\ref{thm: NP function under specialization}, if $\lvert S\rvert^{\calP}$ is non-empty, then $\lvert \calS_1\rvert^{-\calP}$ must also be non-empty. Now, invoking Theorem~\ref{thm: tubes are open near rk 1}, we deduce that there exists an open subset $V\subset \lvert S \rvert$ satisfying $y \in V \subset \mathrm{sp}_{\calS}^{-1}(\lvert \calS_1\rvert^{-\calP})$. By Theorem~\ref{thm: NP function under specialization}, this precisely means $y \in V \subset \lvert S\rvert^{\calP}$.
\end{proof}

\begin{rem}
All the properties discussed in this section hold more generally for log-crystalline $\Q_p$-local systems. Because the statements are \'etale locally on $S$, one may apply Theorem~\ref{thm: equiv of all kinds of locsys} and Theorem~\ref{thm: localalteration} to assume without loss of generality that $\bL$ is an isogeny $\Z_p$-local system and that $S$ is the generic fiber of a \emph{small affine} log formal scheme $(\calS,M_{\calS})$ (in the sense of \cite[\S2.2]{du-liu-moon-shimizu-purity-F-crystal}). In this scenario, the results of \cite{du-liu-moon-shimizu-purity-F-crystal} apply directly.
\end{rem}

\subsection{\texorpdfstring{$\RpMT_1 \Longrightarrow \CNP_1$}{RpMT1 implies CNP1}}\label{RpMT1 to CNP1}

The goal of this subsection is to establish that $\RpMT_1 \Longrightarrow \CNP_1$, which forms one direction of Theorem~\ref{thm:main_rpmt1}. We recall the local relative $p$-adic monodromy conjecture $\RpMT_1$.

\begin{conj}[Local $p$-adic monodromy conjecture]\label{conj: local pMT}
Let $S$ be a smooth rigid-analytic space over a $p$-adic field $K$, and let $\mathbb{L}$ be a de Rham $\mathbb{Q}_p$-local system on $S$. For any point $s\in S$, there exists an open neighborhood $U$ of $s$, and a finite \'etale covering $V \to U$ of smooth rigid analytic varieties over $K$, such that $\bL|_{V}$ is log-crystalline at all classical points of $V$.
\end{conj}

Note that this paper will only consider Conjecture~\ref{conj: local pMT} for rank-$1$ points, which corresponds exactly to Conjecture~\ref{conj: RpMT1}.

Assuming Conjecture~\ref{conj: local pMT} holds for rank-$1$ points, we will demonstrate that the level set of the Newton polygon function attached to a de Rham local system is topologically open around any rank-$1$ point. In other words, the Newton polygon function $\NP(\bL)$ associated with a de Rham local system $\bL$ is locally constant around rank-$1$ points ($\CNP_1$).

First, we recall the following theorem in $p$-adic geometry.

\begin{thm}[\cite{hartl-semistablemodels, ZavyalovLocalUniformization}, simplified version]\label{thm: localalteration}
Let $X$ be a smooth, affinoid rigid analytic variety over a complete discretely valued field $K$. For any $x\in X$, there exist an open neighborhood $U$ of $x$, a finite separable extension $L/K$, and an \'etale surjective morphism $U'\to U$ of rigid analytic $K$-spaces such that $U'$ is the generic fiber of an affine strictly semistable admissible formal $\calO_L$-scheme $\calU'$.
\end{thm}

Suppose Conjecture~\ref{conj: local pMT} holds, and let $V\to U$ be the resulting finite \'etale covering. After a further \'etale covering of $V$, Theorem~\ref{thm: localalteration} allows us to work on the generic fiber of an affine semistable formal scheme. The pulled-back local system is log-crystalline at all classical points and hence is log-crystalline (in the sense of Faltings, see Definition~\ref{defn: log crystalline loc sys}) with respect to this model by Theorem~\ref{thm: pointwise}. 

This leads to the following consequence of Theorem~\ref{thm: logcris case_open around rk 1 point}.

\begin{thm}\label{thm: one direct}
Let $S$ be a smooth rigid analytic variety over a $p$-adic field $K$, and let $\bL$ be a de Rham local system on $S$. Assume Conjecture~\ref{conj: local pMT} holds around rank-$1$ points. Then, for any piecewise linear function $\calP$, the level set 
\[
\{x\in \lvert S\rvert \mid \NP(\bL)(x)=\calP \},
\] 
if non-empty, admits non-empty interior. Moreover, every rank-$1$ point inside $\{x\in \lvert S\rvert \mid \NP(\bL)(x)=\calP \}$ is contained in the interior of $\{x\in \lvert S\rvert \mid \NP(\bL)(x)=\calP \}$. 
\end{thm}
\begin{proof}
Following the logic in the proof of Theorem~\ref{thm: logcris case_open around rk 1 point}, the statement is local. It is sufficient to prove that for any rank-$1$ point $x \in S$ located in a non-empty level set of $\NP(\bL)$, there is an open neighborhood $U$ of $x$ entirely contained within that same level set. 

Because the value $\NP(\bL)(x)$ is invariant under pullback to any point lying over $x$ (cf.\ Definition~\ref{defn: many NP functions} (6) and Lemma~\ref{lem: fact about Bun X L}), and because \'etale morphisms between rigid analytic spaces are open maps on the underlying topological spaces (\cite[Prop.~1.7.8]{Huber-etale}), the statement can be checked after applying finite base field extensions and \'etale coverings. 

Assuming Conjecture~\ref{conj: local pMT} holds, we can invoke Theorem~\ref{thm: localalteration}, Theorem~\ref{thm: pointwise}, and Theorem~\ref{thm: logcris case_open around rk 1 point} to deduce that, after appropriate base field extensions and \'etale coverings, there exists an open neighborhood $V$ of a preimage of $x$ that is contained entirely in the level set determined by $\NP(\bL)(x)=\NP(\underline{\calF})(x)$, where $\underline{\calF}=T_{\mathrm{sht}}(\bL)$. This deduction relies on the fact that the construction of $\underline{\calF}$ starting from a de Rham local system is compatible with the construction of $\underline{\calF}$ starting from a log-crystalline local system whenever the initial de Rham local system happens to be log-crystalline (cf.\ Proposition~\ref{prop: Tsht are compatible}).
\end{proof}

\begin{rem}\label{rem: shimizu thm}
Note that Shimizu proved that Conjecture~\ref{conj: local pMT} holds in a neighborhood of every classical point; see \cite[Thm.~9.2]{Shimizu-monodromy}. In fact, Shimizu established a stronger statement: after restricting to an affinoid open neighborhood $V$ of $x$, the de Rham local system $\bL$ becomes horizontally de Rham in the sense of Definition~\ref{defn: horizontally de Rham}. Moreover, after a finite extension of the base field, $\bL$ becomes horizontally log-crystalline, in the sense that it is associated with a \emph{constant} $(\phi, N)$-module over $W(k)[1/p]$.

One way to see that this is compatible with \cite[\S4.2]{Shimizu-monodromy} is via purity; see \cite[Thm.~6.1]{Shimizu-monodromy} and \cite[Thm.~4.1]{du-liu-moon-shimizu-purity-F-crystal}. Assuming Shimizu's result, it follows immediately that $\NP(\bL)$ is locally constant and is equal to the negatives of the slopes of $D_{\mathrm{pst}}(\bL|_V)$, as defined in \cite[Defn.~5.6]{Shimizu-monodromy}.
\end{rem}

\section{The theory of slopes and slope filtrations}\label{sec: slopes}

This section reviews the theory of slopes and slope filtrations, following Kedlaya--Liu. Moreover, we carefully study the category of semistable vector bundles over diamantine Fargues--Fontaine curves and relate them to the category of $(d,h)$-local systems.

\subsection{Generic and special Newton polygons}

Recall the variants of period rings for $\varphi$-modules from Definition~\ref{defn: Robba rings}. As mentioned in \S\ref{sec: vb over FF}, the category of $\varphi$-modules over $\widetilde{\mathcal{R}}_A$ is equivalent to the category of vector bundles over $X_S$ when $S=\Spa(A,A^+)$ is affinoid (see also \cite[Thm.~6.3.12]{kedlaya-liu-relative-padichodge}). 

\begin{defn}\label{defn: NP function for a phi module over tdR_A}
For a $\varphi$-module $M$ over $\widetilde{\mathcal{R}}_A$, one can define the Harder--Narasimhan slope function 
\[
\NP(M)\colon \lvert S \rvert \to \Conv
\]
by requiring $\NP(M)(x) =\NP({\calE}(M))(x)$, where ${\calE}(M)$ is the vector bundle over $X_S$ corresponding to $M$ under Lemma~\ref{lem: vb as phi modules}.
\end{defn}

For a $\varphi$-module $M^{\bd}$ over $\widetilde{\mathcal{R}}^{\mathrm{bd}}_A$, there are two associated Newton polygon functions, which we now recall.

\begin{defn}[Special and generic Newton polygons]\label{defn: special and generic Newton polygons}
Let $M^{\bd}$ be a nontrivial $\varphi$-module over $\widetilde{\mathcal{R}}^{\mathrm{bd}}_A$.
\begin{itemize}
	\item Define the \emph{special} Newton polygon function 
\[
\NP_{s}(M^{\bd})\colon \lvert S \rvert \to \Conv
\] 
via $\NP_{s}(M^{\bd})(x) \coloneqq \NP(M)(x)$, where $M$ is the base change of $M^{\bd}$ along $\widetilde{\mathcal{R}}^{\mathrm{bd}}_A \to \widetilde{\mathcal{R}}_A$, and $\NP(M)$ is the function from Definition~\ref{defn: NP function for a phi module over tdR_A}.
	\item Define the \emph{generic} Newton polygon function
\[
\NP_{\eta}(M^{\bd})\colon \lvert S\rvert \to \Conv
\]
by
\[
\NP_{\eta}(M^{\bd})(x) \coloneqq 
-\NP\!\left(
M^{\bd}_{x_{\max}}
\otimes_{\widetilde{\mathcal{R}}^{\mathrm{bd}}_{K_{x_{\max}}}}
\widetilde{\calE}_{K_{x}}
\right),
\]
where $x_{\max}$ is the rank-$1$ generalization of $x$. Note that
$K_{x_{\max}}=K_{x}$ by Lemma~\ref{lem: Spa(K,K^+) is a chain}(4), so the tensor product inside $\NP$ on the right-hand side is an isocrystal over $K_x$. The Newton polygon function of an isocrystal is given in Definition~\ref{defn: many NP functions}.
\item When $S=\Spa(L,L^\circ)$ is an analytic field, we simply write $\NP_{s}(M^{\bd})$ and $\NP_{\eta}(M^{\bd})$ for $\NP_{s}(M^{\bd})(\ast)$ and $\NP_{\eta}(M^{\bd})(\ast)$, respectively, where $\ast$ is the unique point in $S$.
\end{itemize}
\end{defn}

\begin{rem}\label{rem: slope of phi^h module convention}
For any positive integer $h$, one can similarly define the category of $\varphi^h$-modules, together with their associated Newton polygon functions. Note that for a $\varphi$-module $(M,\varphi_M)$, the pair $(M,\varphi_M^h)$ is a $\varphi^h$-module, and we have
\[
h\NP_{\ast}(M,\varphi_M) = \NP_{\ast}(M,\varphi^h_M),
\]
where $\ast$ could be $\emptyset$, or $s$ and $\eta$ if $M$ is defined over the bounded Robba ring. Note that this convention may differ from the one in \cite[Defn.~7.4.1]{kedlaya-liu-relative-padichodge}, which is specifically defined for $h>1$. Our convention agrees with \cite[\S11]{FarguesFontaineAsterisque}.
\end{rem}

We summarize some basic properties of $\NP_{s}(M^{\bd})$ and $\NP_{\eta}(M^{\bd})$ from \cite{Kedlayapadiclocalmonodromy,Kedlayaslopefiltrationsrevisited} and \cite[\S7.4]{kedlaya-liu-relative-padichodge}.

\begin{thm}\label{thm: properties of generic and special NP} 
For any nontrivial $\varphi$-module $M^{\bd}$ over $\widetilde{\mathcal{R}}^{\mathrm{bd}}_A$ and any $x\in \lvert S\rvert$:
\begin{enumerate}
	\item Let $s\colon \Spa(L,L^\circ) \to S$ be a map from an analytic field $L$, and let $x$ be the image of the unique point in $\Spa(L,L^\circ)$. Then $\NP_{s}(M^{\bd})(x)$ (resp.\ $\NP_{\eta}(M^{\bd})(x)$) is equal to $\NP_{s}(M^{\bd}_L)$ (resp.\ $\NP_{\eta}(M^{\bd}_L)$). 
	\item The endpoints of $\NP_{s}(M^{\bd})$ and $\NP_{\eta}(M^{\bd})$ are locally constant.
	\item $\NP_s(M^{\bd})\succeq\NP_\eta(M^{\bd})$ as functions in $\Conv$ (in the sense of Definition~\ref{defn: many NP functions}).
\end{enumerate}
\end{thm}

\begin{proof}
Claim (1) follows from the definition or \cite[Rem.~7.4.2]{kedlaya-liu-relative-padichodge}. Claim (2) is \cite[Lem.~7.2.2]{kedlaya-liu-relative-padichodge}. Claim \textup{(3)} follows from \cite[Prop.~7.4.3(a)]{kedlaya-liu-relative-padichodge}, applied at each rank-$1$ point.
\end{proof}

\subsection{The semi-continuity result of Kedlaya--Liu}

In this subsection, we review the theorem of Kedlaya--Liu regarding the semi-continuity of the Newton polygon function attached to a vector bundle over $X_S$, or equivalently, a $\varphi$-module over $\widetilde{\mathcal{R}}_A$. This can be viewed as a generalization of Shatz's result \cite{Shatzfamilyofvb} for vector bundles over a family of Fargues--Fontaine curves.

\begin{thm}\label{thm: KL-semi-continuity}
For any $\varphi$-module $M$ over $\widetilde{\mathcal{R}}_A$, the function $\NP(M)$ is lower semicontinuous. More precisely, the following statements hold:
\begin{enumerate}
	\item When $M$ is of constant rank $h$, the endpoint function $\NP(M)(\mathrm{-})(h)$ is locally constant as a function from $\lvert S \rvert$ to $\bR$. For each $a \in [0,h]$, viewing $\NP(M)(\mathrm{-})(a)$ as a function from $\lvert S \rvert$ to $\bR$, for every rank-$1$ point $\beta \in \lvert S \rvert$ there is an open neighborhood $U$ of $\beta$ in $\lvert S \rvert$ such that $\NP(M)(x)(a) \geq \NP(M)(\beta)(a)$ for all $x \in U$. In particular, $\NP(M)(\mathrm{-})(a)$ is lower semi-continuous as a function from $\lvert S \rvert$ to $\bR$ for each $a \in [0,h]$.
	\item Moreover, after shrinking $S$ to a rational localization $S' \subset S$ containing $\beta$ and passing to a finite \'etale (Galois) covering $T \to S'$ such that $\beta$ has a unique preimage, there is a positive integer $m$ such that $(M_T,\varphi_{M_T}^m)$ admits a $\widetilde{\mathcal{R}}_T^{\bd}$-lattice $M_T^{\bd}$ as a $\varphi^m$-module. The generic Newton polygon function of this $M_T^{\bd}$ is constant on $T$, and the generic Newton polygon function at each fiber equals the Newton polygon of $(M_T,\varphi_{M_T}^m)$ at $\beta$.
\end{enumerate}
\end{thm}

Kedlaya and Liu only stated the semi-continuity in \cite[Thm.~7.4.5]{kedlaya-liu-relative-padichodge}, but their proof establishes the above slightly stronger formulation. We recall their argument from \textit{loc.\ cit.} here. 

This slightly stronger result presented here will not be used in the subsequent proofs. We record it for its independent interest and potential usefulness in other contexts. Readers familiar with the semi-continuity result may proceed directly to Proposition~\ref{prop: NP boundedness}.

\begin{proof}[Proof of Theorem~\ref{thm: KL-semi-continuity}] 
Many arguments mirror \cite[Thm.~7.3.7]{kedlaya-liu-relative-padichodge}. If Part (2) of this theorem holds, then Parts (1) and (3) of Theorem~\ref{thm: properties of generic and special NP}, combined with the identity $\NP(M_\beta,\varphi_{M_\beta}^m)=m \NP(M_\beta,\varphi_{M_\beta})$, imply Part (1) and the remaining claims. We now outline Kedlaya--Liu's proof of Part (2).

Fix $\beta\in \lvert S \rvert$ as in the statement of the theorem. 

The first key input is a \emph{spreading out lemma}. 
\begin{lem}\label{lem: spead out lemma in KL}
After replacing $S$ by a rational localization $S' \subset S$ followed by a finite \'etale (Galois) covering $T \to S'$, one may arrange that $M_S$ is free over $\widetilde{\mathcal{R}}_{S}$ on a basis $\{v_1,\ldots,v_h\}$. Moreover, there is a positive integer $m$ such that $\varphi^m_M$ acts on $\{v_1,\ldots,v_h\}$ via a matrix $C'=FD$, where $D$ is diagonal with entries in $p^{\Z}$ and $F-I_h$ has entries in $p\widetilde{\mathcal{R}}^{\mathrm{int}}_{S}$ with strictly positive $w_r$-valuations (i.e., $w_r(F-I_h)>0$). Furthermore, there is a basis of $M_{\overline{\beta}}$ on which $\varphi^m_{M_{\overline{\beta}}}$ acts via $D$.
\end{lem}
\begin{proof}
Because $\widetilde{\mathcal{R}}_{K_\beta}$ is a B\'ezout domain (cf.\ \cite[Lem.~4.2.6]{kedlaya-liu-relative-padichodge}), the fiber $M_{\overline{\beta}}$ of $M$ at the geometric point ${\overline{\beta}}$ over $\beta$ is free. Moreover, by \cite[Prop.~4.2.16]{kedlaya-liu-relative-padichodge}, there is a basis $\{e_1,\ldots,e_h\}$ of $M_{\overline{\beta}}$ and a positive integer $m$ such that $\varphi^m_{M_{\overline{\beta}}}$ acts on this basis via a diagonal matrix $D_{\overline{\beta}}$ with entries in $p^{\Z}$. We may also assume $M_{\overline{\beta}}$ is defined over $\widetilde{\mathcal{R}}_{K_\beta}^r$ for some $r>0$. Let $M_{\overline{\beta}}^r$ (resp.\ $M_{\overline{\beta}}^{[r,pr]}$) be the $\widetilde{\mathcal{R}}_{K_{\overline{\beta}}}^r$-module (resp.\ $\widetilde{\mathcal{R}}_{K_{\overline{\beta}}}^{[r,pr]}$-module) generated by $\{e_1,\ldots,e_h\}$. 

Recall that the separable closure of $K_\beta$ is dense inside $\overline{\beta}$ by definition. Consequently, the colimit of $\widetilde{\mathcal{R}}_{L}^{[r,pr]}$ over all finite \'etale (Galois) extensions $L$ of $K_\beta$ is dense inside $\widetilde{\mathcal{R}}_{K_{\overline{\beta}}}^{[r,pr]}$. Using \cite[Lem.~2.2.13]{kedlaya-liu-relative-padichodge}, we can find a finite \'etale (Galois) extension $E$ of the residue field of $\calO_{S,\beta}$ (which is dense in $K_\beta$) and elements $\{e'_1,\ldots,e'_h\}$ inside $M_{\overline{\beta}}^{[r,pr]}$ such that the transition matrix $U$ from $\{e_1,\ldots,e_h\}$ to $\{e'_1,\ldots,e'_h\}$ satisfies $w_r(U-I_h)>0$. That is, all entries of $U-I_h$ lie inside $\widetilde{\mathcal{R}}_{L}^{[r,pr]}$ (where $L$ is the completion of $E$) and have positive $w_r$-valuations. In particular, $\{e'_1,\ldots,e'_h\}$ forms a basis; let $M_{L}^{[r,pr]}$ be the free $\widetilde{\mathcal{R}}_{L}^{[r,pr]}$-module generated by it. 

Since the stalk $\calO_{S,\beta}$ of $S$ at $\beta$ is a Henselian local ring whose residue field is dense in $K_\beta$ (cf.\ \cite[Lem.~2.4.17]{kedlaya-liu-relative-padichodge}), we can find a rational localization $S'=\Spa(A',A'^+)$ containing $\beta$ and a finite \'etale covering $T=\Spa(B,B^+)$ of $S'$ such that there is a unique point $\beta'$ lying over $\beta$, and the stalk $\calO_{T,\beta'}$ has residue field $E$. By \cite[Lem.~5.1.7]{kedlaya-liu-relative-padichodge}\footnote{The same proof works for rational localizations of $S$.}, up to shrinking $S'$ further, we may assume $\{e'_1,\ldots,e'_h\}$ generates $M_{T}^{[r,pr]}\coloneqq M\otimes \widetilde{\mathcal{R}}_{T}^{[r,pr]}$ as a module. Moreover, $\varphi^m_M$ acts on $\{e'_1,\ldots,e'_h\}$ via $U^{-1}D\varphi^m(U)$, such that $U^{-1}D\varphi^m(U)D^{-1}$ has entries inside $p\widetilde{\mathcal{R}}_{T}^{[r,pr]}$ with positive $w_r$-valuations. 

We now apply \cite[Lem.~7.1.2]{kedlaya-liu-relative-padichodge} to produce another basis $(v_1,\ldots,v_h)=(e_1',\ldots,e_h')U'$ of $M_{T} \coloneqq M\otimes \widetilde{\mathcal{R}}_{T}$. Here, $U'$ is invertible, $U'-I_h$ has entries in $p\widetilde{\mathcal{R}}_{T}^{[r,pr]}$, and both $U'-I_h$ and $D^{-1}U'D-I_h$ have positive $w_r$-valuations. Now $\varphi^m$ acts on $(v_1,\ldots,v_h)$ via the matrix $C'\coloneqq (UU')^{-1}D\varphi^m(UU')$\footnote{In the statement of \cite[Lem.~7.1.2]{kedlaya-liu-relative-padichodge}, take their $D$ to be our $D^{-1}$.}. The proof of \cite[Lem.~7.1.2]{kedlaya-liu-relative-padichodge} guarantees that $C'D^{-1}-I_h$ has entries inside $p\widetilde{\mathcal{R}}_{T}^{[r,pr]}$ with strictly positive $w_r$-valuations. For the final valuation statement, note that the matrix $G-I_h\coloneqq C'D^{-1}-I_h$ is the limit of a Cauchy sequence $C'_\ell D^{-1}-I_h$ of matrices with entries in $\widetilde{\mathcal{R}}_{T}^{[r,pr]}$ satisfying $\lambda(\alpha^r)(C'_\ell D^{-1}-I_h)<1$. In particular, their limit also satisfies $w_r(C'D^{-1}-I_h)>0$.
\end{proof}

By the lemma above, we can replace $S$ with $T$ to assume $M$ is free and admits a basis on which $\varphi$ acts via the matrix $C'=FD$, where $C'$ has entries in $\widetilde{\mathcal{R}}^{\bd}_{S}$. Let $M^{\bd}$ be the $\widetilde{\mathcal{R}}^{\bd}_{S}$-module generated by $\{v_1,\ldots,v_h\}$; it forms a $\varphi^m_M$-module. Furthermore, the following holds:

\begin{lem}[{\cite[Lem.~7.4.4]{kedlaya-liu-relative-padichodge}}\footnote{Note that we do not need to divide by $m$ due to our convention, cf.\ Remark~\ref{rem: slope of phi^h module convention}.}]\label{lem: find generic slopes from valuations of D}
Let $N$ be any finite free $\varphi^m$-module over $\widetilde{\mathcal{R}}^{\bd}_{S}$ such that $\varphi^m_N$ acts on a basis via the matrix $FD$, where $F-I_h$ has entries in $p\widetilde{\mathcal{R}}^{\mathrm{int},r}_{S} \subset p\widetilde{\mathcal{R}}^{\mathrm{int}}_{S}$, and $D$ is diagonal with entries in $p^{\Z}$. Then there exists an $A$-algebra $A'$ (which is a filtered union of finite faithful \'etale $A$-subalgebras) and an invertible matrix $U$ over $W(A')$ congruent to $I_h$ modulo $p$, such that $U^{-1}FD \varphi^m(U) = D$. Moreover, the generic slopes of $N$ (as a $\varphi^m$-module) equal the multiset of the negatives of the $p$-adic valuations of the entries of $D$.
\end{lem}

In particular, the $\varphi^m$-module $M^{\bd}$ has a constant generic Newton polygon on $T$, which is equal to the multiset of the negatives of the $p$-adic valuations of the entries of $D$. Conversely, by Lemma~\ref{lem: spead out lemma in KL}, this constant generic Newton polygon agrees with the special Newton polygon of $(M_\beta,\varphi_{M_\beta}^m)$ at $\beta$. This concludes the proof of Theorem~\ref{thm: KL-semi-continuity}.
\end{proof}

This theorem establishes the boundedness of the Newton polygon.

\begin{prop}[Boundedness]\label{prop: NP boundedness}
For any $\varphi$-module $M$ over $\widetilde{\mathcal{R}}_A$, the Newton polygon function $\NP(M_\beta,\varphi_{M_\beta})$ is bounded above and below on $\lvert S \rvert$.
\end{prop}
\begin{proof}
Part (2) of Theorem~\ref{thm: KL-semi-continuity} implies that $\NP(M_\beta,\varphi_{M_\beta})$ is locally bounded below. Since $S$ is quasi-compact by assumption, it is bounded below on all of $\lvert S \rvert$. By Theorem~\ref{thm: properties of generic and special NP}(2), the endpoint is locally constant. Since $S$ is quasi-compact, it is bounded. A convex polygon of fixed rank with bounded endpoint that is bounded below is also bounded above. We also note that the proposition is exactly \cite[Prop.~7.4.6]{kedlaya-liu-relative-padichodge}.
\end{proof}

\begin{prop}[{\cite[Cor.~7.4.7]{kedlaya-liu-relative-padichodge}}]\label{prop: dense open on which NP is locconst}
For any $\varphi$-module $M$ over $\widetilde{\mathcal{R}}_A$, there is a dense open subset $\lvert U \rvert$ in $\lvert S \rvert$ such that the Newton polygon function $\NP(M_\beta,\varphi_{M_\beta})$ is locally constant.
\end{prop}
\begin{proof}
This is \cite[Cor.~7.4.7]{kedlaya-liu-relative-padichodge}.
\end{proof}

\subsection{\texorpdfstring{{$\bDhat_\mu$}}{Dmu}-local systems}

In this subsection, we fix a rational number $\mu=d/h$, where $d\in\Z$, $h>0$, and $\gcd(d,h)=1$. When $\mu=0$, we take $d=0$ and $h=1$. 

\begin{defn}\label{defn: phi modules}
A $\varphi$-module over $\widetilde{\mathcal R}_A$ is called \emph{semistable of slope $\mu$} if all its Newton slopes at every point of $\Spa(A,A^+)$ are equal to $\mu$. The zero module is of any slope by convention. A vector bundle $\calE$ over $X_S^\diamondsuit$ is \emph{semistable of slope $\mu$} if, for every affinoid perfectoid space $T\to S$, the corresponding $\varphi$-module over $\widetilde{\mathcal R}_T$ is semistable of slope $\mu$. This agrees with Definition~\ref{defn: vector bundle being semistable of slope lambda}. 
\end{defn}

\begin{rem}\label{rem: phi and phih modules}
With the convention of Remark~\ref{rem: slope of phi^h module convention}, if $(M,\varphi_M)$ is semistable of slope $\mu=d/h$, then $(M,\varphi_M^h)$ is semistable of slope $d$, and $(M,p^d\varphi_M^h)$ is an \'etale $\varphi^h$-module\footnote{In the normalized convention of Kedlaya--Liu, the corresponding slopes are $\mu$ and $0$, respectively.}. Moreover, given $(M,p^d\varphi_M^h)$, one can also recover the original $\varphi$-module $(M,\varphi_M)$ by retaining a $\varphi$-semilinear action $\varphi_M$ whose $h$-th iterate is $p^{-d}$ times the given \'etale $\varphi^h$-action. We note this is observed in \cite{kedlaya-liu-relative-padichodge}.
\end{rem}

\begin{defn}\label{defn: B(lambda)}
Let $B$ be either $\widetilde{\mathcal R}^{\mathrm{bd}}_A$ or $\widetilde{\mathcal R}_A$. Define the finite free $\varphi$-module $B(\mu)$ with basis $e_1,\ldots,e_h$ by
\[
\varphi_{B(\mu)}(e_i)=e_{i+1}\quad(1\leq i<h),\qquad \varphi_{B(\mu)}(e_h)=p^{-d}e_1.
\]
With the convention of Definition~\ref{defn: special and generic Newton polygons}, which agrees with \cite[Defn.~7.2.3]{kedlaya-liu-relative-padichodge}, all Newton slopes of $B(\mu)$ are equal to $\mu$. For any diamond $S$, also define $\calO(\mu)$ to be the vector bundle over $X_S^\diamondsuit$ such that for every affinoid perfectoids $T\to S$, the corresponding $\varphi$-module over $\widetilde{\mathcal R}_T$ is $\widetilde{\mathcal R}_T(\mu)$.
\end{defn}

\begin{defn}\label{defn: simple-isocrystal-Dlambda}
Recall that $\breve{\Q}_p=W(\overline{\F}_p)[1/p]$, and let $\sigma$ denote its Frobenius automorphism. Define the simple isocrystal of slope $\mu$ by
\[
\breve{\Q}_p(\mu)\coloneqq\breve{\Q}_p\langle\varphi\rangle/\breve{\Q}_p\langle\varphi\rangle(\varphi^h-p^d),
\]
where $\varphi a=\sigma(a)\varphi$ and the quotient is taken as a left module over the skew polynomial ring. Define $\breve{\Q}_p(-\mu)$ by replacing $d$ with $-d$, and put
\[
D_\mu\coloneqq\End_{\mathrm{Isoc}(\overline{\F}_p)}\bigl(\breve{\Q}_p(-\mu)\bigr).
\]
Then $D_\mu$ is a central division algebra over $\Q_p$.
\end{defn}

\begin{rem}\label{rem: cyclic-algebra-Dlambda}
The algebra $D_\mu$ admits a more concrete description
\[
D_\mu\cong\bigoplus_{i=0}^{h-1}\Q_{p^h}\cdot\Pi^i,\qquad \Pi a=\sigma(a)\Pi\quad(a\in\Q_{p^h}),\qquad \Pi^h=p^d.
\]
Its center is $\Q_p$, its dimension over $\Q_p$ is $h^2$. Moreover, $D_{-\mu}\cong D_\mu^{\op}$, and the isomorphism class of $D_\mu$ depends only on the class of $\mu$ in $\Q/\Z$.
\end{rem}

\begin{thmdefn}\label{thm/defn: (d,h)-locsys and Dlambdamod}
A \emph{$(d,h)$-$\Q_p$-vector space} is a finite-dimensional $\Q_{p^h}$-vector space $V$ equipped with a $\sigma$-semilinear automorphism $\varphi_V$ satisfying
\[
p^d\varphi_V^h=\mathrm{id}_V.
\]
Morphisms are $\Q_{p^h}$-linear maps commuting with $\varphi_V$. This category is equivalent to the category $\Isoc_{\overline{\F}_p}^{-\mu}$ of isocrystals over $\overline{\F}_p$ of slope $-\mu$, and to the category $\mathrm{FPMod}_{D_\mu^{\op}}$ of finite projective right $D_\mu$-modules.
\end{thmdefn}
\begin{proof}
Give $V$ a right $D_\mu$-action by
\[
v\cdot a=av\quad(a\in\Q_{p^h}),\qquad v\cdot\Pi=\varphi_V^{-1}(v).
\]
The identities $\varphi_V(av)=\sigma(a)\varphi_V(v)$ and $\varphi_V^h=p^{-d}$ give the defining relations of $D_\mu$. Conversely, a right $D_\mu$-module becomes a $(d,h)$-$\Q_p$-vector space by setting $av=v\cdot a$ for $a \in \Q_{p^h}$, and $\varphi_V(v)=v\cdot\Pi^{-1}$. These constructions are inverse to each other.

By the Dieudonn\'e--Manin classification, every isocrystal over $\overline{\F}_p$ of slope $-\mu$ is a finite direct sum of copies of $\breve{\Q}_p(-\mu)$. Thus
\begin{equation}\label{eq: morita field case}
\Isoc_{\overline{\F}_p}^{-\mu}\xrightarrow{\sim}\mathrm{FPMod}_{D_\mu^{\op}},\qquad N\longmapsto\Hom\bigl(\breve{\Q}_p(-\mu),N\bigr),
\end{equation}
where the right action is by precomposition. Its quasi-inverse is $M\mapsto M\otimes_{D_\mu}\breve{\Q}_p(-\mu)$. Every finite projective module over $D_\mu$ is finite free, since $D_\mu$ is a division algebra, cf. \cite[Ex.~I.2.1.1]{WeibelKbook}.
\end{proof}

\begin{rem}
As we mentioned the isomorphism class of $D_\mu$ depends only on $\mu\bmod\Z$, while the equivalence~\eqref{eq: morita field case} depends on the fix the rational number $\mu$, or more precisely the choice of the simple isocrystal $\breve{\Q}_p(-\mu)$. 
\end{rem}

For the definitions of local systems, let $S$ be a diamond and let $\tau\in\{\qproet,\vv\}$. The same definitions apply on $S_\proet$ when $S$ is an analytic adic space.

\begin{defn}[{\cite[Defn.~8.5.7]{kedlaya-liu-relative-padichodge}}]
\begin{enumerate}
\item An \emph{isogeny $(d,h)$-$\Z_p$-local system} on $S_\tau$ is an isogeny $\Z_{p^h}$-local system $\bL$ equipped with a $\sigma$-semilinear automorphism $\varphi_{\bL}$ satisfying
\[
p^d\varphi_{\bL}^h=\mathrm{id}_{\bL}.
\]
\item A \emph{$(d,h)$-$\Q_p$-local system} on $S_\tau$ is a $\Q_{p^h}$-local system $\bL$ equipped with a $\sigma$-semilinear automorphism $\varphi_{\bL}$ satisfying
\[
p^d\varphi_{\bL}^h=\mathrm{id}_{\bL}.
\]
Here $\sigma$ acts on the constant coefficient sheaf $\underline{\Q}_{p^h}$. Morphisms in both categories are morphisms of the underlying local systems commuting with $\varphi_{\bL}$.
\item Let $\bDhat_\mu$ denote the presheaf of rings associated with the algebra $D_\mu$ with its $p$-adic topology. On every affinoid perfectoids, it is given by
\[
\bDhat_\mu(T)=\mathrm{Cont}(\lvert T\rvert,D_\mu).
\]
This defines a sheaf by \cite[Prop.~2.9]{LeBrasBC}. Note that a choice of a $\Q_p$-basis of $D_\mu$ identifies $\bDhat_\mu$ with $\underline{\Q}_p^{h^2}$ as abelian sheaves.
\item A \emph{$\bDhat_\mu^{\op}$-local system} on $S_\tau$ is a sheaf $\bV$ of right $\bDhat_\mu$-modules which is locally in the $\tau$-topology isomorphic to $M\otimes_{D_\mu}\bDhat_\mu$ for a finite projective right $D_\mu$-module $M$. Denote this category by $\Loc_{\bDhat_\mu^{\op}}(S_\tau)$. Forgetting the $\bDhat_\mu$-structure defines a $\Q_p$-local system, we denote it by $\bV_{\underline{\Q}_p}$.
\end{enumerate}
\end{defn}

\begin{rem}
When $S=\Spa(K)$ for a $p$-adic field $K$ containing $\Q_{p^h}$, our $(d,h)$-$\Q_p$-local systems agree with Berger's category $\mathrm{Rep}(-d,h)$, cf. \cite[Defn.~3.2.1]{BergerConstructionphigamma}. Over a $p$-adic field, $\bDhat_{\mu}^{\op}$-local systems were studied in \cite[Thm.~10.1.7]{FarguesFontaineAsterisque} as Galois representations. Kedlaya--Liu also considered $(d,h)$-$\Q_p$-local systems over perfectoid spaces in \cite[Defn.~8.5.7]{kedlaya-liu-relative-padichodge}. 
\end{rem}

For vector bundles $\calE,\calF$ on $X_S^\diamondsuit$, we denote by $\underline{\Hom}(\calE,\calF)$ the sheaf on $S_\vv$ defined by for each perfectoid affinoids $T\to S$
\[
\underline{\Hom}(\calE,\calF)(T)
=
\Hom_{\calO_{X_T}}(\calE|_{X_T},\calF|_{X_T}).
\]
This is a $\vv$-sheaf by \cite[Prop.~II.2.1]{FarguesScholze}. We write $\underline{\End}(\calE)=\underline{\Hom}(\calE,\calE)$.

\begin{thm}\label{thm: Dlambda locsys and ss vb}
Let $S$ be a locally spatial diamond equipped with a morphism $S\to\Spd\Q_{p^h}$. The following categories are equivalent:
\begin{enumerate}
\item $\Bun(X_S^\diamondsuit)^{ss,\mu}$;
\item the category of $(d,h)$-$\Q_p$-local systems on $S_\vv$;
\item $\Loc_{\bDhat_\mu^{\op}}(S_\vv)$.
\end{enumerate}
Recall $\calO(\mu)$ be as Definition~\ref{defn: B(lambda)}. The equivalence between (1) and (3) is given by
\[
\calE\longmapsto\underline{\Hom}(\calO(\mu),\calE),\qquad \bV\longmapsto\bV\otimes_{\bDhat_\mu}\calO(\mu),
\]
where the right $\bDhat_\mu$-action on $\underline{\Hom}$ is by precomposition. These equivalences commute with pullback preserving the prescribed $\Q_{p^h}$-structure.
\end{thm}
\begin{proof}
On each affinoid perfectoid perfectoids $T\to S \to \Q_{p^h}$, apply \cite[Thm.~8.5.12 and Cor.~8.5.13]{kedlaya-liu-relative-padichodge} to the prescribed untilt $T^\sharp$ over $\Q_{p^h}$. This identifies semistable bundles of slope $\mu$ with $(d,h)$-$\Q_p$-local systems. These commute with pullback. Since both categories satisfy $\vv$-descent by \cite[Prop.~II.2.1]{FarguesScholze} the our definition of $(d,h)$-$\Q_p$-local systems, this proves the equivalence of (1) and (2).

The equivalence between (2) and (3) can be viewed as the sheaf version of Theorem~\ref{thm/defn: (d,h)-locsys and Dlambdamod}. The same construction works over perfectoid affinoids that trivialize the objects in the corresponding categories. This provides the equivalence between (2) and (3). We also note that for this equivalence, the structure map $S\to\Spd\Q_{p^h}$ is not needed.
\end{proof}

\begin{rem}\label{rem: O(lambda) to Dlambda explicit functor}
Let $\bV$ be the $(d,h)$-$\Q_p$-local system corresponding to $\calE$ in $\Loc_{\bDhat_\mu^{\op}}(S_\vv)$. The full faithfulness gives
\[
\underline{\Hom}(\calO(\mu),\calE)
\simeq
\underline{\Hom}_{\bDhat_\mu}(\bDhat_\mu,\bV)
\xrightarrow{\sim}\bV,
\]
where the last map is given by evaluation at $1$. In particular,
\[
\underline{\End}(\calO(\mu))
\simeq
\underline{\End}_{\bDhat_\mu}(\bDhat_\mu)
\simeq\bDhat_\mu,
\]
with the last identification given by left multiplication. Under these identifications, precomposition agrees with the right $\bDhat_\mu$-action on $\bV$. The evaluation map
\[
\underline{\Hom}(\calO(\mu),\calE)\otimes_{\bDhat_\mu}\calO(\mu)
\longrightarrow\calE
\]
is an isomorphism. This can be checked $\vv$-locally, where $\bV\simeq\bDhat_\mu^m$ and $\calE\simeq\calO(\mu)^{\oplus m}$.
\end{rem}

Now we fix a morphism $S\to\Spd K$ such that $K$ contains $\Q_{p^h}$. As discussed around Proposition~\ref{prop: tau vector bundles}, it defines the sheaf $\hat{\calO}_{S^\sharp}$. Moreover, the same construction defines a functor $\calE\mapsto\calE|_{S^\sharp}$ from vector bundles on $X_S^\diamondsuit$ to $\vv$-vector bundles over $S$. The following can be viewed as the sheaf version of a key claim used in \cite[\S10.6.4]{FarguesFontaineAsterisque}, which relies on a direct application of Morita equivalence.

\begin{lem}[A Morita equivalence]\label{lem: morita}
Retain the hypotheses of Theorem~\ref{thm: Dlambda locsys and ss vb}, and fix $S\to\Spd K\to\Spd\Q_{p^h}$ as above. Let $\calE$ be a vector bundle over $X_S^\diamondsuit$ semistable of slope $\mu$, and let $\bV=\underline{\Hom}(\calO(\mu),\calE)$ be the corresponding $\bDhat_\mu^{\op}$-local system. There is a canonical isomorphism
\[
\bV_{\underline{\Q}_p}\otimes_{\underline{\Q}_p}\hat{\calO}_{S^\sharp}
\simeq
\calE|_{S^\sharp}\otimes_{\hat{\calO}_{S^\sharp}}
\bigl(\calO(\mu)|_{S^\sharp}\bigr)^\vee.
\]
In particular, naturally in $\calE$ and $\bV$, there is an isomorphism
\[
\bV_{\underline{\Q}_p}\otimes_{\underline{\Q}_p}\hat{\calO}_{S^\sharp}
\simeq
\calE|_{S^\sharp}\otimes_{\hat{\calO}_{S^\sharp}}
\hat{\calO}_{S^\sharp}^{\oplus h}.
\]
\end{lem}
\begin{proof}
Restricting to the Cartier divisor defined by the $S \to \Spd K$ gives
\begin{equation}\label{eq: basic case}
\bDhat_\mu\otimes_{\underline{\Q}_p}\hat{\calO}_{S^\sharp}
\longrightarrow
\underline{\End}_{\hat{\calO}_{S^\sharp}}
\bigl(\calO(\mu)|_{S^\sharp}\bigr),
\end{equation}
where the target denotes the sheaf of $\hat{\calO}_{S^\sharp}$-linear endomorphisms. This map is an isomorphism. Indeed, both sides are finite locally free of rank $h^2$. On each geometric fibers, the map is an algebra map between central simple algebras, in particular, the above map is also injective. Therefore \eqref{eq: basic case} is an isomorphism.

Consider the canonical restriction map
\[
\bV_{\underline{\Q}_p}\otimes_{\underline{\Q}_p}\hat{\calO}_{S^\sharp}
\longrightarrow
\underline{\Hom}_{\hat{\calO}_{S^\sharp}}
\bigl(\calO(\mu)|_{S^\sharp},\calE|_{S^\sharp}\bigr),
\qquad
f\otimes a\longmapsto a\,f|_{S^\sharp}.
\]
By Theorem~\ref{thm: Dlambda locsys and ss vb}, $\calE$ is $\vv$-locally a finite direct sum of copies of $\calO(\mu)$. On such a cover, this map is a finite direct sum of copies of the isomorphism in \eqref{eq: basic case}. Therefore the above morphism is also an isomorphism. Since $\calO(\mu)|_{S^\sharp}$ is finite locally free, there is a canonical isomorphism
\[
\underline{\Hom}_{\hat{\calO}_{S^\sharp}}
\bigl(\calO(\mu)|_{S^\sharp},\calE|_{S^\sharp}\bigr)
\simeq
\calE|_{S^\sharp}\otimes_{\hat{\calO}_{S^\sharp}}
\bigl(\calO(\mu)|_{S^\sharp}\bigr)^\vee.
\]
This proves the first assertion, which can also be proved via the usual Morita equalavence, cf. \cite[Cor.~22.5 and Cor.~22.6]{AndersonFuller}. For the second assertion, use $\calO(\mu)|_{T^\sharp}\simeq\calO_{T^\sharp}^{\oplus h}.$
\end{proof}

We end this subsection with the following remark.

\begin{rem}
The equivalence between \textup{(1)} and \textup{(3)} in Theorem~\ref{thm: Dlambda locsys and ss vb} admits a formulation without given the $\Q_{p^h}$-structure. Let $S$ be an arbitrary locally spatial diamond. Definition~\ref{defn: B(lambda)} defines a vector bundle $\calO_S(\mu)$ over $X_S^\diamondsuit$ by $\vv$-descent. Put
\[
\mathscr D_{\mu,S}
  \coloneqq\underline{\End}\bigl(\calO_S(\mu)\bigr),
\]
in other words, for perfectoid test objects $T\to S$, put $\mathscr D_{\mu,S}(T)
  =\End_{X_T}\bigl(\calO_T(\mu)\bigr)$. This is a sheaf of rings with multiplication given by composition.

Define $\Loc_{\mathscr D_{\mu,S}^{\op}}(S_\vv)$ to be the category of right $\mathscr D_{\mu,S}$-modules which are $\vv$-locally isomorphic to
$\mathscr D_{\mu,S}^{\oplus m}$ for some finite $m$.
Then similar to the proof of Theorem~\ref{thm: Dlambda locsys and ss vb}(essentially due to \cite[Thm.~II.2.19(ii)]{FarguesScholze}), there is an equivalence
\[
\Bun(X_S^\diamondsuit)^{ss,\mu}
  \simeq
\Loc_{\mathscr D_{\mu,S}^{\op}}(S_\vv),
\]
given by
\[
\calE\longmapsto
  \underline{\Hom}\bigl(\calO_S(\mu),\calE\bigr),
\qquad
\bV\longmapsto
  \bV\otimes_{\mathscr D_{\mu,S}}\calO_S(\mu).
\]

This recovers Theorem~\ref{thm: Dlambda locsys and ss vb}, since the restriction of $\mathscr D_{\mu,S}$ over $S\times_{\Spd\F_p}\Spd\F_{p^h}$ is $\bDhat_\mu$. Moreover, in general, the trivialization $\bDhat_\mu$ of $\mathscr D_{\mu,S}$ carries a nontrivial descent datum. The Frobenius in $\Gal(\F_{p^h}/\F_p)$ acts on $\bDhat_\mu(T)$ via
\[
x \mapsto \Pi x \Pi^{-1},
\]
where we use $\Pi$ to denote the constant map with value $\Pi$,  in Remark~\ref{rem: cyclic-algebra-Dlambda}, viewed as an element in $\bDhat_\mu(T)$. Note that the above formula already appear right after \cite[Defn.~9.2.1]{FarguesFontaineAsterisque}.
\end{rem}

\subsection{Slope filtrations}\label{sec: slope filtrations}

In this subsection, we review the abstract Harder--Narasimhan theory for $\varphi$-modules and establish the existence of a global HN filtration when the Newton polygon function is constant.

\begin{defn}[Semistable filtration]\label{defn: semistable and HN filtration}
Let $M$ be a nonzero $\varphi$-module over a ring $B$ as in Definition~\ref{defn: phi modules}. A \emph{semistable slope filtration} of $M$ is a sequence of saturated $\varphi$-submodules $\{M_i\}_{i=0}^{\ell}$
\[
0=M_0 \subsetneq M_1 \subsetneq \cdots \subsetneq M_\ell = M
\]
such that each graded piece $M_i/M_{i-1}$ is \emph{semistable}. A semistable filtration $\{M_i\}$ is called a Harder--Narasimhan filtration (HN filtration for short) if 
\begin{equation}\label{eq: semistable filtration}
\mu(M_1/M_0) > \mu(M_2/M_1) > \cdots > \mu(M_\ell/M_{\ell-1}).
\end{equation}
For any semistable filtration $\{M_i\}$, one constructs a multiset comprised of the values $\{\mu(M_{i}/M_{i-1})\}$, where each $\mu(M_{i}/M_{i-1})$ appears with multiplicity $\mathrm{rank}(M_{i}/M_{i-1})$. The Newton polygon formed by this multiset is denoted $\NP(\{M_i\}_{i=0}^\ell)$, the Newton polygon of the semistable filtration $\{M_i\}$. 
\end{defn}

\begin{rem}
Our definition of slope follows the convention of \cite{kedlaya-liu-relative-padichodge}, which differs by a sign from the convention in \cite{Kedlayaslopefiltrationsrevisited}. However, we define the Newton polygon as a convex function (as in Definition~\ref{defn: many NP functions}), meaning the slopes of the Newton polygon are arranged in increasing order to ensure convexity. Consequently, in Definition~\ref{defn: semistable and HN filtration}, we reverse the direction of the inequalities in Eq.\eqref{eq: semistable filtration}. The partial order on the corresponding Newton polygons remains unchanged.
\end{rem}

Assume $R=L$ is an analytic field (possibly with a trivial norm in this subsection). We recall the following result regarding semistable filtrations from \cite{Kedlayaslopefiltrationsrevisited}.

\begin{thm}[Harder--Narasimhan filtration]\label{thm: HN-filtration}
\begin{enumerate}
	\item Any $\varphi$-module over $\widetilde{\mathcal{R}}_L$ admits a unique HN filtration.
	\item For any semistable filtration $\{M_i\}$ of a $\varphi$-module $M$ over $\widetilde{\mathcal{R}}_L$, we have 
	\[
\NP(\{M_i\})\preceq\NP(M).
\]
\end{enumerate}
\end{thm}
\begin{proof}
Claim (1) is \cite[Prop.~5.10]{Kedlayapadiclocalmonodromy} or \cite[Prop.~3.5.4]{Kedlayaslopefiltrationsrevisited}. Claim (2) is \cite[Prop.~3.5.4]{Kedlayaslopefiltrationsrevisited}.
\end{proof}

When $S$ is a locally spatial diamond, one can formulate an analogous statement for vector bundles over the diamantine Fargues--Fontaine curve $X_S^\diamondsuit$.

\begin{thm}[Global slope filtration]\label{thm: global-slope-filtration}
Let $\calE$ be a vector bundle over the diamantine Fargues--Fontaine curve $X_S^\diamondsuit$.
Assume the Newton polygon function $\NP(\calE)$ is constant on $\lvert S \rvert$.
Then there exists a unique filtration
\[
0=\calE_0 \subsetneq \calE_1 \subsetneq \cdots \subsetneq \calE_\ell = \calE
\]
by saturated subbundles such that each graded piece $\calE_i/\calE_{i-1}$ is semistable of slope $\mu_i$ and $\mu_1>\mu_2>\cdots>\mu_\ell.$ Moreover, $\NP(\{\calE_i\})=\NP(\calE)$.
\end{thm} 

\begin{proof}
For every affinoid perfectoid $T\to S$, the assertion follows from \cite[Thm.~7.4.9]{kedlaya-liu-relative-padichodge}. See also \cite[Thm.~II.2.19]{FarguesScholze}. The filtration is unique, hence compatible with pullback. It therefore descends along $\vv$-covers by \cite[Prop.~II.2.1]{FarguesScholze}. So the above results also extend to the locally spatial diamond case considered here.
\end{proof}

\section{A relative \texorpdfstring{$p$}{p}-adic monodromy theorem}\label{sec: proof of RpMT}

In this section, we establish several of the relative $p$-adic monodromy theorems mentioned in the introduction. A key result is that the local constancy of the Newton polygon function associated with a de Rham local system implies both a local and a global relative $p$-adic monodromy theorem.

\subsection{Reviewing the Fargues--Fontaine proof over classical points}\label{sec: FF's proof of pMT}

Let $K$ be a $p$-adic field, and let $\C_p$ be the completion of an algebraic closure of $K$. In this subsection, $X=X_{C}$ denotes the absolute Fargues--Fontaine curve defined by the tilt $C$ of $\C_p$. Note that the structure sheaf $\calO_X$ admits a continuous action by $G_K$.

Here, by the classical $p$-adic monodromy theorem ($p$MT), we refer to the theorem of Berger, which states that every de Rham $p$-adic Galois representation is potentially log-crystalline (i.e., potentially semi-stable). Fargues and Fontaine restate and generalize this theorem to the setting of $G_K$-equivariant vector bundles on the Fargues--Fontaine curve $X$ \cite{FarguesFontaineAsterisque}.

By interpreting $G_K$-equivariant vector bundles on the Fargues--Fontaine curve $X$ via $B$-pairs, their formulation aligns precisely with Berger's definitions. In particular, for a $G_K$-equivariant vector bundle on $X$, the de Rham property is controlled by its $B_{\dR}^+$-part, and the property of being (potentially) log-crystalline is determined by the $B_{e}$-part.

We review the strategy employed by Fargues--Fontaine in \cite[\S10.6]{FarguesFontaineAsterisque}. The most critical results extracted from their proof are Lemma~\ref{lem: Hyodolem} and Theorem~\ref{thm: key from FF}.

Let $V$ be a de Rham representation of $G_K$ on a $\Q_p$-vector space, and let $\calE_1=V\otimes_{\Q_p} \calO_X$ be the corresponding $G_K$-equivariant vector bundle on the Fargues--Fontaine curve $X$. The complete stalk $M_1 \coloneqq \calE_{1,\infty}^\wedge$ of $\calE_1$ at the point $\infty$ corresponding to the untilt $\C_p$ is isomorphic to $V\otimes B_{\dR}^+$. Let $\calE_0$ be the modification of $\calE_1$ by $M_0 \coloneqq D_{\dR}(V)\otimes B_{\dR}^+$ at $\infty$. Since $M_0$ is $G_K$-stable inside $M_1[1/t]$, $\calE_0$ is another $G_K$-equivariant vector bundle on $X$. Furthermore, the stalk $\calE_{0,\infty}^\wedge$, viewed as a $B_{\dR}^+$-representation of $G_K$, is $B_{\dR}^+$-flat in the following sense. (We note that this is the admissible modification associated with the de Rham local system over $\Spd K$ constructed in \S\ref{sec: Shtukas}).

\begin{defn}[Fargues--Fontaine]
A $B_{\dR}^+$-representation of $G_K$ is a finite free $B_{\dR}^+$-module equipped with a semilinear action of $G_K$ that is continuous with respect to the canonical topology of $B_{\dR}^+$. A $B_{\dR}^+$-representation $M$ of $G_K$ is \emph{generically flat} if $\dim_K (M\otimes_{B_{\dR}^+} B_{\dR})^{G_K} = \mathrm{rk}_{B_{\dR}^+} M$. A $B_{\dR}^+$-representation $M$ of $G_K$ is \emph{flat} if it is generically flat and admits a $G_K$-equivariant isomorphism
\[
M \cong (M\otimes_{B_{\dR}^+} B_{\dR})^{G_K} \otimes_K B_{\dR}^+.
\] 
\end{defn}

A key observation by Fargues and Fontaine is the following lemma.

\begin{lem}\label{lem: Hyodolem}
Suppose we are given an exact sequence of $G_K$-equivariant vector bundles on the Fargues--Fontaine curve $X$:
	\[
	0 \to \calE \to \calF \to \calG \to 0.
	\]
Assume that the complete stalks $\calE_\infty$, $\calF_\infty$ and $\calG_\infty$ are flat $B_{\dR}^+$-representations of $G_K$, that the slopes of $\calE$ are strictly greater than the slopes of $\calG$, and that both $\calE$ and $\calG$ are log-crystalline. Then $\calF$ is also log-crystalline. 
\end{lem}
\begin{proof}
This is exactly \cite[\S10.6.5.2]{FarguesFontaineAsterisque}.  The claim is established by comparing extension classes on both sides. Note that log-crystalline $G_K$-equivariant vector bundles on the Fargues--Fontaine curve $X$ with flat stalks at $\infty$ are equivalent to the category of $(\varphi,N)$-modules via \cite[Prop.~10.3.3]{FarguesFontaineAsterisque}. Therefore, the statement reduces to the computation of extension classes. Using internal Homs, the claim effectively reduces to \cite[Prop.~10.6.14]{FarguesFontaineAsterisque}.
\end{proof}

\begin{lem}\label{lem: ext of BdR+}
\begin{enumerate}
	\item A $B_{\dR}^+$-representation $M$ of $G_K$ is flat if and only if $\overline{M}\coloneqq M\otimes_{B_{\dR}^+} \C_p$, considered as a $\C_p$-representation of $G_K$, is Hodge--Tate with all weights equal to $0$.
	\item Flat $B_{\dR}^+$-representations of $G_K$ are stable under taking saturated subquotients.
\end{enumerate}
\end{lem}

\begin{proof}
Claim (1) is \cite[Prop.~2.18]{HengduABKFM}. Claim (2) follows directly from (1). (We note that (2) can also be obtained independently using \cite[Prop.~10.4.4]{FarguesFontaineAsterisque}(2), although there is a typo in \textit{loc.\ cit.} where $\otimes$ should be $\oplus$. Connecting it to Hodge--Tate $\C_p$-representations facilitates the generalization to the relative case). 
\end{proof}
 
To apply Lemma~\ref{lem: Hyodolem}, Fargues and Fontaine utilize the HN filtration theory on $\calE_0$, which furnishes a filtration $0=\calF_0 \subsetneq \calF_1 \cdots \calF_r = \calF$ on $\calF\coloneqq \calE_0$. Because the HN filtration is canonical, the $\calF_i$ are $G_K$-equivariant \emph{sub-bundles}. Moreover, since the $B_e$-parts of $\calE_0$ and $\calE_1$ are identical by design, it suffices to show that $\calE_0$ is potentially log-crystalline. Let $\calG_i\coloneqq \calF_{i}/\calF_{i-1}$ for $i=1,\ldots, r$ denote the HN constituents of $\calF\coloneqq \calE_0$.

This logic yields a slightly refined version of the $p$MT.

\begin{thm}[Fargues--Fontaine]\label{thm: key from FF}
Every $G_K$-equivariant vector bundle $\calE$ on $X$ whose complete stalk $\calE_{\infty}$ is $B_{\dR}^+$-flat as a $B_{\dR}^+$-representation of $G_K$ is potentially log-crystalline. More explicitly:
\begin{enumerate}
	\item If a $G_K$-equivariant vector bundle $\calE$ on $X$ is semistable as a vector bundle over $X$ and $\calE_{\infty}$ is $B_{\dR}^+$-flat, then $\calE$ is potentially log-crystalline.
	\item For every $G_K$-equivariant vector bundle $\calE$ on $X$ with a $B_{\dR}^+$-flat stalk $\calE_{\infty}$, if all HN constituents of $\calE$ are log-crystalline as $G_K$-equivariant vector bundles, then $\calE$ itself is log-crystalline as a $G_K$-equivariant vector bundle.
\end{enumerate}
\end{thm}

\subsection{The Newton partition defined by a de Rham local system}

Let $\bL$ be a de Rham local system on a smooth rigid analytic space $S$ defined over a $p$-adic field $K$. In \S\ref{sec: Shtukas}, we demonstrated that $\bL$ defines an admissible modification of vector bundles $\calE_1 \to \calE_0$ over $X_S^\diamondsuit$ along $S^\diamondsuit \to \Spd K$. Consequently, $\bL$ defines a Newton polygon function $\NP(\bL)$ on $\lvert S \rvert$ via the relation $\NP(\bL)=\NP(\calE_0)$.

Motivated by \cite[\S3.5]{CaraianiScholzecpt}, we define the Newton partition on $\lvert S \rvert$ via the HN polygons of $\calE_0$.

\begin{defn}\label{defn: newton partitions}
We define the Newton partition on the underlying topological space $\lvert S \rvert$ of $S$ as the disjoint union
\[
\lvert S \rvert = \coprod_{\calP} \lvert S \rvert^{\calP},
\]
indexed by Newton polygons $\calP$ of length equal to the rank of $\bL$. Explicitly, a point $s \in \lvert S \rvert$ belongs to $\lvert S \rvert^{\calP}$ if and only if $\NP(\calE_0)(s)=\calP$.
\end{defn}

\begin{rem}
As demonstrated in Proposition~\ref{prop: properties of |S|P}, the Newton partition is a valid topological partition of $\lvert S \rvert$ in the sense of \cite[\href{https://stacks.math.columbia.edu/tag/09XZ}{Tag~09XZ}]{stacks-project}. However, unlike the setting in \cite{CaraianiScholzecpt}, it is unknown whether it defines a stratification in the strict sense of \cite[\href{https://stacks.math.columbia.edu/tag/09Y1}{Tag~09Y1}]{stacks-project} or \cite{Shatzfamilyofvb}. 
\end{rem}

The principal result of this subsection is the following proposition.

\begin{prop}\label{prop: properties of |S|P}
The subsets $\lvert S \rvert^{\calP} \subset \lvert S \rvert$ have the following properties:
\begin{enumerate}
    \item \textbf{Local closedness.}
    Each $\lvert S \rvert^{\calP}$ is a locally closed subset of $\lvert S \rvert$.
    
    \item \textbf{Local finiteness.}
    Rational locally on $S$, there are only finitely many $\calP$ for which $\lvert S \rvert^{\calP}\neq \emptyset$.
    
    \item \textbf{Finiteness.}
    If $S$ is connected, there are only finitely many $\calP$ such that $\lvert S \rvert^{\calP}\neq \emptyset$.
    
    \item \textbf{$\CNP_1 \Longrightarrow$ Nonempty interior.}
    Assume that $\NP(\bL)$ is locally constant around rank-$1$ points. If $\lvert S \rvert^{\calP}\neq \emptyset$, then its topological interior $\lvert S \rvert^{\calP,\circ}$ is nonempty. Moreover, every rank-$1$ point of $\lvert S \rvert^{\calP}$ is contained in $\lvert S \rvert^{\calP,\circ}$.
\end{enumerate}
\end{prop}
\begin{proof}
Claims (1) and (2) can be verified \'etale locally on $S$, allowing us to assume $S$ is affinoid and admits a chart. In this setting, by Lemma~\ref{lem: S as G-torsor quot}, there is a pro-\'etale covering of $S$ by an affinoid perfectoid $S_\infty$ such that $\lvert S \rvert = \lvert S_\infty \rvert / G$ for some profinite group $G$. By the construction of the Newton polygon function, the parts of the Newton partitions in $\lvert S_\infty \rvert$ are $G$-stable. Because the quotient map sends saturated open/closed sets to open/closed sets (cf.\ Proposition~\ref{prop: useful top result}), claims (1) and (2) reduce to Theorem~\ref{thm: KL-semi-continuity}. 

For claim (3), if $S$ is connected, the Hodge--Tate weights of $\bL$ are constant on $S$ (cf.\ Remark~\ref{rem: local constancy of HT wts} and \cite[Thm.~1.1]{Shimizu-HT}). It follows that $\lvert S \rvert^{\calP}\neq \emptyset$ is only possible for $\calP$ belonging to the finite set of Newton polygons bounded by $\mu$, that is, inside the Kottwitz set $B(GL_r,\mu)$, where $r$ is the rank of $\bL$ and $\mu$ is the cocharacter defined by the Hodge--Tate weights. For the finiteness of $B(GL_r,\mu)$ and this deduction, we refer to \cite[\S6.2 \& 6.4]{KottwitzIsocII}. 

For claim (4), suppose $s \in \lvert S \rvert^{\calP}$. Then its maximal rank-$1$ generalization $s_{\max}$ is also inside $\lvert S \rvert^{\calP}$. $\CNP_1$ means there exists an open neighborhood of $s_{\max}$ fully contained in $\lvert S \rvert^{\calP}$.
\end{proof}

The following proposition is used in the proof.

\begin{prop}\label{prop: useful top result}
Let $X$ be a topological space equipped with a group action by $G$. Let $Y = X/G$ be the orbit space endowed with the quotient topology, and let $\pi\colon X \to Y$ be the canonical projection. If $Z \subset X$ is a $G$-saturated subset (i.e., $Z$ is $G$-stable, meaning $\pi^{-1}(\pi(Z)) = Z$), then $Z$ is open in $X$ if and only if $\pi(Z)$ is open in $Y$. Similarly, $Z$ is closed in $X$ if and only if $\pi(Z)$ is closed in $Y$.
\end{prop}
\begin{proof}
By the definition of the quotient topology, a subset $W \subset Y$ is open (resp.\ closed) if and only if its preimage $\pi^{-1}(W)$ is open (resp.\ closed) in $X$. Let $W = \pi(Z)$. Since $Z$ is $G$-saturated, $\pi^{-1}(W) = \pi^{-1}(\pi(Z)) = Z$. Therefore, $\pi(Z)$ is open (resp.\ closed) in $Y$ if and only if $Z$ is open (resp.\ closed) in $X$.
\end{proof}

\subsection{End of the proof of Theorem~\ref{thm:main_rpmt1}}

Recall in \S\ref{RpMT1 to CNP1}, we have shown $\RpMT_1 \to \CNP_1$. To prove the other direction, it is enough to show the following local version of Theorem~\ref{thm: CNP to RpMT}.

\begin{thm}[\,$\CNP\Longrightarrow \text{local }\RpMT$\,]\label{thm: constant case}
Let $S$ be a smooth rigid analytic space over a $p$-adic field $K$, and let $\bL$ be a de Rham $\Q_p$-local system on $S$. Let $\{\calE_1  \dashrightarrow \calE_0\}$ be the admissible modification of vector bundles on $X_S^\diamondsuit$ along $S^\diamondsuit \to \Spd K$. If the Newton polygon function $\NP(\calE_0)$ is constant on $\lvert S \rvert$, then for any $s \in \lvert S \rvert$, there is a rational open subset $V \subset S$ containing $s$ and a finite \'etale covering $V' \to V$, such that $\bL |_{V'}$ is log-crystalline at all classical points. 
\end{thm}

\begin{proof}
Under the assumption of constancy, Theorem~\ref{thm: global-slope-filtration}
 guarantees the existence of a global HN filtration on $\calF \coloneqq \calE_0$ as follows
\[
0=\calF_0 \subsetneq \calF_1 \subsetneq \cdots \subsetneq \calF_r = \calF.
\]	
Let $\calG_i \coloneqq \calF_i/\calF_{i-1}$ denote the HN constituents. Because the HN filtration is unique, for each classical point $s \in S$, $\calG_i|_{X_s^\diamondsuit}$ are precisely the HN constituents of $\calF|_{X_s^\diamondsuit}$. Based on the discussion in \S\ref{sec: FF's proof of pMT} (specifically Lemma~\ref{lem: Hyodolem}), it suffices to show that for each $i=1,\ldots, r$, there is a rational open $V \subset S$ containing $s$ and a finite \'etale covering $V' \to V$ such that $\calG_i|_{V'}$ is log-crystalline at all classical points. 

Up to a finite base field extension, we may apply Theorem~\ref{thm: Dlambda locsys and ss vb} and Remark~\ref{rem: O(lambda) to Dlambda explicit functor}, and write
\[
\calG_i \cong \bV_i \otimes_{\bDhat_{\lambda}} \calO(\lambda)
\]
for some $\bDhat_{\lambda}$-local system $\bV_i$.
	
Let $\bV_{\underline{\Q}_p}$ be the underlying $\underline{\Q}_p$-local system of $\bV_i$. The restriction $\calG_i|_{S^\sharp}$ of $\calG_i$ along the Cartier divisor defined by $S^\sharp \to \Spd K$ is a subquotient of a $\vv$-vector bundle originating from an \'etale vector bundle over $S^\diamondsuit$ (cf.\ Remark~\ref{rem: M0 and M_1 for de Rham as proet vb} and Remark~\ref{rem: extend BdR and HT to v loc sys}). By applying Lemma~\ref{lem: morita} and Proposition~\ref{prop: Condition for L to be HT wt 0}, we deduce that $\bV_{\underline{\Q}_p}$ is a Hodge--Tate local system of weight $0$.

For any $s \in \lvert S \rvert$, Theorem~\ref{thm: equiv of all kinds of locsys}(3b) allows us to find a rational open subset $U$ containing $s$ such that $\bV_{\underline{\Q}_p}|_{U}$ is an isogeny $\Z_p$-local system. Subsequently, Theorem~\ref{thm: RpMT for HT0} ensures the existence of a further finite \'etale covering $V'$ of $U$ over which $\bV_{\underline{\Q}_p}|_{V'}$ is unramified at all classical points. From the discussions in \S\ref{sec: FF's proof of pMT}, for each classical point $s' \in V'$, the restriction $\calG_i|_{X_{s'}}$ is log-crystalline as a $\Gal_{\kappa(s)}$-equivariant $\calO_{X_{\overline{s}}}$-representation, which is also $B_{\dR}^+$-flat. 

Hence all HN constituents $\calG_i|_{X_{s'}}$ are log-crystalline and $B_{\dR}^+$-flat, Theorem~\ref{thm: key from FF}(2) implies that the bundle $\calE_0|_{X_{s'}}$ is log-crystalline. Therefore, $\bL|_{V'}$ is log-crystalline at all classical points.
\end{proof}

\begin{rem}\label{rem: on conjecture of HoweKlevdalIII}
A special case of Theorem~\ref{thm: CNP to RpMT} occurs when $\NP(\bL)$ consists of a single $\lambda$ with multiplicity. In this case, it is immediate that $\bL|_{V'}$ is pointwise crystalline, as the monodromy operators change the Frobenius slope by $1$. 
\end{rem}

The hypothesis in the preceding theorem is always satisfied over a dense open subset. 

\begin{prop}\label{prop: CNP is dense open}
Let $S$ be a smooth rigid analytic space over a $p$-adic field $K$, and let $\bL$ be a de Rham $\Q_p$-local system on $S$. There exists a dense open subset $U \subset \lvert S\rvert$ over which $\NP(\bL)$ is locally constant. Moreover, $U$ contains all classical points of $S$.
\end{prop}
\begin{proof}
The existence of a dense open subset $U$ is a direct consequence of Proposition~\ref{prop: useful top result} and Proposition~\ref{prop: dense open on which NP is locconst}. The fact that $U$ contains all classical points follows from Theorem~\ref{thm: one direct} and \cite[Thm.~9.2]{Shimizu-monodromy} (see also Remark~\ref{rem: shimizu thm}).
\end{proof}

\subsection{A relative \texorpdfstring{$p$}{p}-adic monodromy}

In this subsection, we will complete the proof of Theorems~\ref{thm: RpMT over dense open}, \ref{thm: CNP to RpMT}, and \ref{CNP to RpMTw}. We begin with Theorem~\ref{thm: CNP to RpMT}.

\begin{thm}[\,$\CNP\Longrightarrow \RpMT$\,]\label{thm: global case}
Let $S$ be a smooth rigid analytic space over a $p$-adic field $K$, and let $\bL$ be a de Rham local system on $S$. Let $\{\calE_1  \dashrightarrow \calE_0\}$ be the admissible modification of vector bundles on $X_S^\diamondsuit$ along $S^\diamondsuit \to \Spd K$. If the Newton polygon function $\NP(\calE_0)$ is constant on $\lvert S \rvert$, then there is a de Jong \'etale covering $S'$ of $S$ such that $\bL |_{S'}$ is log-crystalline at all classical points. Moreover, if $S$ is connected, then $S'$ may be chosen to be connected.
\end{thm}

\begin{proof}
The proof mirrors that of Theorem~\ref{thm: constant case}, and we retain the notation used there. The only difference is that, rather than using a rational open subset to promote the $\Q_p$-local system $\bV_{\underline{\Q}_p}$ to an isogeny $\Z_p$-local system, we invoke Theorem~\ref{thm: equiv of all kinds of locsys} to find a de Jong \'etale covering $S' \to S$ such that, after pullback to $S'$, the local system $\bV_{\underline{\Q}_p}$ admits a $\Z_p$-lattice. The remainder of the proof remains identical. 

If $S$ is connected, part~(4) of Proposition~\ref{prop: geometric covering-input} implies that every connected component of $S'$ is itself a de Jong \'etale covering of $S$. 
\end{proof}

\begin{cor}\label{cor: RpMT over dense}
Let $S$ be a smooth rigid-analytic space over a $p$-adic field $K$, and let $\mathbb{L}$ be a de Rham $\mathbb{Q}_p$-local system on $S$. There exists a dense open subset $U \subset S$ and a de Jong \'etale covering $V \to U$ of smooth rigid-analytic varieties over $K$ such that $\mathbb{L}|_{V}$ is log-crystalline at all classical points of $V$. Moreover, $U$ contains all classical points of $S$.
\end{cor}
\begin{proof}
Let $\lvert U \rvert$ be the union of all open subsets of $\lvert S\rvert$ on which $\NP(\bL)$ is locally constant. Then $\lvert U \rvert$ is open, and $\lvert U \rvert$ is dense by Proposition~\ref{prop: dense open on which NP is locconst} and descent. By \cite[Thm.~9.2]{Shimizu-monodromy}, together with Remark~\ref{rem: shimizu thm}, every classical point belongs to $\lvert U \rvert$. Then we can apply Theorem~\ref{thm: global case} to all the level sets $\lvert U \rvert^{\calP}$ of $\lvert U \rvert$. Note that each level set $\lvert U \rvert^{\calP}$ is open from the definition of $\lvert U \rvert$.
\end{proof}

\begin{cor}[\,$\CNP_1 \Longrightarrow \RpMT^w$\,]\label{cor: RpMTw}
Let $S$ be a connected smooth rigid analytic space over a $p$-adic field $K$, and let $\bL$ be a de Rham $\Q_p$-local system on $S$. \emph{Assume $\CNP_1$}. Then there exists an \'etale map $\pi\colon S' \to S$ such that $\bL |_{S'}$ is log-crystalline at all classical points. Moreover, the image of $\pi$ contains every rank-$1$ point of $S$, and the induced morphism
\[
S' \longrightarrow \pi(S')
\]
is a de Jong \'etale covering.
\end{cor}
\begin{proof}
One applies Theorem~\ref{thm: global case} to the interiors of the Newton partitions of $\lvert S\rvert$ defined in Definition~\ref{defn: newton partitions}, using Proposition~\ref{prop: properties of |S|P}.
\end{proof}

\section{Proof of Theorem~\ref{thm: intro HK conjecture}}\label{sec: proof of HK conjecture}

In this section, we recall the notions of neutral basic admissible pairs, admissible pairs with potential good reduction, and variations of $p$-adic Hodge structure from \cite{HoweKlevdalIII}. We then prove \cite[Conj.~4.3.4]{HoweKlevdalIII} using Theorem~\ref{thm: HT0 implies unramified proet}.

We retain the notation of \S\ref{sec: Shtukas} that $S$ is a smooth rigid analytic variety over a $p$-adic field $K$, and $S^\diamondsuit$ is its associated diamond.

\begin{defn}\label{defn: VpHS}
A \emph{variation of $p$-adic Hodge structure over $S$} is a (finitely supported) $\Q$-graded de Rham $\Qp$-local system
\[
  \mathbb V=\oplus_{\lambda\in\Q}\mathbb V_\lambda
\]
such that, for every $\lambda\in\Q$, $\mathbb V_\lambda$ is de Rham, and whenever $\mathbb V_\lambda\neq0$, the only slope of $\NP(\mathbb V_\lambda)$ is $-\lambda/2$. These objects, with grading-preserving morphisms, form the category $\mathrm{VHS}(S)$.
\end{defn}

\begin{defn}\label{defn: basic admissible pair}
Let $T \to \Spd\Qp$ be a small $\vv$-sheaf. An \emph{isoclinic neutral admissible pair of weight $\lambda$ over $T$} is an admissible modification
\[
  \mathcal E_0\dasharrow\mathcal E_1
\]
on the relative Fargues--Fontaine curve $X_{T}$ along $T \to \Spd \Qp$ such that $\mathcal E_1$ is semistable of slope $0$ and there is an isomorphism
\[
  \mathcal E_0\simeq\mathcal E(\underline D)|_{X_T}
\]
for a semisimple isocrystal $\underline D$ over $\F_p$ of slope $\lambda/2$. Here $\calE(\underline{D})|_{X_T}$ is the vector bundle defined by $(D\otimes \calO(Y_{T'}), \varphi_D \otimes \varphi)$ for all affinoid perfectoid space $T'$ under Lemma~\ref{lem: vb as phi modules}. Note that the choice of $\underline D$ and the above isomorphism is not part of the data. A \emph{basic neutral admissible pair} is a finite direct sum
\[
  \oplus_{\lambda\in\Q}\mathcal E_{0,\lambda}
  \dasharrow
  \oplus_{\lambda\in\Q}\mathcal E_{1,\lambda}
\]
such that, for every $\lambda\in\Q$, $\mathcal E_{0,\lambda} \dasharrow \mathcal E_{1,\lambda}$ is an isoclinic neutral admissible pair of weight $\lambda$. We denote their category by $\AdmPair^{\mathrm{basic},\mathrm{neut}}(T)$.
\end{defn}

\begin{defn}\label{defn: pgr admissible pair}
Let $S$ be a smooth rigid analytic variety over a $p$-adic field $K$, and let $S^\diamondsuit/\Spd K$ be the associated diamond. A $\Q$-graded admissible modification $\oplus_{\lambda \in \Q}\mathcal E_{0,\lambda} \dasharrow\oplus_{\lambda \in \Q}\mathcal E_{1,\lambda}$ of vector bundle over $X_{S^\diamondsuit}$ along $S \to \Spd\Qp$ is called an admissible pairs with potential good reduction if there is a covering $\{\widetilde{S}_i\}_{i\in I}$ by potentially unramified objects in $S_{\proet}$ as in Definition~\ref{defn:unramified-proet} such that the pullback of $\oplus_{\lambda \in \Q}\mathcal E_{0,\lambda} \dasharrow\oplus_{\lambda \in \Q}\mathcal E_{1,\lambda}$ to each $\widetilde{S}_i$ is inside $\AdmPair^{\text{basic,neut}}(\widetilde{S}_i/\Spd\Qp)$. The category of potentially unramified neutural admissible pairs is denoted by $\AdmPair^{\text{basic,pgr}}(S)$.
\end{defn}

Using Theorem~\ref{thm: Fargues thm for admissible modification}, one can check that Definition~\ref{defn: VpHS} (resp. Definition~\ref{defn: pgr admissible pair}) agrees with \cite[Defn.~4.3.1]{HoweKlevdalIII} (resp. \cite[Defn.~4.2.1]{HoweKlevdalIII}). 

\begin{thm}\label{thm: HK thm input}
Let $S$ be a smooth rigid analytic variety over a $p$-adic field. Their is a fully faithful functor
\[
  \omega_{\et}^{\mathcal L}\colon
  \AdmPair^{\mathrm{basic},\mathrm{pgr}}(S)
  \longrightarrow
  \mathrm{VHS}(S).
\]
More precisely, it sends
\[
  \oplus_\lambda\mathcal E_{0,\lambda}
  \dasharrow
  \oplus_\lambda\mathcal E_{1,\lambda}
\]
to $\oplus_\lambda\mathcal E_{1,\lambda}$, with the corresponded $\Q$-grading, and the underlying graded local system under Corollary~\ref{cor: shtukas with no leg and ss vb of slope 0} is de Rham.
\end{thm}

\begin{proof}
This is \cite[Thm.~3.1.2 and Prop.~4.3.3]{HoweKlevdalIII}, together with \cite[Lem.~4.3.2]{HoweKlevdalIII}.
\end{proof}

Only finitely many graded pieces of an object of $\mathrm{VHS}(S)$ are nonzero. Once essential surjectivity for Theorem~\ref{thm: intro HK conjecture} is known for the isoclinic case, choose potentially unramified covers for each nonzero isoclinic direct summand and take their fiber product over $S$. This is again potentially unramified by \cite[Lem.~4.1.4]{HoweKlevdalIII}, and the direct sum is neutral over it. It therefore suffices to prove essential surjectivity in the isoclinic case.

The following result the key input in this section. Fix a rational number $\mu$, and put $D=D_\mu$ as in Definition~\ref{defn: simple-isocrystal-Dlambda}.

\begin{thm}\label{thm: HT0 implies unramified proet}
Let $\mathbb W$ be a $\bDhat_\mu^{\op}$-local system on $S$ of constant $D$-rank $m$. Let $\mathbb U$ be its underlying $\Qp$-local system, and assume that $\mathbb U$ is Hodge--Tate with single weight $0$. Put $G=\operatorname{GL}_m(D)$ and $
P_D(\mathbb W)=\underline{\Isom}_D(\underline D^m,\mathbb W)$
as in Theorem~\ref{thm: generalized de Jong lattice sheaf} and Proposition~\ref{prop: frame tower}. Then the pro-\'etale presentation associate to it
\[
P_D(\mathbb W)^{\proet}\coloneqq\bigl(P_D(\mathbb W)/\underline H\bigr)_{H\subseteq G\text{ compact open}}
\]
defines it a potentially unramified object in $S_{\proet}$. More precisely, there are a compact open subgroup $K\subseteq G$ and an \'etale covering $\{U_i\to Y_{K}\}_{i\in I}$ of $Y_{K}=P_D(\mathbb W)/\underline K$ such that, for every open normal subgroup $N\triangleleft K$ and every $i$, the finite \'etale morphism
\[
\bigl(P_D(\mathbb W)/\underline N\bigr)\times_{Y_{K}}U_i\longrightarrow U_i
\]
is strongly \'etale. Note that here $Y_{K}$ is represented by an \'etale covering of $S$, cf. Theorem~\ref{thm: generalized de Jong lattice sheaf} and Proposition~\ref{prop: frame tower} for details.
\end{thm}

\begin{proof}
Consider the standard faithful representation $GL_m(D) \to GL_{mh^2}(\Q_p)$ where $h^2$ is the $\Qp$-rank of $D$. Let $K=1+p^3M_{mh^2}(\Zp)\cap GL_m(D)$ which is a compact open inside $GL_m(D)$. We claim the statement of the theorem holds for such $K$. More precisely, we may work \'etale locally on $Y_K$ so that it admits a semistable formal model $\frakX=\Spf(A)$ over $\OL$ for some finite extension $L$ of $K$ by Theorem~\ref{thm: localalteration}.

Let $\mathfrak q\subset A$ be a height-one prime containing $\varpi_L$. $A_{\mathfrak q}$ is a DVR, let $\widehat A_{\mathfrak q}$ be its completion and write $\widehat F_{\mathfrak q}=\operatorname{Frac}(\widehat A_{\mathfrak q})$. $\widehat A_{\mathfrak q}$ is a CDVR. Its residue field is finitely generated over the residue field of $L$, hence has finite $p$-basis 

By Remark~\ref{rem: HT wt 0 local system}, $\mathbb U$ being Hodge--Tate with the single weight $0$ is equivalent to that there is an \'etale vector bundle $\mathcal F$ on $S$ and an isomorphism
\begin{equation}\label{eq: HT0 condition}
	  \mathbb U\otimes_{\Qp}\widehat{\mathcal O}_{S_{\proet}}
  \simeq
  \mathcal F\otimes_{\mathcal O_S}\widehat{\mathcal O}_{S_{\proet}}.
\end{equation}
Choose a geometric point $\bar\eta_{\mathfrak q}$ over $\Spa(\widehat F_{\mathfrak q},\widehat A_{\mathfrak q})$, let $\mathcal F_{\eta_{\mathfrak q}}$ be the pullback of $\calF$ to $\widehat F_{\mathfrak q}$, and let $C_{\mathfrak q}$ be the completion of a fixed algebraic closure of $\widehat F_{\mathfrak q}$. Pulling back \eqref{eq: HT0 condition} gives natural isomorphism
\[
  \mathbb U_{\bar\eta_{\mathfrak q}}
  \otimes_{\Qp}C_{\mathfrak q}
  \simeq
  \mathcal F_{\eta_{\mathfrak q}}
  \otimes_{\widehat F_{\mathfrak q}}C_{\mathfrak q}.
\]
Identify $\mathbb U_{\bar\eta_{\mathfrak q}}$ with a $\Q_p$-representation of $G_{\widehat F_{\mathfrak q}}$. Then the above isomorphism implies this $\Q_p$-representation of $G_{\widehat F_{\mathfrak q}}$ is Hodge--Tate with weight $0$, i.e., $C_{\mathfrak q}$-admissible. Hence \cite[Cor.~3.2]{OhkuboLie} implies so that $I_{\widehat F_{\mathfrak q}}$ has finite image. Since we work over $Y_K$, the image also lies in the torsion free group $1+p^3M_{mh^2}(\Zp)$, this implies that the $I_{\widehat F_{\mathfrak q}}$ action on the geometric fiber over $\mathfrak q$ is trivial.

Let $N\subseteq K$ be an open normal subgroup. By Theorem~\ref{thm: generalized de Jong lattice sheaf}, the quotient
\[
P_D(\mathbb W)/\underline N \to Y_K
\]
is represented by a finite \'etale morphism. The above discussion implies the inertia group acts trivially on each fiber above prime ideal containing $\varpi_L$ and of height-one of the above finite morphism, since the action factor through the trivial map into $K$. In particular, condition (2) in Theorem~\ref{thm:strong-etale-purity} holds \'etale locally on $Y_K$. Therefore, it implies the strongly \'etaleness stated in the theorem. 
\end{proof}

\begin{proof}[Proof of Theorem~\ref{thm: intro HK conjecture}]
By Theorem~\ref{thm: HK thm input}, it remains to prove essential surjectivity. We may work on one connected component and assume that $\mathbb V=\mathbb V_\lambda$ is concentrated in a single weight $\lambda$. Put $\mu=-\lambda/2=d/h$ in lowest terms, with $h>0$. After a finite unramified extension of $K$, we may assume that $K$ contains $\Q_{p^h}$.

Consider the modification of vector bundles
\[
\mathcal E_0(\mathbb V)\dasharrow\mathcal E_1(\mathbb V)
\]
defined by the de Rham local system $\bV$, which is a $p$-adic variation of Hodge structure as in \cite[Thm.~3.1.2]{HoweKlevdalIII}. $\mathcal E_0(\mathbb V)$ is semistable of slope $\mu$ by assumption. We may apply Theorem~\ref{thm: Dlambda locsys and ss vb} and Remark~\ref{rem: O(lambda) to Dlambda explicit functor}, and it gives a $\bDhat_\mu^{\op}$-local system $\mathbb W$ and an isomorphism
\begin{equation}\label{eq: trivialize D_lambda}
  \mathcal E_0(\mathbb V)
  \simeq
  \mathbb W\otimes_{\bDhat_\mu}\mathcal O(\mu).
\end{equation}

Define $P_D(\mathbb W)\coloneqq \underline{\Isom}_D(\underline D^m,\mathbb W)$. From its definition, over $P_D(\mathbb W)$ there is a tautological isomorphism of $\mathbb W$ with the constant sheaf $\underline D^m=\bDhat_\mu^{\oplus m}$. By \eqref{eq: trivialize D_lambda} and the construction of $\mathcal O(\mu)$, we obtain
\[
\mathcal E_0(\mathbb V)|_{X_{P_D(\mathbb W)}^\diamondsuit}
\simeq
\mathcal O(\mu)^{\oplus m}|_{X_{P_D(\mathbb W)}^\diamondsuit}
\simeq
\mathcal E(\underline E)|_{X_{P_D(\mathbb W)}^\diamondsuit},
\]
where $\underline E=\breve{\Q}_p(-\mu)^{\oplus m}$, wtih $\breve{\Q}_p(-\mu)$ the fixed simple isocrystal of slope $-\mu$. In other words, the admissible pair $\mathcal E_0(\mathbb V)\dasharrow\mathcal E_1(\mathbb V)$ becomes neutral over $P_D(\mathbb W)$.

On the other hand, write $\mu=d/h$ in lowest terms, and let $\mathbb U$ be the underlying $\Qp$-local system of $\mathbb W$. Lemma~\ref{lem: morita} and the identity
\[
  \bM_0(\mathbb V_\lambda)
  \otimes_{\BdRp,\theta}
  \widehat{\mathcal O}_{S_{\proet}}
  \simeq
  D_{\dR}(\mathbb V_\lambda)
  \otimes_{\mathcal O_S}
  \widehat{\mathcal O}_{S_{\proet}}
\]
give
\[
  \mathbb U\otimes_{\Qp}\widehat{\mathcal O}_{S_{\proet}}
  \simeq
  D_{\dR}(\mathbb V_\lambda)^{\oplus h}
  \otimes_{\mathcal O_S}
  \widehat{\mathcal O}_{S_{\proet}}.
\]
Thus $\mathbb U$ is Hodge--Tate with single weight $0$. Theorem~\ref{thm: HT0 implies unramified proet} shows that $P_D(\mathbb W)\to S$ is potentially unramified.
\end{proof}

\appendix

\section{Specialization maps and tubes}\label{sec: specialization}

This appendix recalls basic facts about the specialization map associated with an admissible formal model of a rigid analytic variety. Throughout, $\mathcal O$ is a complete rank-$1$ valuation ring with nontrivial valuation, and $\varpi\in\mathcal O$ is a pseudo-uniformizer. 

For a point $x$ of a rigid analytic adic space $X$, let
	\[
  	k(x)\coloneqq\mathcal O_{X,x}/\mathfrak m_x
	\]
be the uncompleted residue field of the adic local ring, equipped with the valuation induced by $x$, and let $k(x)^+$ be its valuation ring. We write $K_x$ and $K_x^+$ for the completion of $k(x)$ and $k(x)^+$. 

\begin{lem}\label{lem: Spa(K,K^+) is a chain}
Let $K$ be a complete nonarchimedean field, let $K^\circ$ be its ring of power-bounded elements, and let $K^+\subseteq K^\circ$ be an open bounded valuation subring. The points of $\Spa(K,K^+)$ correspond to the valuation rings $R$ satisfying
\[
  K^+\subseteq R\subseteq K^\circ,
\]
and are totally ordered by specialization. More precisely:
\begin{enumerate}
\item if $x_R,x_{R'}$ correspond to $R,R'$, then $x_R$ is a specialization of $x_{R'}$ if and only if $R\subseteq R'$;
\item there is a unique closed point, corresponding to $K^+$;
\item the point corresponding to $K^\circ$ is the unique generic point;
\item for an \emph{analytic} adic space $X$ and $x\in X$, the canonical morphism
\[
  j_x\colon\Spa(K_x,K_x^+)\longrightarrow X
\]
sends the closed point to $x$ and the generic point to the unique rank-$1$ maximal generalization $x_{\max}$ of $x$. Moreover, there is a canonical identification of topological fields
\[
  K_x\simeq K_{x_{\max}}.
\]
\end{enumerate}
\end{lem}

\begin{proof}
This is \cite[(1.1.6)--(1.1.10)]{Huber-etale}, see also \cite[Prop.~2.29]{scholze-perfectoid}.
\end{proof}

\begin{defn}\label{defn: specialization in affine}
Let $(A,A^+)$ be a Tate Huber pair and let $\varpi\in A^+$ be a topologically nilpotent unit of $A$. For $x\in\Spa(A,A^+)$, define
\[
  \mathrm{sp}_{A^+}(x)
  \coloneqq
  \bigl\{\overline a\in A^+/\varpi A^+:\lvert a(x)\rvert<1\bigr\}
  \in\Spec(A^+/\varpi A^+),
\]
where $a\in A^+$ is any lift of $\overline a$. This is independent of the lift and defines a continuous map. The construction is functorial for morphisms of such pairs. For specialization maps in this generality, see \cite[\S4.1]{GleasonSpecialization}.
\end{defn}

\begin{construction}\label{const: specialization map for formal scheme of certain type}
Let $\mathfrak X$ be an admissible normal formal $\mathcal O$-scheme with ideal of definition $(\varpi)$, cf. \cite[\S7.4]{Bosch-Lectures}. On an affine formal open $\mathfrak U=\Spf(R)$ with generic fiber $\mathfrak U_\eta=\Spa(R[1/\varpi],R)$, Definition~\ref{defn: specialization in affine} gives
\[
  \mathrm{sp}_{\mathfrak U}\colon
  \lvert\mathfrak U_\eta\rvert
  \longrightarrow
  \lvert\Spec(R/\varpi R)\rvert.
\]
These maps glue to a continuous map
\[
  \mathrm{sp}_{\mathfrak X}\colon
  \lvert\mathfrak X_\eta\rvert
  \longrightarrow
  \lvert\mathfrak X_{\mathrm{red}}\rvert.
\]
Here $\mathfrak X_{\mathrm{red}}$ is the reduced special fiber, it has the same underlying topological space as the modulo $\varpi$ fiber.

For $Z\subseteq\lvert\mathfrak X_{\mathrm{red}}\rvert$, define its  \emph{tube} to be
\[
  ]Z[_{\mathfrak X}
  \coloneqq
  \mathrm{sp}_{\mathfrak X}^{-1}(Z)
  \subseteq\lvert\mathfrak X_\eta\rvert.
\]
\end{construction}

\begin{rem}
We use the full inverse image as our convention. Some references instead use its interior, for example \cite[II~\S4.2]{fujiwara-kato}.
\end{rem}

The next statement is motivated by Berthelot's result on tubes of locally closed subschemes in rigid geometry in the sense of Tate, see \cite[Prop.~1.1.1]{Berthelot-rigid}.

\begin{thm}\label{thm: tubes are open near rk 1}
Let $\mathfrak X$ be as in Construction~\ref{const: specialization map for formal scheme of certain type}, and assume that $\mathfrak X_{\mathrm{red}}$ is Noetherian. If $Z\subseteq\lvert\mathfrak X_{\mathrm{red}}\rvert$ is locally closed, then every rank-$1$ point of $]Z[_{\mathfrak X}$ has an open neighborhood contained in $]Z[_{\mathfrak X}$.
\end{thm}

\begin{proof}
The assertion is local on $\mathfrak X$, so we may assume $\mathfrak X=\Spf(R)$. Let $x\in]Z[_{\mathfrak X}$ be a rank-$1$ point. Write $Z=F\cap W$, with $F$ closed and $W$ open. $F=V(\overline f_1,\ldots,\overline f_r)$ since $\lvert \frakX_{\mathrm{red}}\rvert $ is Noetherian. Choose $\overline g\in R/\varpi R$ such that
\[
  \mathrm{sp}_{\mathfrak X}(x)\in D(\overline g)\subseteq W.
\]
Then
\[
  Z'=V(\overline f_1,\ldots,\overline f_r)\cap D(\overline g)
\]
contains $\mathrm{sp}_{\mathfrak X}(x)$ and is contained in $Z$. Choose lifts $f_1,\ldots,f_r,g\in R$. Definition~\ref{defn: specialization in affine} gives
\[
  \lvert g(x)\rvert=1,
  \qquad
  \lvert f_i(x)\rvert<1\quad(1\leq i\leq r).
\]
Because $x$ has rank $1$, for every $i$ there is an integer $N_i\geq1$ such that $\lvert f_i(x)\rvert^{N_i}\leq\lvert\varpi(x)\rvert$. Hence $x$ belongs to the rational open subset
\[
  \Omega_x
  =
  \bigl\{y:
    \lvert g(y)\rvert=1,
    \ \lvert f_i(y)\rvert^{N_i}\leq\lvert\varpi(y)\rvert
    \text{ for all }i
  \bigr\}.
\]
Since $g,f_i\in R$, the displayed inequalities imply that $\mathrm{sp}_{\mathfrak X}(y)\in D(\overline g)\cap V(\overline f_1,\ldots,\overline f_r)$ for every $y\in\Omega_x$. Thus
\[
  \Omega_x\subseteq]Z'[_{\mathfrak X}\subseteq]Z[_{\mathfrak X},
\]
as required.
\end{proof}

\begin{defn}\label{defn: spmax}
For $\mathfrak X$ as in Construction~\ref{const: specialization map for formal scheme of certain type}, define the \emph{maximal specialization map}
\[
  \mathrm{sp}_{\max,\mathfrak X}\colon
  \lvert\mathfrak X_\eta\rvert
  \longrightarrow
  \lvert\mathfrak X_{\mathrm{red}}\rvert,
  \qquad
  \mathrm{sp}_{\max,\mathfrak X}(x)
  \coloneqq
  \mathrm{sp}_{\mathfrak X}(x_{\max}).
\]
It agrees with $\mathrm{sp}_{\mathfrak X}$ on rank-$1$ points.
\end{defn}

\begin{lem}\label{lem: anticontinuous}
Assume in addition that $\mathfrak X=\Spf(R)$ is affine and that $\mathfrak X_{\mathrm{red}}$ is Noetherian. Then $\mathrm{sp}_{\max,\mathfrak X}$ is anticontinuous: inverse images of closed subsets are open, and inverse images of open subsets are closed.
\end{lem}

\begin{proof}
For $f\in R$, the rank-$1$ generalization $x_{\max}$ of $x$ satisfies
\[
  \lvert f(x_{\max})\rvert<1
  \quad\text{ if and only if }\quad
  \lvert f(x)\rvert^N\leq\lvert\varpi(x)\rvert
  \text{ for some }N\geq1.
\]
Consequently,
\[
  \mathrm{sp}_{\max,\mathfrak X}^{-1}\bigl(V(\overline f)\bigr)
  =
  \bigcup_{N\geq1}
  \bigl\{x\in\lvert\mathfrak X_\eta\rvert :
    \lvert f(x)\rvert^N\leq\lvert\varpi(x)\rvert
  \bigr\},
\]
which is open. Every closed subset of the Noetherian special fiber is a finite intersection of principal closed subsets, so its inverse image under $\mathrm{sp}_{\max,\mathfrak X}$ is open. Taking complements proves the assertion for open subsets of the special fiber.
\end{proof}

\section{Coverings, local systems, and the lattice functor}\label{app:local-systems}

This appendix recalls those (\emph{geometric}) \'etale covering categories needed in the statement of our main results. We also record a mild generalization of de Jong's lattice functor. The notion of \emph{geometric} \'etale coverings (and their variants) of a rigid-analytic space was introduced by Achinger, Lara, and Youcis in \cite{ALYArcs}, generalizing the corresponding notion for schemes introduced by Bhatt and Scholze in \cite{BhattScholzeproetale}. 

\subsection{Geometric coverings}

Following \cite[Defn.~3.1.1]{ALYVariants}, let $\tau\in\{\oc,\et\}$. A $\tau$-cover of a rigid analytic variety $Z$ is a surjection $U\to Z$ which is a disjoint union of overconvergent open immersions if $\tau=\oc$, and a disjoint union of \'etale maps if $\tau=\et$ respectively. Let $\Cov_Z^\tau$ be the full subcategory of all \'etale morphism to $Z$ consisting of maps $Y\to Z$ for which there is a $\tau$-cover $U\to Z$ such that $Y\times_Z U$ is a disjoint union of finite \'etale maps to $U$. Objects of $\Cov_Z^\oc$ are de Jong covering spaces. Let $\UCov_Z^\tau$ denote the category of arbitrary disjoint unions of objects of $\Cov_Z^\tau$.

Achinger, Lara, and Youcis also introduce the notion of a \emph{geometric covering}: an \'etale morphism that is partially proper and satisfies the unique lifting property for geometric arcs after pullback to every test curve, we refer to \cite[Defn.~5.2.2]{ALYArcs} for the precise definition. Denote the category of geometric coverings of $Z$ by $\GCov_Z$.\footnote{This category is denoted by $\Cov_Z$ in \cite{ALYArcs}.}

\begin{prop}\label{prop: geometric covering-input}
Let $Z$ be a rigid analytic variety over $K$.
\begin{enumerate}
\item Geometric coverings are stable under disjoint unions, compositions, and arbitrary pullbacks, and the image of a geometric covering is open and closed.
\item Every finite \'etale map to $Z$ belongs to $\UCov_Z^\oc$, and
\[
  \UCov_Z^\oc\subseteq\UCov_Z^\et\subseteq\GCov_Z.
\]
\item When $Z$ is connected, let $\overline z$ be a geometric point, the fiber functor at $\overline z$ induces an equivalence
\[
  \GCov_Z\xrightarrow{\sim}
  \pi_1^{\mathrm{ga}}(Z,\overline z)
  \text{-}\mathrm{Set},
\]
where $\pi_1^{\mathrm{ga}}(Z,\overline z)$ is the geometric arc fundamental group defined in \cite[Defn.~5.4.2]{ALYArcs}. Connected objects correspond to transitive sets. In particular, each connected component of a geometric covering of $Z$ is again a geometric covering of $Z$.
\item Let $f\colon Y\to X$ be a morphism in $\GCov_Z$ that induces injective maps on geometric fibers. If $X\in\Cov_Z^\tau$, then $Y\in\Cov_Z^\tau$. In particular, each connected component of a de Jong covering of $Z$ is again a de Jong covering of $Z$.
\end{enumerate}
\end{prop}

\begin{proof}
(1) follows from \cite[Props.~5.2.6, 5.2.8, and 5.2.11]{ALYArcs}. (2) follows from \cite[Rem.~3.1.2 and Prop.~3.1.3]{ALYVariants}. (3) is \cite[Thm.~5.4.1 and Cor.~5.4.3, Prop.~5.4.5]{ALYArcs}. 

For (4), since $X\in\Cov_Z^\tau$, there exists a $\tau$-covering family $\{U_\alpha\to Z\}_{\alpha\in A}$ such that each $X\times_Z U_\alpha$ is a disjoint union of finite \'etale coverings of $U_\alpha$. We may assume that each $U_\alpha$ is connected. It is enough to show $Y\times_Z U_\alpha$ is a disjoint union of finite \'etale coverings of $U_\alpha$.

We fix $\alpha$ and put $U=U_\alpha$. Write $X\times_Z U=\coprod_{i\in I}X_i$, with each $X_i\to U$ being finite \'etale and connected. By (1), $Y\times_Z U$ is a geometric covering of $U$.

Fix a geometric point $\overline u$ of $U$, and consider the fiber functor $F$ in (3),  $f$ then induces an injective map
\[
  F(Y\times_Z U)\longrightarrow\coprod_{i\in I}F(X_i)
\]
of $\pi_1^{\mathrm{ga}}(U,\overline u)$-sets. Each $F(X_i)$ is transitive by (3). Hence the image of this map is a union of $F(X_i)$ for $i \in J \subset I$. Applying (3) again gives
\[
  Y\times_Z U\simeq\coprod_{i\in J}X_i.
\]
So we have $Y\in\Cov_Z^\tau$.
\end{proof}

\subsection{The lattice sheaf of de Jong and its generalization}

Let $\bbL$ be a $\Qp$-local system on a rigid analytic variety $Z$.

\begin{defn}\label{defn: lattice functor}
For an \'etale map $T\to Z$, let $\Lat_{\bbL/Z}(T)$ be the set of isomorphism classes of pairs $(\bbT,\iota)$, where $\bbT$ is a finite free $\Zp$-local system on $T$ and
\[
  \iota\colon
  \bbT\otimes_{\Zp}\Qp
  \xrightarrow{\sim}
  \bbL|_T.
\]
An isomorphism $(\bbT,\iota)\to(\bbT',\iota')$ is an isomorphism $a\colon\bbT\to\bbT'$ satisfying $\iota'\circ(a\otimes\Qp)=\iota$.
\end{defn}

This is the lattice functor introduced by de Jong in the proof of \cite[Thm.~4.2]{dJfundamentalgroup}. de Jong proves that it is represented by a de Jong covering. We need the following generalization.

\begin{thm}\label{thm: generalized de Jong lattice sheaf}
Let $Z$ be a rigid analytic variety, let $G$ be a locally profinite group, and let $\underline G$ be the pro-\'etale sheaf
\[
  U\longmapsto\operatorname{Cont}(\lvert U\rvert,G).
\]
Let $P$ be a right $\underline G$-torsor on $Z_{\proet}$. For every compact open subgroup $K\subseteq G$, let $P/\underline K$ denote the quotient in $\operatorname{Sh}(Z_{\proet})$. Then $P/\underline K$ is represented, uniquely up to unique isomorphism, by an \'etale surjection
\[
  Y_K\longrightarrow Z
\]
belonging to $\Cov_Z^\et$. In particular, $Y_K\to Z$ is a geometric covering. If $K'\subseteq K$ is compact open, then $Y_{K'}\to Y_K$ is finite \'etale.
\end{thm}

\begin{proof}
The space $G/K$ is discrete, and $P/\underline K$ is a locally constant sheaf with discrete fiber $G/K$. It is therefore a pro-\'etale local system of sets in the sense of \cite[p.~1435]{ALYVariants}, so representability by an object of $\Cov_Z^\et$ follows from \cite[Thm.~4.4.1]{ALYVariants}. Since $G/K$ is nonempty, the representing map is surjective. Over $Y_K$, the sheaf $Y_{K'}$ is the finite locally constant sheaf associated with the $K$-set $K/K'$. It is therefore represented by a finite \'etale cover. 
\end{proof}

\begin{rem}\label{rem: relation with lattice cover}
We keep the notations as in Theorem~\ref{thm: generalized de Jong lattice sheaf}.
\begin{enumerate}
	\item There is a canonical isomorphism of sheaves
\[
  P/\underline K \simeq \operatorname{Red}_K(P),
\]
where $\operatorname{Red}_K(P)$ denotes the sheaf whose sections over a test object $U$ are the isomorphism classes of pairs $(Q,f_Q)$ consisting of a right $\underline K$-torsor $Q$ over $U$ and an isomorphism
\[
  f_Q\colon
  Q\times^{\underline K}\underline G
  \xrightarrow{\sim} P|_U
\]
of right $\underline G$-torsors over $U$.

In particular, suppose that
\[
  G=\operatorname{GL}_r(\Qp),\qquad
  K=\operatorname{GL}_r(\Zp),\qquad
  P=\underline{\Isom}_{\Qp}(\Qp^r,\bbL).
\]
Then $P/\underline K$ is precisely the lattice sheaf introduced in Definition~\ref{defn: lattice functor}. Theorem~\ref{thm: generalized de Jong lattice sheaf} shows that this sheaf is represented by an object of $\Cov_Z^\et$. De Jong proved the stronger statement that it is represented by an object of $\Cov_Z^\oc$, cf. the proof of \cite[Thm.~4.2]{dJfundamentalgroup}.
\item We may generalize (1) to the following situation. We call $P$ is a \emph{$\Qp$-linear} (right) $\underline G$-torsor if there exists a closed embedding $G\hookrightarrow\operatorname{GL}(V)$, where $V$ is a finite-dimensional $\Qp$-vector space, such that $\bL\coloneqq P\times^G V$ is a  pro-\'etale $\Qp$-local system. We claim, under this setup, $Y_K\to Z$ is a de Jong covering for every compact open subgroup $K\subseteq G$.

To prove this, identify $G$ with its image in $\operatorname{GL}(V)$. One can always find a $K$-stable full lattice $\Lambda\subseteq V$. Consider $H=G\cap\operatorname{GL}(\Lambda)$, then $H$ is a compact open subgroup of $G$ containing $K$. The above $\Lambda$ induces a morphism
\[
  Y_H=P/\underline H\longrightarrow\Lat_{\bbL/Z},
\]
and by (1), we have $\Lat_{\bbL/Z}$ is represented by a de Jong covering of $Z$. Moreover, on each geometric fiber, this map is injective by construction.

By Theorem~\ref{thm: generalized de Jong lattice sheaf}, $Y_H\in\Cov_Z^\et$, so both $Y_H$ and $\Lat_{\bbL/Z}$ are geometric coverings. Proposition~\ref{prop: geometric covering-input}(4) therefore implies $Y_H\in\Cov_Z^\oc$. It is then easy to deduce $Y_K\in\Cov_Z^\oc$ as well.
\end{enumerate}
\end{rem}

We also record the following fact about framing sheaf, which should be well-known.

\begin{prop}\label{prop: frame tower}
Keep the hypotheses and notations as in the statement of Theorem~\ref{thm: generalized de Jong lattice sheaf}, fix a compact open subgroup $K\subseteq G$, and put $Y_H=P/\underline H$ for every compact open subgroup $H\subseteq G$. Regard
\[
  P\longrightarrow Y_K
\]
as a right $\underline K$-torsor. If $N$ is an open normal subgroup of $K$, then
\[
  Y_N\longrightarrow Y_K
\]
is a finite \'etale right $K/N$-torsor, and the canonical map
\begin{equation}
  P\xrightarrow{\sim}
  \varprojlim_{N\triangleleft K}Y_N
  \label{eq:general-compact-frame-limit}
\end{equation}
is an isomorphism of pro-\'etale sheaves over $Y_K$, where $N$ ranges over all open normal subgroups of $K$. In particular, \eqref{eq:general-compact-frame-limit} defines a pro-\'etale presentation $
P^{\proet}=\bigl(Y_N)_{N\triangleleft K}\in (Y_K)_{\proet}$ of $P$, cf. \cite[Defn.~3.9]{scholze-p-adic-hodge}.
\end{prop}

\begin{proof}
We just prove that \eqref{eq:general-compact-frame-limit} defines a pro-\'etale presentation of $P$. This assertion is local for the pro-\'etale topology, so we may assume that $P$ is trivial. It then reduces to the constant case, and to the statement about the quotient map $G\to G/K$. 
\end{proof}

\section{A purity result for strongly \'etale morphisms}\label{app:valuation}

This appendix recalls the notions of strongly \'etale morphisms and unramified pro-\'etale objects and proves a purity result for strongly \'etaleness. Throughout, $L$ is a complete discrete value extension of $\Q_p$ with perfect residue field. Write $\OL$ to be the ring of integers of $L$ and $\varpi_L$ its uniformizer. By an admissible formal $\OL$-scheme we mean a $\varpi_L$-adic formal scheme locally topologically of finite type and flat over $\OL$. For a valuation field $F$, let $F^+$ be its valuation ring.

\begin{defn}\label{defn:strongly-unramified}
\begin{enumerate}
\item Let $E/F$ be a finite separable extension of valuation fields. Fix an algebraic closure containing $E$. Let $F^{\mathrm{sh}}$ and $E^{\mathrm{sh}}$ be the fraction fields of the corresponding strict henselizations of the valuation rings. The extension $E/F$ is called \emph{unramified} if
\[
  F^{\mathrm{sh}}=E^{\mathrm{sh}}.
\]
\item Let $f\colon Y\to X$ be a morphism of strongly noetherian adic spaces. It is \emph{strongly \'etale at $y\in\lvert Y\rvert$} if it is \'etale at $y$ and the extension of residue fields
\[
  k(y)/k(f(y))
\]
is unramified, where $k(y)$ is the (uncompleted) residue field at $y$ as in Appendix~\ref{sec: specialization}. The map $f$ is \emph{strongly \'etale} if it is strongly \'etale at every point.
\end{enumerate}
\end{defn}

This is defined in \cite{HubnerAdictame} and is the convention used in \cite[\S4.1]{HoweKlevdalIII}. For complete discrete valuation fields it agrees with the usual notion, cf. \cite[Lem.~3.4]{HubnerAdictame}. More generally, if $E^+$ and $F^+$ are the valuation subrings, then $E/F$ is unramified if and only if $E^+$ is a localization of an \'etale $F^+$-algebra, cf. \cite[Lem.~3.2]{HubnerAdictame}. Strong \'etaleness is stable under base change, cf. \cite[page.~885]{HubnerAdictame}. If $f\colon Y\to X$ and $g\colon Z\to Y$ are \'etale and $z\in\lvert Z\rvert$, then $f\circ g$ is strongly \'etale at $z$ if and only if $g$ is strongly \'etale at $z$ and $f$ is strongly \'etale at $g(z)$, cf. \cite[Lem.~4.1.2]{HoweKlevdalIII}.

\begin{defn}[{\cite[Defn.~4.1.3]{HoweKlevdalIII}}]\label{defn:unramified-proet}
Let $X$ be a rigid analytic variety over $L$, and let $\widetilde X\in X_{\proet}$.
\begin{enumerate}
\item The object $\widetilde X$ is \emph{unramified} if it admits a pro-\'etale presentation
\[
  \widetilde X= \varprojlim_{i\in I} X_i \to X
\]
such that every morphism $X_i\to X$ is strongly \'etale.
\item The object $\widetilde X$ is \emph{potentially unramified} if there is an \'etale cover $\{U_j\to X\}$ such that $\widetilde X\times_X U_j\to U_j$ is unramified for every $j$.
\end{enumerate}
\end{defn}

We found the following facts from \cite[Appendix~A]{ALYSpecialization} about $\eta$-normal formal models very convenient for stating our results in this appendix.

\begin{defn} An admissible formal $\OL$-scheme $\frakX$ is \emph{$\eta$-normal} if, for every affine open $\Spf(A)\subseteq\frakX$, the ring $A$ is $\varpi_L$-torsion-free and integrally closed in $A[1/\varpi_L]$. Moreover, one has
\[
  \bigl(\Spf(A)\bigr)_\eta
  =
  \Spa(A[1/\varpi_L],A).
\]
\end{defn}

See \cite[Defn.~A.1]{ALYSpecialization}. A regular admissible formal scheme is $\eta$-normal by \cite[Lem.~A.2(b)]{ALYSpecialization}. Moreover, $\eta$-normality is preserved under \'etale morphisms of formal schemes by \cite[Cor.~A.16]{ALYSpecialization}.

\begin{lem}\label{lem:eta-normalization}
Let $\frakX$ be an admissible formal $\OL$-scheme, put $X=\frakX_\eta$, and let $f\colon Y\to X$ be a finite morphism with $Y$ reduced.
\begin{enumerate}
\item There is a finite morphism
\[
  \mathfrak f\colon\frakY\longrightarrow\frakX
\]
such that $(\mathfrak f)_\eta\simeq f$ and $\frakY$ is $\eta$-normal. With a fixed identification $(\mathfrak f)_\eta\simeq f$, the pair $(\mathfrak f, \frakY)$ is unique up to a unique isomorphism. If $\frakX=\Spf(A)$ and $Y=\Spa(T,T^\circ)$, then $\frakY\simeq\Spf(T^\circ)$. This construction is functorial and compatible with restriction to open formal subschemes.
\item In the notation of (1), $T^\circ$ is the integral closure of the image of $A$ in $T$.
\end{enumerate}
\end{lem}

\begin{proof}
(1) is \cite[Lem.~4.1]{ALYSpecialization}. Functoriality follows from uniqueness, and compatibility with restriction is used in the proof of \cite[Cor.~4.2]{ALYSpecialization}. For (2), $T^\circ$ is finite over $A$ by (1), hence integral over $A$. Conversely, if $t\in T$ is integral over $A$, then $A[t]$ is a finite $A$-submodule of $T$. Since the image of $A$ is bounded, every finite $A$-submodule is bounded. As $A[t]$ contains every power of $t$, the element $t$ is power-bounded and therefore belongs to $T^\circ$.
\end{proof}

Let $\frakU=\Spf(A)$ be an admissible affine formal scheme over $\calO_L$ with $A$ noetherian and regular. For a height-one prime $\mathfrak q\subset A$ containing $\varpi_L$, put $F_{\mathfrak q}=\operatorname{Frac}(A_{\mathfrak q})$.  Since $A$ is regular and $\mathfrak q$ has height one, $A_{\mathfrak q}$ is a DVR. Let $\widehat A_{\mathfrak q}$ be the completion of $A_{\mathfrak q}$ and $\widehat F_{\mathfrak q}=\operatorname{Frac}(\widehat A_{\mathfrak q})$, so that $\widehat A_{\mathfrak q}$ is a CDVR. There is a morphism
\[
  \Spa(\widehat F_{\mathfrak q},\widehat A_{\mathfrak q})
  \longrightarrow
  \frakU_\eta.
\]
Let $x_{\mathfrak q}$ be the rank-$1$ point defined by the above map in $\frakU_\eta$, which we will view $x_{\mathfrak q}$ as a rank-$1$ valuation. One has a natural isomorphism
\[
\operatorname{Frac}\left(A[1/\varpi_L]/\operatorname{supp}(x_{\mathfrak q})\right) \simeq F_{\mathfrak q},
\]
and there is a canonical embedding $
  F_{\mathfrak q}
  \hookrightarrow
  k(x_{\mathfrak q})$, which induces canonical identifications
\[
  K_{x_{\mathfrak q}}\simeq\widehat F_{\mathfrak q},
  \qquad
  K_{x_{\mathfrak q}}^{+}\simeq\widehat A_{\mathfrak q},
\]
Here recall that $k(x_{\mathfrak q})$ and $K_{x_{\mathfrak q}}$ are the uncompleted and completed residue valuation field as in Appendix~\ref{sec: specialization}. The above facts can be found in \cite[Defn.~III.5.1.1 and Prop.~III.6.3.1]{MorelAdicSpaces}. We note that in general, $F_{\mathfrak q}$ need not equal the uncompleted residue field $k(x_{\mathfrak q})$. 

We can now state the main result of this appendix.

\begin{thm}[A purity result for strong \'etaleness]\label{thm:strong-etale-purity}
Let $X$ be an affinoid rigid analytic variety over $L$ with an admissible formal $\OL$-model $\frakX =\Spf(A)$ with $A$ noetherian and regular. Let
\[
  f\colon Y\longrightarrow X
\]
be finite \'etale, and let
\[
  \mathfrak f\colon\frakY\longrightarrow\frakX
\]
be its finite $\eta$-normal formal model from Lemma~\ref{lem:eta-normalization}. For $A$ as above, write
\[
  Y=\Spa(T,T^\circ).
\]
For every height-one prime $\mathfrak q\subset A$ containing $\varpi_L$, put
\[
  \Omega_{\mathfrak q}
  \coloneqq
  \Hom_{\widehat F_{\mathfrak q}\text{-alg}}
  \left(
    T\otimes_{A[1/\varpi_L]}\widehat F_{\mathfrak q},
    \widehat F_{\mathfrak q}^{\mathrm{sep}}
  \right),
\]
where $\widehat F_{\mathfrak q}^{\mathrm{sep}}$ is a fixed separable closure. Let $I_{\widehat F_{\mathfrak q}}$ be the inertia subgroup of the absolute Galois group $G_{\widehat F_{\mathfrak q}}$. The following are equivalent:
\begin{enumerate}
\item $f\colon Y\to X$ is strongly \'etale;
\item for every height-one prime $\mathfrak q\subset A$ containing $\varpi_L$, the inertia group $I_{\widehat F_{\mathfrak q}}$ acts trivially on $\Omega_{\mathfrak q}$;
\item $\mathfrak f\colon\frakY\to\frakX$ is finite \'etale.
\end{enumerate}
\end{thm}

\begin{proof}
Write $B=T^\circ$. By Lemma~\ref{lem:eta-normalization}, $B[1/\varpi_L]=T$, $B$ is $\eta$-normal, $B$ is the integral closure of $A$ in $T$, and $B$ is finite over $A$. Since $T$ is finite \'etale over the regular ring $A[1/\varpi_L]$, it is normal. The rings $B$ and $T$ have the same total quotient rings, so the normality of $T$ and the $\eta$-normality of $B$ imply that $B$ is also normal.

We will use the following fact: For any $\mathfrak q\subset A$ height-one prime containing $\varpi_L$, using \cite[Tags~00MC and 02VN]{stacks-project}, one has $B\otimes_A A_{\mathfrak q}\text{ is finite \'etale over }A_{\mathfrak q}$ if and only if $B\otimes_A \widehat A_{\mathfrak q}$ is finite \'etale over $\widehat A_{\mathfrak q}$. 

If $\mathfrak n_1,\ldots,\mathfrak n_s$ are the primes of $B$ over $\mathfrak q$, then
\[
B\otimes_A \widehat A_{\mathfrak q}\simeq\prod_{i=1}^s\widehat{B_{\mathfrak n_i}}
\]
by \cite[Tag~07N9]{stacks-project}, here the local rings $B_{\mathfrak n_i}$ are DVRs, so are their completions. Therefore, $B\otimes_A \widehat A_{\mathfrak q}$ is the integral closure of $\widehat A_{\mathfrak q}$ in $T\otimes_{A[1/\varpi_L]}\widehat F_{\mathfrak q}$. By \cite[Tags~04GK and 0BSW]{stacks-project}, triviality of $I_{\widehat F_{\mathfrak q}}$ on $\Omega_{\mathfrak q}$ is equivalent to this integral closure being finite \'etale over $\widehat A_{\mathfrak q}$, which is also equivalent to $B\otimes_A A_{\mathfrak q}$ is finite \'etale over $A_{\mathfrak q}$.

Now we prove the theorem by showing $(3)\Rightarrow (1) \Rightarrow (2)\Rightarrow (3)$.

\smallskip
\noindent
\emph{$(3) \Rightarrow (1)$.}
By \cite[Lem.~A.13]{ALYSpecialization}, condition~(3) implies $A\to B$ is finite \'etale. We follow the idea of proof of \cite[Cor.~4.3]{HubnerAdictame}.

Fix $y\in Y$ and put $x=f(y)$. Let
\[
\mathfrak q_y
=
\{b\in B:\lvert b(y)\rvert<1\}.
\]

Since $A\to B$ is \'etale, by
\cite[Tag~00UE]{stacks-project} there exists an element $s\in B\setminus\mathfrak q_y$ such that $A\to B_s$ is standard \'etale, i.e., there are $P,G\in A[Z]$ and an $A$-algebra isomorphism
\[
B_s
\simeq
\bigl(A[Z]/(P)\bigr)_G,
\]
where $P$ is monic and $P'$ is invertible in $\bigl(A[Z]/(P)\bigr)_G$.

Since $s\notin\mathfrak q_y$, and elements in $B$  are power bounded, one has $\lvert s(y)\rvert=1$. Consequently, the homomorphism $B\to k(y)^+$ extends to $B_s\to k(y)^+$. The image of $G$ is invertible in $B_s$, so its image in $k(y)^+$ is a unit. Therefore $\lvert G(y)\rvert=1$.

Let $P_x,G_x\in k(x)^+[Z]$ be the images of $P,G$ defined by the natural map $A\to k(x)^+$. The argument in the proof of \cite[Cor.~4.3]{HubnerAdictame} shows that $k(y)^+$ is a localization of $ \bigl(k(x)^+[Z]/(P_x)\bigr)_{G_x}$. By \cite[Lem.~3.2]{HubnerAdictame}, this implies $k(y)/k(x)$ is unramified. Since this holds for every $y\in Y$, $f$ is strongly \'etale.

\smallskip
\noindent
\emph{$(1)\Rightarrow (2)$.}
Let $\mathfrak q$ be as above. Strong \'etaleness is stable under base change by \cite[\S3]{HubnerAdictame}. Hence the pullback to $\Spa(\widehat F_{\mathfrak q},\widehat A_{\mathfrak q})$ is strongly \'etale. But this agrees with $\widehat A_{\mathfrak q} \to B\otimes_A \widehat A_{\mathfrak q}$ being \'etale, cf. \cite[Lem.~3.4]{HubnerAdictame}. Using \cite[Lem.~A.13]{ALYSpecialization} again, we have every factor of $T\otimes_{A[1/\varpi_L]}\widehat F_{\mathfrak q}$ is unramified over $\widehat F_{\mathfrak q}$. Equivalently, $I_{\widehat F_{\mathfrak q}}$ acts trivially on $\Omega_{\mathfrak q}$ by \cite[Tag~0BSW]{stacks-project}.

\smallskip
\noindent
\emph{$(2) \Rightarrow (3)$.}
It is enough to show $A\to B$ is finite \'etale. Since is $A$ is noetherian and $A \to B$ is finite, using \cite[Tags~039L]{stacks-project}, it is enough to show that $A_{\mathfrak p} \to B_{\mathfrak n}$ is \'etale for every $\mathfrak n\in\Spec(B)$, with $\mathfrak p \in \Spec A$ the image of $\mathfrak n$.

As we have discussed, $(2)$ implies that $A\to B$ is \'etale on fibers over every height-one primes of $A$ containing $\varpi_L$. Since $\varpi_L$ is a non-zero-divisor, $A \to B$ is also \'etale on the fibers over all height-zero primes and height-one primes that do not contain $\varpi_L$, since $A[1/\varpi_L]\to B[1/\varpi_L]$ is finite \'etale.  

For the rest of prime ideals $\mathfrak n\in\Spec(B)$, after restricting to the connected components containing these points, we reduce to the situation that $A$ and $B$ are domains, and $A\to B$ is finite and injective. Going down and incomparability give $\dim(B_{\mathfrak n})=\dim(A_{\mathfrak p})$ and show that every height-one prime of $B$ contracts to a height-one prime of $A$, cf. \cite[Tags~00H8 and 00GT]{stacks-project}. If $\dim(B_{\mathfrak n})=\dim(A_{\mathfrak p})\geq 2$, since we have already shown $B_{\mathfrak n}$ is normal, $A_{\mathfrak p}$ is regular, the morphism is quasi-finite, and it is \'etale at every $\mathfrak n'\subseteq\mathfrak n$ with $\dim(B_{\mathfrak n'})=1$. So \cite[Tag~0BMB]{stacks-project} applies and $A\to B$ is \'etale at all $\mathfrak n$.  
\end{proof}

\medskip\noindent
\textbf{Acknowledgments.}
The author thanks Laurent Fargues and Yihang Zhu for many inspiring conversations. The author is especially grateful to Laurent Fargues for his encouragement and for suggesting the use of the language of diamonds to frame the results, and to Yihang Zhu for suggesting the example of the Legendre family discussed in the introduction. He also thanks Tongmu He for helpful discussions concerning \cite{OhkuboLie}, and Shizhang Li for discussions related to Remark~\ref{rem: nonCNP} and \cite{dJfundamentalgroup}, which led to the present form of Appendix~\ref{app:local-systems}. This paper is partially based on the author's talk at the 2025 ICCM, and he thanks the organizers for their invitation and hospitality. The author was partially supported by the National Key R\&D Program of China No.~2023YFA1009703, the National Natural Science Foundation of China No.~12275066, and the Beijing Natural Science Foundation under Grant No.~1254044.
\bibliographystyle{amsalpha}
\bibliography{library}
	
\end{document}